\documentclass[11pt,reqno]{amsart}
\usepackage{url}
\usepackage{bbm}
\usepackage{dsfont}
\usepackage[T1]{fontenc}
\usepackage{enumitem}
\usepackage{hyperref}
\hypersetup{
colorlinks   = true,
citecolor    = blue,
linkcolor    = blue
}

\usepackage[numbers,sort&compress]{natbib}

\usepackage{amsmath,amsfonts,amsthm,amssymb,color,tikz, comment, xcolor}
\usepackage{bm}
\usepackage{mathrsfs}   
\usepackage[normalem]{ulem}
\usepackage{xfrac}
\usepackage{pdfsync}
\usepackage[font={scriptsize}]{caption}
\usepackage{environ}

\usepackage[left=1in, right=1in, top=1.1in,bottom=1.1in]{geometry}
\newcommand{\der}{\delta}

\usepackage{accents}

\newcommand{\E}{\mathbb E}
\newcommand{\R}{\mathbb R}

\newcommand{\PP}{\mathbb P}

\newcommand{\cf}{\mathcal F}
\newcommand{\cg}{\mathcal G}

\newcommand{\cl}{\mathcal L}
\newcommand{\cm}{\mathcal M}

\newcommand{\cp}{\mathcal P}

\newcommand{\cs}{\mathcal S}

\newcommand{\al}{\alpha}

\newcommand{\ep}{\varepsilon}

\newcommand{\la}{\lambda}

\newcommand{\si}{\sigma}

\newcommand{\lp}{\left(}
\newcommand{\rp}{\right)}
\newcommand{\lc}{\left[}
\newcommand{\rc}{\right]}

\newcommand{\lln}{\left|}
\newcommand{\rrn}{\right|}
\newcommand{\lla}{\left\langle}
\newcommand{\rra}{\right\rangle}

\newtheorem{theorem}{Theorem}[section]

\newtheorem{corollary}[theorem]{Corollary}

\newtheorem{definition}[theorem]{Definition}

\newtheorem{hypothesis}[theorem]{Hypothesis}
\newtheorem{lemma}[theorem]{Lemma}
\newtheorem{notation}[theorem]{Notation}

\newtheorem{proposition}[theorem]{Proposition}

\theoremstyle{remark}
\newtheorem{remark}[theorem]{Remark}

\theoremstyle{remark}

\numberwithin{equation}{section}

\newcommand{\bean}{\begin{eqnarray*}}
\newcommand{\eean}{\end{eqnarray*}}
\newcommand{\ben}{\begin{enumerate}}
\newcommand{\een}{\end{enumerate}}
\newcommand{\beq}{\begin{equation}}
\newcommand{\eeq}{\end{equation}}

\title[Regime Switching Mean Field Game and Convergence]{$N$-Player Stochastic Differential Games with Regime Switching and Mean Field Convergence}

\author[MINGRUI WANG]{MINGRUI WANG}
\address{(M. Wang)Harold and Inge Marcus Department of Industrial and Manufacturing
Engineering, The Pennsylvania State University, University Park, PA 16802 United States\\
}
\email{mvw5822@psu.edu}

\author[PRAKASH CHAKRABORTY]{PRAKASH CHAKRABORTY}
\address{(P. Chakraborty) Harold and Inge Marcus Department of Industrial and Manufacturing
Engineering, The Pennsylvania State University, University Park, PA 16802 United States\\
}
\email{prakashc@psu.edu}

\date{\today}
\thanks{P. Chakraborty is partially supported by the National Science Foundation under grant DMS-2153915}

\begin{document}
\maketitle
\begin{sloppypar}
\begin{abstract}
	In this study, we investigate $N$-player stochastic differential games with regime switching, where the player dynamics are modulated by a finite-state Markov chain. We analyze the associated Nash system, which consists of a system of coupled nonlinear partial differential equations, and establish the existence and uniqueness of solutions to this system, thereby proving the existence of a unique Nash equilibrium. Additionally, we examine the mean field game problem under the same regime-switching framework. We derive a connection between the Nash equilibrium of the MFG and a forward-backward stochastic differential equation with jumps, and demonstrate the unique solvability of this equation. Finally, we explore the propagation of chaos and show that the optimal control obtained from the MFG serves as an approximate Nash equilibrium for the $N$-player problem.
\end{abstract}

\section{Introduction}
Stochastic differential games are powerful tools for analyzing the uncertainty and dynamic strategic interactions among rational agents in a wide range of fields including engineering, economics, finance and biology. In particular, the concept of Nash equilibria has played a central role in understanding how self-interested agents make decisions in a competitive environment. 
This paper investigates multi-player finite horizon stochastic differential games, specifically those with regime-switching dynamics. Within this intricate framework, we establish the existence and uniqueness of Nash equilibria under mild conditions. Our existence and uniqueness results carry forward to symmetric multi-player games where complex player interactions are replaced by mean field approximation, that is when every player dynamics are influenced by, and in turn influence, an aggregate distribution of the states of all players. Finally we explore the regime switching mean field game (MFG) problem where the population of players is allowed to go to infinity. Here a representative player follows a McKean-Vlasov SDE influenced by its conditional law, given the regime switches. We perform a probabilistic study of the associated forward-backward stochastic differential equation (FBSDE), with jumps thanks to regime switching, after motivating the corresponding stochastic maximum principle. Finally we derive convergence and propagation of chaos results which show strategies derived from the limiting problem provide approximate Nash equilibrium strategies to the finite multi-player game.

Multi-player stochastic differential games has been studied in prior works including \cite{friedman-1972,  bensoussan2000stochastic, borkar1992stochastic} both in the finite and infinite horizon setting. These kinds of games are related to classical stochastic control problems in the sense that the latter can be regarded as one player stochastic differential games. Classical stochastic control problems can be studied through an associated Hamilton-Jacobi-Bellman (HJB) PDE, whereas multi-player stochastic differential games exhibit a system of coupled HJB equations. The solutions to such equations are considered in the classical sense in works including \cite{friedman-1972, bensoussan2000stochastic}, and which will be our focus in this paper as well. 

Stochastic control and differential games featuring regime switching dynamics \cite{ lqg-gomez-duncan, mfc-regime} have garnered attention in applications ranging from finance and economics \cite{savku-weber-22, elliott2011stochastic} to engineering \cite{lim1998stochastic, li2017modeling} and climate modeling \cite{elias2014}. Such models capture complex dynamics in which underlying parameters, e.g. drift, volatility, or payoff functions, change according to an exogenous Markov chain, often called a \emph{regime}. This introduces an additional layer of complexity, as agents must adapt their strategies not only to the actions of other agents but also to the evolving state of the system. This dynamical aspect is particularly relevant in various real-world applications, including financial markets, energy management, and epidemiology \cite{reg-switching-ex1, reg-switching-ex3, reg-switching-ex4, bierbrauer2004modeling}. 

MFGs have emerged as a powerful framework to analyze stochastic differential games in which a large number of players interact in a symmetric manner. Introduced and popularized in \cite{Lasry-Lions, Huang-Malhame-Caines}, MFGs have since grown into a vibrant field at the intersection of stochastic analysis, partial differential equations, and game theory \cite{bensoussan2013mean, ref2, ref3} with applications in engineering, science, finance and economics \cite{carmona-applications, djehiche2016mean}. The main insight behind MFGs is that, in the limit as the number of players goes to infinity, one can characterize a Nash equilibrium by analyzing a \emph{representative player} whose state dynamics and cost function depend on the distribution of all the other players’ states. This perspective transforms a high-dimensional problem into one that couples a single-agent stochastic optimal control problem with a consistency condition on the law of the representative player.

Given the success of mean field game theory in handling large-scale interactions, it is natural to incorporate regime-switching features into MFGs and ask whether the same unifying viewpoint continues to yield tractable analysis. Indeed, several works have devoted their attention to this setting. In \cite{bensoussan2020mfg-regime-jump} a regime switching jump diffusion is considered for the state evolution, however the regime switches appear only through jump terms. In the Linear-Quadratic framework, regime switching mean field games have been explored in \cite{jian2024convergence, lv-xiong-2024linear}. While one typically treats the idiosyncratic noise of each player via standard diffusion processes, the common noise in MFGs is often modeled by a Brownian motion shared among all agents \cite{carmona2016mean, ref3}. In regime switching mean field games, we replace that Brownian motion by a finite-state Markov chain that switches randomly among states. However, introducing regime-switching common noise brings about additional subtleties. The presence of these discrete jumps changes both the structure of the underlying partial differential equations and the forward-backward stochastic differential equations  associated with the control problems. While there exist classical results on MFGs with diffusion-type common noise \cite{ref3, carmona2016mean}, we provide a comprehensive treatment of regime-switching common noise in the $N$-player game setting, along with a proof of convergence to the MFG limit. We describe further these two important threads of research in our paper:

\noindent
{\bf $N$-player regime-switching stochastic differential games:} We study existence and uniqueness of Nash equilibria in finite interacting network of players who are all subject to a regime-switching common noise. The regimes of the Markov chain appear as additional discrete states in the coefficients of the SDE. We formulate the game by considering a set of coupling PDEs for the values, one for each player and for each of the discrete states of the common Markov chain, augmented by jump terms reflecting transitions between these regimes. Furthermore, we interpret solutions to the relevant HJB equations in the classical sense. Within the context of multi-player stochastic differential games with regime switching dynamics, establishing the existence and uniqueness of Nash equilibrium is a formidable challenge. Such results are available only in specialized scenarios \cite{xiong-impulse-automatica, song-lqg, reg-switching-ex1} or when the agent action spaces are bounded. However, understanding Nash equilibria in general multi-player games with regime switching dynamics and unbounded action spaces has profound implications across various domains. Our first main contribution is a rigorous analysis of these PDEs that establish the existence and uniqueness of the $N$-player Nash equilibrium under suitable regularity conditions. In addition, we extend these results to a mean-field interaction setting in which each player’s dynamics depend on the empirical distribution of all other players’ states. Our work thus provides the following key advances:
\emph{(i) Coupled PDE Analysis:} By formulating the $N$-player regime-switching game through coupled HJB-type PDEs, linked via jump terms reflecting regime transitions, we establish existence and uniqueness for these PDE systems. Switches in the external regime leads to corresponding switches in the cost functionals and dynamics for all players, and our framework accommodates these discrete jumps.
\emph{(ii) Uniqueness of the $N$-player Nash Equilibrium:} We show that the feedback strategies derived from our PDE analysis give rise to a unique Nash equilibrium. While the discrete nature of the regime-switching noise generally complicates uniqueness arguments, we leverage structural properties to prove uniqueness under appropriate conditions. 

\noindent
{\bf Regime-switching mean field game:} Having established existence and uniqueness at the finite level, we then examine the behavior of the system in the limit as $N \to \infty$. \emph{(i) Probabilistic Analysis and the Mean Field Limit:}  In contrast to purely diffusion-driven games, the common noise in our setting is partially described by a finite-state Markov chain. Rather than continuing with a PDE-based approach, we adopt the FBSDE framework of \cite{carmona-delarue-siam}. In doing so, we derive a McKean–Vlasov FBSDE that incorporates regime-switching by treating the finite-state Markov chain as additional discrete states in both the forward SDE (tracking the continuous state and regime) and the backward SDE (capturing the costate and associated jumps). Under appropriate structural conditions, this regime-switching McKean–Vlasov FBSDE system has a unique solution in the strong sense. We achieve this step through available results in \cite{rolon2024markovian}. 
\emph{(ii) Propagation of Chaos and Approximate Nash Equilibria:} Leveraging the regularity properties of our FBSDE system, we establish a propagation of chaos result, demonstrating that the solution of the infinite-population limit approximates the $N$-player system. Specifically, when each player adopts the mean field feedback strategy, one obtains an $\ep_N$-Nash equilibrium in the finite game, with $\ep_N \to 0$ as $N \to \infty$. We further quantify the convergence rate. 

Technically, our paper draws from both PDE and FBSDE treatments of MFGs. On the PDE side, our work bears similarity to that of \cite{ref3} and others who studied MFG systems with different types of coupling. On the FBSDE side, we lean on techniques similar to those in \cite{carmona-delarue-siam, ref3, rolon2024markovian} but adapted to handle the discontinuities introduced by regime switching. This adaptation requires careful treatment of jump processes and new estimates for solutions of the forward and backward equations. Additionally, we verify maximum principle conditions that tie the distribution of the representative player’s state–regime pair to the coefficients of the FBSDE, and then show that these solutions correspond to equilibrium strategies in the infinite-population limit.

\emph{Structure and Organization of the Paper.}
Our paper is organized as follows. In Section~\ref{sec:prob-setup}, we present the setup of the $N$-player regime-switching stochastic differential game with general interaction not-necessary mean-field. We introduce the state dynamics, the Markov chain driving the regime-switching, and the cost functionals for each player.  Then we state the PDE system that characterizes the value functions of the players, and prove the existence and uniqueness of solutions to this coupled PDE system. Our considerations are complicated since we do not assume the action spaces to be bounded to begin with, and our techniques are similar to \cite{ref3} with added considerations for the presence of the regime switching Markov chain. After considering truncation functions to make coefficients of interest bounded and obtaining regularity properties of the solution, we provide the main theorem on the well-posedness of the PDE system in Theorem~\ref{Thm:3.1}. We also discuss the arguments that link these PDEs to the original $N$-player game, establishing that the candidate strategies indeed constitute a unique Markovian Nash equilibrium.
In Section~\ref{sec:N-mfg},  we use the results obtained in Section~\ref{sec:prob-setup} for general interaction, to state results for $N$-player regime switching games with mean field interaction. This sets the stage for the regime switching mean field game problem in Section~\ref{sec:rs-mfg} . Here we shift gears to the probabilistic approach. We explain the stochastic maximum principle, thereby obtaining the FBSDE system with both  McKean–Vlasov structure and regime switching. We state existence and uniqueness result for this regime switching McKean–Vlasov FBSDE using \cite{rolon2024markovian} and derive added regularity results needed in the next section.
Finally, in Section~\ref{sec:prop}, we establish the propagation of chaos property and show how the finite-player game converges to the mean field game as $N \to \infty$. We quantify the $\ep_N$-Nash property of the strategies derived from the limiting mean field solution when used in the finite $N$-player game. Thus, our results indicate that, under the stated assumptions, each player’s best-response strategy in the large-population limit is approximately optimal for any finite but large $N$, and we estimate the rate of convergence for this approximation.

We provide below few notations used frequently in this paper.
\begin{notation}\label{sec:notation}
\begin{enumerate}[label={(N\arabic*)}]
	\item $|x|$ denotes $l^2$ norm when $x$ is a vector and Frobenius norm when $x$ is a matrix or tensor. 
	\item $B_r(x)$ denote the open ball center at $x$ with radius $r$.
	\item $\mathcal{P}(\mathbb{R}^d)$ denote the space of probability measure on $\mathbb{R}^d$. $\mathcal{P}_p(\mathbb{R}^d)$ stands for the subspace of $\mathcal{P}(\mathbb{R}^d)$ of probability measures of order $p$, i.e., having a finite moment of order $p$, namely $M_p(\mu)=(\int_{\mathbb{R}^d}|x|^pd\mu(x))^{1/p}$, equip with $W_p$ the $p$-Wasserstein distance. 
	\item In this paper, we will use \textbf{boldface symbol} to denote vectors in $\mathbb{R}^{Nd}$.
	\item \label{not:S-L} Denote 
	$\mathcal{L}^2\left(D\right)=\left\{\xi: \Omega \rightarrow D, \mathcal{F}\right.$-measurable, $\left.\mathbb{E}|\xi|^2<\infty\right\}$
	\begin{align*}
		& \mathcal{S}^2_T(D) \\
		& \quad=\left\{\varphi:[0, T] \times \Omega \rightarrow D, \mathcal{F} \text {-adapted càdlàg process, } \mathbb{E}\left[\sup _{0 \leq t \leq T}\left|\varphi_t\right|^2\right]<\infty\right\}
	\end{align*}
	and
	\begin{align*}
		& \mathcal{L}^0_T(D)=\left\{\psi:[0, T] \times \Omega \rightarrow D, \mathcal{F} \text {-progressively measurable process }\right\} \\
		& \mathcal{L}^2_T(D)=\left\{\psi \in \mathcal{L}^0_T(D):\|\psi\|_2^2=\mathbb{E}\left[\int_0^T\left|\psi_t\right|^2 d t\right]<\infty\right\}.
	\end{align*}

\end{enumerate}
\end{notation}

\section{Regime Switching $N$-player Game with General Interaction}\label{sec:prob-setup}
This section is devoted to study the $N$-player regime-switching stochastic differential game with general
interaction not-necessary mean-field. We begin by describing the set-up of the game.
\subsection{General interaction $N$-player game.} Let \(T>0 \) be a fixed time horizon. Consider a probability space \( (\Omega, \mathcal{F},(\mathcal{F}_{s})_{s \in[t, T]}, \mathbb{P}) \) satisfying the {usual assumptions}. Define $N$ independent copies of  $d$-dimensional Brownian motions \( (W_{s}^k)_{s \in[0, T]}, k = 1, \ldots, N \). Let \( {(I_s)}_{s \in [0,T]} \) be a continuous time Markov chain with finite state space \( \mathcal{S}=\{1, \ldots, s_{0}\} \) and generator \( Q=(q_{i j})_{1 \leq i, j \leq s_{0}} \). For each $t>0$, denote $\mathcal{F}_t^I=\sigma\{I_s: 0 \leq s \leq t\}$, $\mathcal{F}_t^W=\sigma\{W_s^k:$ $0 \leq s \leq t, k=1,\cdots,N\}$, and put $\mathcal{F}_t=\mathcal{F}_t^{W, I}=\sigma\{W_s^k, I_s: 0 \leq s \leq t,k=1,\cdots,N\}$. Consider an interacting $N$-player game where the $k$-th player has state $X_t^k \in \R^d$ for times $t \in [0,T]$ satisfying the sde  
\beq\label{eq:player-sde}
dX_{s}^{k}=b^k(s, \bm{X}_{s}, \beta_s^k , I_{s}) d s+\sigma^k(s, \bm{X}_{s}, I_{s}) dW_{s}^{k},
\eeq
for $0 \leq s \leq T$ and $k=1, \ldots, N$.  Here we have used $ \bm{X}_{s}=(X_{s}^{1},..., X_{s}^{N}) \in (\R^{d})^N$. 
Furthermore, $\beta^k = (\beta_s^k)_{0 \leq s \leq T}$ is a $A$-valued progressively measurable stochastic process, where $A$ is the action space and $\beta_s^k$ represents the strategy chosen by player $k$ at time $s$. In addition, $\beta^k$ satisfies the admissibility condition: $\E [\int_0^T |\beta_s^k|^2 ds] < \infty$. Let $\mathbb{A}$ represent the space of all such strategies.
Below we state assumptions that ensure the existence and uniqueness of \eqref{eq:player-sde} for $\beta^k \in \mathbb{A}$. For each $(t, \mathbf{x}, i_0) \in [0,T] \times \R^{dN} \times \cs$ define the following cost functional for the $k$-th player:
\begin{align}\label{eq:cost}
		J_{k}(t, \mathbf{x}, i_0, \beta^{k}, \bm{\beta}^{-k} )= \mathbb{E}\left[{\int_t^{T} f^k (s, \bm{X}_{s}, \beta_{s}^{k}, I_{s-}) d s} +g^k ( \bm{X}_{T}, I_{T})| \bm{X}_t = \mathbf{x}, I_t = i_0\right],
\end{align}
for $k=1, \ldots, N$ 
where $X^k$ solves \eqref{eq:player-sde} modulated by the strategy $\beta^k$.
\begin{remark}\label{rem:A}
	Note that the action space $A$ is often assumed to be a closed and bounded subset of $\R^m$. However, in this work, we allow $A=\R^m$.
\end{remark}
\begin{definition}\label{def:NE-0}
	A strategy profile $ \hat{\bm{\beta}} = (\hat{\bm{\beta}}^1, \ldots, \hat{\bm{\beta}}^N) \in \mathbb{A}^N$ is said to be a Nash equilibrium 
	if for each $(t, \mathbf{x}, i_0) \in [0,T] \times \R^{dN} \times \cs$, the following equilibrium condition holds for all $k=1, \ldots, N$:
	\[
	J_{k}(t, \mathbf{x}, i_0, \hat{\bm{\beta}} ) \leq J_{k}(t, \mathbf{x}, i_0,  (\beta, \hat{\bm{\beta}}^{-k})), \quad \forall \bm{\beta} \in \mathbb{A}.
	\]
\end{definition}

\subsection{Assumptions}
Our results rely on the following assumptions on the coefficients.
\begin{hypothesis}\label{hyp:Nash} 
	The admissible control set $A$ is the entire $\R^{m}$. For any fixed $i_{0} \in \mathcal{S}$ {and $k \in \{1,...,N\}$}, there exist positive constants $L,L^\prime$ and $\lambda$ such that, such that for $b=b^k, \sigma= \sigma^k, f=f^k$ and $g=g^k$
	
	\begin{enumerate}[label={(H\arabic*)}]
		\item \label{hyp:H1} The drift \( b \) is an affine function of \( \alpha \): 
			\[
			b(t, \mathbf{x}, \alpha,i_0)=b_{1}(t, \mathbf{x},i_0)+b_{2}(t,i_0) \alpha,
			\]
			where \( b_{2}:[0, T] \times \cs \mapsto \mathbb{R}^{d \times m} \) and \( b_{1}: [0, T] \times \mathbb{R}^{dN} \times \cs \mapsto \mathbb{R}^{d} \) are measurable, bounded and continuously differentiable in $t$ and $(t, \mathbf{x})$ respectively. 
		
		\item {\label{hyp:H2} For any $ t \in [0,T]$, $f(t,\mathbf{x},\alpha,i_{0}) $ is once continuously differentiable in $ \mathbf{x}$ and $\alpha $, and the derivatives $\nabla_{\mathbf{x}}f , \nabla_{\alpha}f $ are $L$-Lipschitz continuous in $ \mathbf{x} $ and $\al$. Moreover, $ f $ is strongly convex in $\mathbf{x}$ and $ \alpha $:
		\begin{align}\label{eq:fcov}
			f(t,\mathbf{x}, \alpha^{\prime},i_{0})-f(t,\mathbf{x}, \alpha,i_{0})-\langle (\mathbf{x}^{\prime}-\mathbf{x},\alpha^{\prime}-\alpha) , \nabla_{(\mathbf{x},\alpha)} f(t, \mathbf{x}, \alpha,i_{0})\rangle \geq \lambda|\alpha^{\prime}-\alpha|^{2}.
		\end{align}}
		
		\item \label{hyp:H3}For all $ (t,\mathbf{x},\alpha) \in [0,T] \times \mathbb{R}^{dN}  \times A$, 
		\[|\nabla_{\alpha} f(t, \mathbf{x}, \alpha,i_{0})| \leq L^\prime(1+|\alpha|), \text{ where } L^\prime < \lambda. \] 
		 \begin{equation}\label{eq:ff}
			|f(t, \mathbf{x}, \alpha,i_{0})| \leq L(1+|\alpha|^{2}).
		\end{equation}
		\item \label{hyp:H4} For any $ i,j \in \mathcal{S}, i\neq j $, the transition rates
		$q_{ij} \leq L $.
		
		\item {\label{hyp:H5} $ \sigma: [0, T] \times \mathbb{R}^{dN} \times \cs \mapsto \mathbb{R}^{d \times d} $ is Lipschitz continuous in $ \mathbf{x} $:
		\begin{equation*}
		|\sigma(t,\mathbf{x},i_{0})-\sigma(t,\mathbf{y},i_{0})| \leq L(|\mathbf{x}-\mathbf{y}|).
		\end{equation*}
		}
		
		\item \label{hyp:H6}
		{$\sigma(t,\mathbf{x},i_{0})$ is convex in $x^k$. Moreover,} for $(t,\mathbf{x}) \in [0,T] \times \mathbb{R}^{dN}$, $\sigma(t,\mathbf{x},i_{0})$ is invertible has continuous second derivative in $t$ and $\mathbf{x}$. Denote $a_{ij}(t,\mathbf{x},i_{0})=\sigma(\sigma)^{\prime}_{ij}(t,\mathbf{x},i_{0})$ {and $\xi \in \R^d$. Then} the following are satisfied: 
		\begin{align*}
			\nu_{1}|\xi|^{2} &\leq \sum_{i,j}a_{ij}(t,\mathbf{x},i_{0})\xi_{i}\xi_{j} \leq \nu_{2} |\xi|^{2}, \quad \nu_i > 0 \\
			&|\frac{\partial a_{ij}(t,\mathbf{x},i_{0})}{\partial x_{l}}| \leq \nu_{2} . \quad l=1,\cdots, Nd
		\end{align*}
		\item \label{hyp:H7} The {terminal cost} $ g $ is continuously differentiable in $\mathbf{x}$. Moreover,
		\begin{align}
			|g(\mathbf{x},i_{0})| \leq L,   \label{eq:gg}\\
			|\nabla_\mathbf{x}g(\mathbf{x},i_0)|\leq L, \nonumber
		\end{align} 
	\end{enumerate}
\end{hypothesis}
{
As a consequence of our assumptions, we have the following result.
\begin{corollary}\label{cor:sigma-bd}
	From Hypothesis~\ref{hyp:H6} we have that for any $(t,\mathbf{x},i_0) \in [0,T] \times \mathbb{R}^{dN} \times \mathcal{S}$, there exist constant $C_1,C_2>0$ such that 
	\begin{align*}
		C_1\leq|\sigma(t,\mathbf{x},i_0)|^2\leq C_2.
	\end{align*}
\end{corollary}
\begin{proof}
	From \ref{hyp:H6} we know that for any $\xi \in \mathbb{R}^d$ that
	\begin{equation*}
			\nu_{1}|\xi|^{2} \leq \xi^{\prime}\sigma\sigma^{\prime}\xi=|\sigma^{\prime}\xi|^2 \leq \nu_{2} |\xi|^{2}, \quad \nu_i > 0,
	\end{equation*} 
	which is equivalent to 
	\begin{equation*}
		\nu_1 \leq \|\sigma^{\prime}\|_2^2 \leq \nu_2.
	\end{equation*}
	Since $\|\sigma^\prime\|_2 \leq |\sigma^\prime|\leq \sqrt{d}\|\sigma^\prime\|_2$ we have that 
	\begin{equation*}
		\nu_1 \leq |\sigma^{\prime}|^2 \leq d\nu_2.
	\end{equation*}
	\end{proof}
	We introduce a few more notations related to the regime switching Markov chain $I_t$.
\begin{notation}
	\begin{enumerate}[label={(N\arabic*)}, start=6]
			\item \label{not:Mt} Define $M_{i_{0} j_{0}}(t)$ to be
		\beq\label{eq:M}
		M_{i_{0} j_{0}}(t)=[M_{i_{0} j_{0}}](t)-\left\langle M_{i_{0} j_{0}}\right\rangle(t) 
		\eeq
		and {for} $i_0 \neq j_0 \in \cs$
		\begin{align}\label{eq:[M],<M>}
			[M_{i_{0} j_{0}}](t)=\sum_{0 \leq s \leq t} \mathds{1}_{\{I_{s-}=i_{0}\}} \mathds{1}_{\{I_{s}=j_{0}\}}, \quad
			\left\langle M_{i_{0} j_{0}}\right\rangle(t)=\int_{0}^{t} q_{i_{0} j_{0}} \mathds{1}_{\{I_{s-}=i_{0}\}} ds,
		\end{align}
		and 
		$$
		[ M_{i_0 i_0} ] (t) = \lla M_{i_0 i_0} \rra (t) = 0 \quad \text{for each } i_0 \in \cs.
		$$
		\item \label{not:M}	Denote
		\begin{align*}
			& \mathcal{M}^2_T(D) \\
			& \quad=\left\{\lambda=\left(\lambda_{i_0 j_0}: i_0, j_0 \in \mathcal{M}\right) \text { such that } \lambda_{i_0 j_0} \in \mathcal{L}^0_T(D)\right. \text { predictable } \\
			& \left.\quad \lambda_{i_0 i_0} \equiv 0, \text { and } \sum_{i_0, j_0 \in \mathcal{M}} \mathbb{E} \int_0^T\left|\lambda_{i_0 j_0}(t)\right|^2 d\left[M_{i_0 j_0}\right](t)<\infty\right\}.
		\end{align*}
		For a collection of $\mathcal{F}$-predictable functions $\lambda_t=\left(\lambda_{i_0 j_0}(t)\right)_{i_0, j_0 \in \mathcal{M}}, t \geq 0$, we denote
		\begin{align*}
			\lambda_t \cdot d M_t & =\sum_{i_0, j_0 \in \mathcal{M}} \lambda_{i_0 j_0}(t) d M_{i_0 j_0}(t),\\
			\lambda_t^{\circ 2} \cdot d [M]_t & =\sum_{i_0, j_0 \in \mathcal{M}} |\lambda_{i_0 j_0}(t)|^2 d [M_{i_0 j_0}](t).
		\end{align*}
	\end{enumerate}
\end{notation}
}
\begin{remark}
Note $M_{i_0 j_0}(t)$ is a purely discontinuous and square integrable martingale with respect to $\cf_t$. 
The processes $[M_{i_0, j_0}](t)$ and $\left\langle M_{i_{0} j_{0}}\right\rangle(t)$ are respectively its optional and predictable
quadratic variations.
\end{remark}
\subsection{The $N$-player Nash system.}
We now provide heuristics for obtaining the PDEs describing the player dynamics. Fix $k$ and assume every player $i \neq k$ adopts the strategy $\hat{\beta}^i$. Then the best response problem for the $k$-th player is the optimization $\inf_{{\beta \in \mathbb{A}}} J_k(t, \mathbf{x}, i_0, (\beta, \hat{\bm \beta}^{-k})) =: U_k(t, \mathbf{x}, i_0)$ which is a stochastic control problem whose value solves the following HJB PDE
\begin{align}\label{eq:pde-bestr}
	0 =& \partial_t U_k(t, \mathbf{x}, i_0) + H^k(t, \mathbf{x}, \nabla_kU_k(t, \mathbf{x}, i_0),i_0 ) 
	+ \sum_{{i \neq k}} \nabla_{i} U_{k}(t, \mathbf{x}, i_{0}) \cdot  b^k(t, \mathbf{x}, \hat{\beta}^i, i_0) \nonumber\\
	&+ \frac{1}{2} \sum_{i=1}^N \textnormal{tr} (\nabla_i^2 U_k(t, \mathbf{x}, i_0) (\si^k)' \si^k(t, \mathbf{x}, i_0)) +\sum_{j_0 \in \cs} q_{i_0 j_0} (U_k(t, \mathbf{x}, j_0) - U_k(t, \mathbf{x}, i_0)), \\
	&U_k(T, \mathbf{x}, i_0) = g^k(\mathbf{x}, i_0),  \nonumber
\end{align}

where $H^k$ are the Hamiltonian given by
\begin{align}\label{eq:H}
	H^k(t,\mathbf{x}, p, i_0)= \inf_{\al \in \R^m} H^k(t,\mathbf{x}, p, \al, i_0)=\inf_{\al \in \R^m} [ f^k(t, \mathbf{x}, \al, i_0) + p \cdot b^k(t, \mathbf{x}, \al, i_0)  ]. 
\end{align}
The right hand side of equation~\eqref{eq:H} has a unique minimizer thanks to the following lemma, while the proof of this lemma is similar to \cite[Lemma 2.1]{carmona-delarue-siam} and \cite[Lemma 3.3]{ref2}.
\begin{lemma}\label{Lem:UniOpt}
	Under \ref{hyp:H1}-\ref{hyp:H3} of Hypothesis~\ref{hyp:Nash}, {for any $k=1,\cdots, N$} there exists a unique minimizer $ \hat{\alpha}(t, \mathbf{x}, p, i_{0}) $ of {
	\begin{equation} \label{eq:Hdef}
		\alpha \mapsto H^k(t,\mathbf{x}, p, \al, i_0)=f^k(t, \mathbf{x}, \alpha, i_{0}) + p \cdot b^k(t, \mathbf{x}, \alpha, i_{0}) . 
		\end{equation}}
	Furthermore, $ \hat{\alpha} $ is measurable, locally bounded, and Lipschitz continuous with respect to $ \mathbf{x} $ and $ p $ uniformly in $ t $, with 
	\beq\label{eq:al-ub}
	|\hat{\alpha}(t, \mathbf{x}, p, i_{0})| \leq C(1+|p|) ,
	\eeq
	for any $(t,x,m,i_0) \in [0,T] \times \mathbb{R}^{dN} \times \cs. $
\end{lemma}
\begin{proof}
	{For simplicity we omit the index $k$ in this proof.} Using strong convexity in $\al$ of $f$ from \ref{hyp:H2} and linearity in $\alpha$ of $b$ from \ref{hyp:H1}, we have strong convexity of the Hamiltonian $H$. Thus there exists a unique minimizer $ \hat{\alpha}(t, \mathbf{x}, p, i_{0}) $ of $\alpha \mapsto H(t,\mathbf{x}, p, \al, i_0)$. The measurablility of $ \hat{\alpha}(t, \mathbf{x}, p, i_{0}) $ is a consequence of classical measurable selection theorem, see \cite[Theorem 14.37]{rockafellar2009variational}. Furthermore, by strong convexity of $H$ in $\al$ we have
	\begin{align*}
			H(t,\mathbf{x}, p, 0, i_0) & \geq H(t,\mathbf{x}, p, \hat{\alpha}(t, \mathbf{x}, p, i_{0}), i_0 ) \\
			& \geq H(t,\mathbf{x}, p, 0, i_0)+\left\langle\hat{\alpha}(t, \mathbf{x}, p, i_{0}), \nabla_{\alpha} H(t,\mathbf{x}, p, 0, i_0)\right\rangle+\lambda|\hat{\alpha}(t, \mathbf{x}, p, i_{0})|^{2}.
	\end{align*}
	This implies 
	\begin{equation*}
		\lambda \left|\hat{\al}(t,\mathbf{x},p,i_0)\right|^2 \leq -\langle\hat{\al}(t,\mathbf{x},p,i_0), \nabla_\al H(t,\mathbf{x},p,0,i_0) \rangle.
	\end{equation*}
	From Cauchy-Schwarz inequality we now get
	\begin{equation*}
		\left|\hat{\al}(t,\mathbf{x},p,i_0)\right| \leq \frac{1}{\lambda}\left| \nabla_\al H(t,\mathbf{x},p,0,i_0) \right|.
	\end{equation*}
	From the definition of $H$ in \eqref{eq:Hdef} and Hypothesis~\ref{hyp:H1}
	\begin{align*}
		|\hat{\alpha}(t, \mathbf{x}, p, i_{0})| \leq \lambda^{-1}\left(\left|\nabla_{\alpha} f(t, \mathbf{x}, 0,i_0)\right|+\left|b_{2}(t,i_0)\right||p|\right) 
		\leq C(1+|p|),
	\end{align*}
	for some constant $C$, where the last inequality follow from Hypothesis~\ref{hyp:H3}. The Lipschitz continuity follows from a suitable adaptation of the implicit function theorem to the variational inequalities that $\hat{\alpha}(t, \mathbf{x}, p, i_{0})$ appears as the unique solution of. Indeed, for {$\mathbf{x}, \mathbf{x}^{\prime} \in\mathbb{R}^{Nd} , p, p^{\prime} \in \mathbb{R}^d$} and $(t, i_0) \in[0, T] \times \mathcal{S}$, from the first order optimality condition we have following the two inequalities:
	\begin{align*}
		& \left(\hat{\alpha}^\prime-\hat{\alpha}\right) \cdot \nabla_\al H(t, \mathbf{x},  p, \hat{\alpha},i_0) \geq 0, \\
		& \left(\hat{\alpha}-\hat{\alpha}^\prime\right) \cdot \nabla_\al H\left(t, \mathbf{x}^{\prime},  p^{\prime}, \hat{\alpha}^\prime,i_0\right) \geq 0 .
	\end{align*}
	where $\hat{\al}=\hat{\alpha}(t, \mathbf{x},  p,i_0)$ and $\hat{\al}^{\prime}=\hat{\alpha}\left(t, \mathbf{x}^{\prime},  p^{\prime},i_0\right)$. Summing these inequalities, we get:
	\begin{align*}
		 \left(\hat{\alpha}^\prime-\hat{\alpha}\right) \cdot\left(\nabla_\al H\left(t, \mathbf{x}^{\prime},  p^{\prime}, \hat{\alpha}^\prime,i_0\right)-\nabla_\al H(t, \mathbf{x},  p, \hat{\alpha},i_0)\right) \leq 0.
	\end{align*}
	By adding and subtracting the term $\left(\hat{\alpha}^\prime-\hat{\alpha}\right) \cdot\nabla_\al H\left(t, \mathbf{x},  p, \hat{\alpha}^\prime,i_0\right)$ the above inequality becomes:
	\begin{align}\label{eq:al'-alpH'-pH}
		& \left(\hat{\alpha}^\prime-\hat{\alpha}\right) 
		  \cdot\left(\nabla_\al H\left(t, \mathbf{x},  p, \hat{\alpha}^\prime,i_0\right)-\nabla_\al H(t, \mathbf{x},  p, \hat{\alpha},i_0)\right) \nonumber\\
		& \leq\left(\hat{\alpha}^\prime-\hat{\alpha}\right) 
		  \cdot\left(\nabla_\al H\left(t, \mathbf{x},  p, \hat{\alpha}^\prime,i_0\right)-\nabla_\al H\left(t, \mathbf{x}^{\prime},  p^{\prime}, \hat{\alpha}^\prime,i_0\right)\right).
	\end{align}	
	Exchanging the roles of $\alpha$ and $\alpha^{\prime}$ in \eqref{eq:fcov} and summing the resulting inequalities, we have for any $\alpha, \alpha^{\prime} \in A$,
\begin{equation}\label{eq:al'-alpf'-pf}
	\left(\alpha^{\prime}-\alpha\right) \cdot\left(\nabla_\al f\left(t, \mathbf{x},  \alpha^{\prime},i_0\right)-\nabla_\al f(t, \mathbf{x},  \alpha,i_0)\right) \geq 2 \lambda\left|\alpha^{\prime}-\alpha\right|^2 .
\end{equation}
	Using the \eqref{eq:al'-alpf'-pf} together with the fact that $\nabla_{\al} f + p\cdot b_2 = \nabla_{\al}H$ from the definition of $H$, we deduce that:
	\begin{align*}
		 2 \lambda\left|\hat{\alpha}^\prime-\hat{\alpha}\right|^2  \leq\left(\hat{\alpha}^\prime-\hat{\alpha}\right) \cdot\left(\nabla_\al H\left(t, \mathbf{x},  p, \hat{\alpha}^\prime,i_0\right)-\nabla_\al H\left(t, \mathbf{x},  p, \hat{\alpha},i_0\right)\right) .
	\end{align*}
	Plugging \eqref{eq:al'-alpH'-pH} into the above inequality we have
		\begin{align*}
			 2 \lambda\left|\hat{\alpha}^\prime-\hat{\alpha}\right|^2  &\leq\left(\hat{\alpha}^\prime-\hat{\alpha}\right) \cdot\left(\nabla_\al H\left(t, \mathbf{x},  p, \hat{\alpha}^\prime,i_0\right)-\nabla_\al H\left(t, \mathbf{x}^{\prime},  p^{\prime}, \hat{\alpha}^\prime,i_0\right)\right) \\
			& \leq C\left|\hat{\alpha}^\prime-\hat{\alpha}\right|\left(\left|\mathbf{x}^{\prime}-\mathbf{x}\right|+\left|p^{\prime}-p\right|\right)
		\end{align*} 
		where  $C$  only depends upon the bound for  $b_2$  and the Lipschitz-constant of $ \nabla_\al f $ as a function of $ \mathbf{x}$.

\end{proof}
Assuming the existence of the solution to \eqref{eq:pde-bestr} the best response to player $k$ is now given by
\beq\label{eq:best-resp}
\beta_s^k = \hat{\al}^k (s, \mathbf{x}, \nabla_k U_k(s, \mathbf{x}, i_0), i_0),
\eeq 
where $\hat{\al}^k$ is introduced in Lemma~\ref{Lem:UniOpt}.
To find a Nash equilibrium, every player solves this best response problem. Plugging \eqref{eq:best-resp} in \eqref{eq:pde-bestr}, the system of equations to solve is transformed to:
\begin{align}\label{eq:Nash}\tag{Nash-HJB}
	0=&  \partial_{t} v_i(t,\mathbf{x}, i_{0}) + \cg^{i} v(t, \mathbf{x}, i_0) + f^i(t, \mathbf{x}, \hat{\alpha}^{i}(t,\mathbf{x},\nabla_{i}v_{i}(t,\mathbf{x},i_0),i_{0}),i_{0}),  \\
	&v_i(T,\mathbf{x}, i_{0})=g^i(\mathbf{x}, i_{0}), \nonumber
\end{align}
where the operator $\cg^{i}$ is defined as
	\begin{align*}
		\cg^{i} v(t, \mathbf{x}, i_0)   = 
		&\sum_{k = 1}^{N} \nabla_{k} v_i(t,\mathbf{x}, i_{0}) \cdot	b^i(t, \mathbf{x}, \hat{\alpha}^{k}(t,\mathbf{x},\nabla_{k}v_{k}(t,\mathbf{x},i_0),i_{0}),i_{0} ) \\
		&+\frac{1}{2} \sum_{k = 1}^{N} \operatorname{tr} ( \nabla^{2}_{k} v_i(t,\mathbf{x}, i_{0})(\sigma^i)^{\prime} \sigma^i(t, \mathbf{x}, i_{0} )) 
		+\sum_{j_{0} \in \mathcal{S}} q_{i_{0} j_{0}}(v_i(t, \mathbf{x}, j_{0})-v_i(t, \mathbf{x}, i_{0}))
	\end{align*}
\subsection{Unique Solvability of \eqref{eq:Nash}.}
The majority of our work in this paper is to show that \eqref{eq:Nash} exhibits a unique solution. We perform this operation through several stages. In the following Lemma~\ref{Lem:BdCoef}, we show this result under some additional assumptions on the coefficients. We first state an elementary inequality proved for completeness. 
\begin{lemma}\label{Lem:positivec1c2}
	For $a_{ij}, a_i, x_i \in \mathbb{R}$, where $i,j \in \{1,2,...,n\}$, if $\{a_{ij}\}$ and $\{a_i\}$ are bounded, there always exist positive constants $c_1, c_2$ such that 
	\begin{align*}
		c_1 \sum_{i=1}^n x_i^2 +\sum_{i,j=1}^na_{ij}x_ix_j+\sum_{i=1}^na_ix_i+c_2 \geq 0.
		\end{align*}	
\end{lemma}
\begin{proof}
		Observe 
			\begin{align}\label{eq:scra1}
				\mathscr{A}_1=\frac{1}{2}\sum_{i,j=1}^n a_{ij}\left(x_i^2+x_j^2\right)+\sum_{i,j=1}^na_{ij}x_ix_j=\frac{1}{2}\sum_{i,j=1}^n a_{ij}\left(x_i+x_j\right)^2 \geq 0,
			\end{align}
			and 
			\begin{align}\label{eq:scra2}
				\mathscr{A}_2=\sum_{i=1}^N\left(x_i^2+a_ix_i+\frac{a_i^2}{4}\right)=\sum_{i=1}^N\left(x_i+\frac{a_i}{2}\right)^2 \geq 0.
				\end{align}
Then
	\begin{align}\label{eq:c1c2aiaij}
	c_1 \sum_{i=1}^n x_i^2 &+\sum_{i,j=1}^na_{ij}x_ix_j+\sum_{i=1}^na_ix_i+c_2 \nonumber\\
	&=c_1 \sum_{i=1}^n x_i^2-\frac{1}{2}\sum_{i,j=1}^n a_{ij}\left(x_i^2+x_j^2\right) +\mathscr{A}_1+\mathscr{A}_2-\sum_{i=1}^n x_i^2-\sum_{i=1}^n \frac{a_i^2}{4}+c_2 
\end{align}
Using \eqref{eq:scra1} and \eqref{eq:scra2} in \eqref{eq:c1c2aiaij} we have
\begin{align*}
	c_1 \sum_{i=1}^n x_i^2 &+\sum_{i,j=1}^na_{ij}x_ix_j+\sum_{i=1}^na_ix_i+c_2 \\
	&\geq (c_1-1) \sum_{i=1}^n x_i^2-\frac{1}{2}\sum_{i,j=1}^n a_{ij}\left(x_i^2+x_j^2\right)-\sum_{i=1}^n \frac{a_i^2}{4}+c_2\geq 0,
\end{align*}
where the last inequality follows by choosing $c_1 \geq 1+n \sup_{i,j} a_{ij}$ and $c_2 \geq \sum_{i=1}^n\frac{a_i^2}{4}$.
\end{proof}
In the following lemma, we prove the existence and uniqueness of \eqref{eq:Nash} under certain assumptions on the coefficients.
\begin{lemma}\label{Lem:BdCoef}
	 Assume the running cost $ f^i $ is bounded. Then under Hypothesis~\ref{hyp:Nash}, the system of equations~\eqref{eq:Nash} admits a unique solution $\{v_i(t, \mathbf{x},i_0), \, i \in [N], i_0 \in \cs \}  $ 
which is bounded and continuous on $[0,T] \times \mathbb{R}^{Nd}  $, differentiable in $\mathbf{x} \in \mathbb{R}^{Nd}$ with a bounded and continuous gradient on $ [0,T) \times \mathbb{R}^{Nd} $.
\end{lemma}
\begin{proof} 
{First modify the Nash system with bounded $f^i$ to be defined on $[0,T] \times B_r$ where for some positive integer $r$, $B_r \subset \R^{Nd}$ is the ball centered at the origin with radius $r$, with boundary conditions $v_i(t, \mathbf{x}, i_0) = g^i(\mathbf{x}, i_0)$. Then}	note by Hypothesis~\ref{hyp:Nash}
	\[ \sigma^i(\sigma^i)^{\prime}_{ij}(t,\mathbf{x},i_{0})\xi_{i}\xi_{j} \geq 0 .\]
	Since $f^i$ and $ q_{i_{0}j_{0}} $ are bounded, by Lemma~\ref{Lem:positivec1c2} there exist positive constants $ c_{1},c_{2} $ such that
	\begin{align*}
		\left[\sum_{j_{0} \neq i_{0}} q_{i_{0} j_{0}}(v_i(t, \mathbf{x}, j_{0})-v_i(t, \mathbf{x}, i_{0}))+ f^i(t,\mathbf{x},\alpha,i_{0})\right]&v_{i}(t,\mathbf{x},i_{0})\\
		&\geq-c_{1}(\sum_{k_0 \in \mathcal{S}}\sum_{i=1}(v_{i}(t,\mathbf{x},k_0))^{2})-c_{2}.
	\end{align*}
	Consequently from \cite[p596, (7.4)]{ref8} we have
	\begin{align}\label{ieq:a}
		|v_i(t, \mathbf{x}, i_{0})| 
		\leq \min _{c>c_{1}} e^{c T}\left[\sup_{\mathbf{x} \in \mathbb{R}^{d},i \in [N]}|g^i(\mathbf{x},i_{0})|+\sqrt{\frac{c_{2}}{c-c_{1}}}\right] \equiv M.
	\end{align}
	Thus, using \eqref{ieq:a}
	\begin{align}\label{ieq:b}
		\left|f^i(t,\mathbf{x},\alpha,i_{0})
		+\sum_{j_{0} \neq i_{0}} q_{i_{0} j_{0}}(v_i(t, \mathbf{x}, j_{0})-v_i(t, \mathbf{x}, i_{0}))\right| \leq C .
	\end{align}
	Then by choosing appropriate $P$, e.g. $P(|p|,|v|)=Ce^{|v|-|p|}$, one has
	\[ C \leq \varepsilon(|v|)+P(|p|,|v|), \]
	where $\varepsilon(M)$ is a sufficiently small number, while $P(|p|,|v|) \rightarrow 0 $ when $|p| \rightarrow \infty$. {From \ref{hyp:H1} and  \eqref{eq:al-ub}, and }from \eqref{ieq:a} and \eqref{ieq:b}, \cite[Eq VII.6.3]{ref8} is satisfied, while \cite[Eq. VII.6.1]{ref8} and \cite[Eq. VII.6.4)]{ref8} follow from Hypothesis~\ref{hyp:Nash}. Then from \cite[Theorem VII.7.1]{ref8}, \eqref{eq:Nash} admits a unique solution $\{v_i(t, \mathbf{x},i_0), \, i \in [N], i_0 \in \cs \}  $ which is bounded and continuous on $[0,T] \times B_r  $, differentiable in $\mathbf{x} \in B_r$ with a bounded and continuous gradient on $ [0,T) \times B_r$, and where the bound is independent of $r$. Then taking limit $r \to \infty$, similar to what is done in \cite[Prop. 3.3]{ma1994solving}, we argue that the lemma is true.
\end{proof}
\begin{notation}
	For ease in presentation, we use the following notations in the proof of Theorem~\ref{Thm:3.1}.   Denote $Y_{t}^{i}=v_{i}^{\pi}(t, \tilde{\bm X}_{t},I_{t}), \quad Z_{t}^{i, j}=\nabla_{j} v_{i}^{\pi}(t, \tilde{\bm X}_{t}, I_{t-}), \quad t \in[t_{0}, T], \quad(i, j) \in\{1, \cdots, N\}^{2},$ also $Z_{t}=(Z_{t}^{i,j})_{i,j=1,...,N} \in (\mathbb{R}^{d})^{N \times N}$. We will define $v^{\pi}_{
		i}$ and $\tilde{\bm X}_{t}$ in the proof. Denote $A_{t}^{i}= \hat{\alpha}^{i}(t, \tilde{\bm X}_{t}, Z_{t}^{i,i}, I_{t-})$, $A_{t}^{\pi,i}= \hat{\alpha}^{\pi}_{i}(t, \tilde{\bm X}_{t}, Z_{t}^{i,i}, I_{t-})$. $B_{t}^{i}=b^i(t,\tilde{\bm X}_t,A_{t}^{i},I_{t-})$ and $\Sigma_{t}^{i}=\sigma^{i}(t,\tilde{\bm X}_{t}, I_{t-})$. In the following proof, {Denote $(\delta v_i)_{j_0,i_0}(t,\mathbf{x},\cdot)=v_i(t,\mathbf{x},j_0)-v_i(t,\mathbf{x},i_0)$.}
\end{notation}
\subsubsection{\bf Strategy of proof under Hypothesis~\ref{hyp:Nash}}\label{sec:strategy}
We remind the reader that Hypothesis~\ref{hyp:Nash} does not assume the boundedness of $f^i$ or $b^i$. We proceed towards removing this assumption, by introducing a truncation function. Namely, let $\pi: \mathbb{R}^{k} \mapsto \mathbb{R}^{k}$ be a smooth function such that 
	\[
	\pi(x)=
	\begin{cases}
		x, & |x| \leq R,\\
		0, & |x| > 2R,
	\end{cases}
	\]
	and the Lipschitz constant of $ \pi $, $ ||\pi||_{\text{Lip}} \leq 2 $. Denote $ \hat{\alpha}^{\pi}_{i}(t,\mathbf{x},p,i_{0})=\pi(\hat{\alpha}^{i}(t,\mathbf{x},p,i_{0})) $. 
In order to deal with the unboundedness of $f^i$ in \eqref{eq:Nash}, we first replace \( f^{i}(t, \mathbf{x}, \hat{\alpha}^{i}(t, \mathbf{x}, \nabla_{i} v_{i}(t, \mathbf{x},i_0),i_{0}),i_{0}) \) by \( f^{i}(t, \mathbf{x}, \hat{\alpha}^{\pi}_{i}(t, \mathbf{x}, \nabla_{i} v^{\pi}_{i}(t, \mathbf{x},i_0),i_{0}),i_{0}) \). Let $ v^{\pi} $ be the consequent value function and call the consequent system the \textbf{truncated Nash system} \label{eq:tNash}. 

{
\begin{align}\label{eq:t-Nash}\tag{t-Nash-HJB}
	0=&  \partial_{t} v_i^{\pi}(t,\mathbf{x}, i_{0}) + \cg^{i} v^{\pi}(t, \mathbf{x}, i_0) + f^i(t, \mathbf{x}, \hat{\alpha}_{i}^{\pi}(t,\mathbf{x},\nabla_{i}v_{i}^{\pi}(t,\mathbf{x},i_0),i_{0}),i_{0}),  \\
	&v_i^{\pi}(T,\mathbf{x}, i_{0})=g^i(\mathbf{x}, i_{0}), \nonumber
\end{align}
where the operator $\cg^{i}$ is defined as
	\begin{align*}
		\cg^{i} v^{\pi}(t, \mathbf{x}, i_0)   = 
		&\sum_{k = 1}^{N} \nabla_{k} v_i^{\pi}(t,\mathbf{x}, i_{0}) \cdot	b^i(t, \mathbf{x}, \hat{\alpha}^{k}(t,\mathbf{x},\nabla_{k}v_{k}(t,\mathbf{x},i_0),i_{0}),i_{0} ) \\
		&+\frac{1}{2} \sum_{k = 1}^{N} \operatorname{tr} ( \nabla^{2}_{k} v_i^{\pi}(t,\mathbf{x}, i_{0})(\sigma^i)^{\prime} \sigma^i(t, \mathbf{x}, i_{0} )) 
		+\sum_{j_{0} \in \mathcal{S}} q_{i_{0} j_{0}}(v_i^{\pi}(t, \mathbf{x}, j_{0})-v_i^{\pi}(t, \mathbf{x}, i_{0})).
	\end{align*}
}

Since $ f^{i} $ has quadratic growth on $ \alpha $ due to relation \eqref{eq:ff}, one obtains that $ f^{i}(t, \mathbf{x}, \hat{\alpha}^{\pi}_{i}(t, \mathbf{x}, \nabla_{i} v^{\pi}_{i}(t, \mathbf{x},i_0),i_{0}),i_{0}) $ is bounded for all $ t \in [0,T], \mathbf{x} \in \mathbb{R}^{Nd}, \hat{\alpha}^{i} \in A$. Furthermore, from relation \eqref{eq:gg}, the terminal cost $ g^{i} $ is also bounded for all $ \mathbf{x} \in \mathbb{R}^{Nd}$,  $i=1,\ldots,N $. Therefore, the truncated Nash system becomes a system of quasilinear uniformly parabolic equations with bounded coefficients. By Lemma~\ref{Lem:BdCoef}, the truncated Nash system has a unique solution $ v^{\pi} $ such that $ |v^{\pi}| $ and $|\nabla_{j}v_{i}^{\pi}|$ are bounded (with the bound possibly depending on $\pi$ and $R$). Thus, it suffices to establish a bound for the gradient of \( v^{\pi} \) independent of \( \pi \) and  \( R \), and by Lemma~\ref{Lem:UniOpt} $\hat{\alpha}^{\pi}_i$ is also bounded independent of $\pi$ and $R$. Choosing \( R \) large enough, we can deduce that \( v^{\pi} \) satisfies the system \eqref{eq:Nash}, and the corresponding optimal control $\hat{\alpha}^i$ are bounded. We do this in several steps using standard strategy as e.g. in \cite[Proposition 6.26]{ref3}. 

	\begin{proposition}\label{prop:1}
		Independent of \( \pi \) and \( R \), the value functions of the truncated Nash system $v_i^{\pi}$ are 
		bounded and satisfy
		\begin{equation}\label{eq:f}
			|v^{\pi}_{i}(t, \mathbf{x}, i_{0})-v^{\pi}_{i}(t^{\prime}, \mathbf{x}^{\prime}, i_{0})| \leq C(|t^{\prime}-t|^{\gamma / 2}+|\mathbf{x}^{\prime}-\mathbf{x}|^{\gamma}),
		\end{equation}
		for all \( (t, \mathbf{x}),(t^{\prime}, \mathbf{x}^{\prime}) \in[0, T] \times(\mathbb{R}^{d})^{N} \), $i_0 \in \cs$, for some constant \( C \) and  \( \gamma \in(0,1)\). 
\end{proposition}
\begin{proof} 
We consider the following modified diffusion processes $\tilde{\bm{X}}_{t}=(\tilde{X}_{t}^{1},...,\tilde{X}_{t}^{N})$:
	\begin{align}\label{eq:b}
		d \tilde{X}_{t}^{1} =& [B_{t}^{1}+\delta f^{1}(t, \tilde{\bm X}_{t}, Z_{t}^{1,1},I_{t-})] d t +\Sigma_{t}^{1} dW_{t}^{1}, \nonumber\\
		d \tilde{X}_{t}^{i} =& B_{t}^{i} d t+\Sigma_{t}^{1} dW_{t}^{i}, \quad i \neq 1,
	\end{align}
	where $ \delta f^{1} :[0, T] \times(\mathbb{R}^{d})^{N} \times \mathbb{R}^{d} \times \mathcal{S} \rightarrow \mathbb{R}^{d}$ satisfies for $ i_{0} \in \mathcal{S}$,
	\begin{align*}
		f^{1}(t, \mathbf{x}, \hat{\alpha}^{\pi}_{1}(t, \mathbf{x}, z,i_{0}), i_{0})=f^{1}(t, \mathbf{x}, \hat{\alpha}^{\pi}_{1}(t, \mathbf{x}, 0,i_{0}), i_{0})+\delta f^{1}(t, \mathbf{x}, z, i_{0}) \cdot z
	\end{align*}
	and $ \delta f^{1} $ can be constructed by {denoting $z^{l+}=(0, \cdots, z_{l}, \cdots, z_{d})$:
	\begin{align}\label{eq:del-f}
		(\delta {f}^{1}(t, \mathbf{x}, z, i_{0}))_{l} =\frac{f^{1}(t, \mathbf{x}, \hat{\alpha}^{\pi}_{1}(t, \mathbf{x},z^{l+},i_{0}),i_{0})-f^{1}(t, \mathbf{x}, \hat{\alpha}^{\pi}_{1}(t, \mathbf{x},z^{(l+1)+},i_{0}),i_{0})}{z_{l}},
	\end{align}}
	if $ z_l \neq 0 $ and $ 0 $ otherwise. {Since $ f $ satisfies \ref{hyp:H3} and from Lemma~\ref{Lem:UniOpt} $ \hat{\alpha} $ (and consequently $ \hat{\alpha}^{\pi} $) is Lipschitz continuous in $ z $, we have from Mean Value Theorem that:
		\begin{equation*}
			| (\delta f^{1}(t, \mathbf{x}, z, i_{0}))_l | \leq C\frac{(1+|\tilde{\al}|)|z_l|}{|z_l|},
		\end{equation*}
	where $\tilde{\al}=\hat{\al}^\pi_1(t,\mathbf{x},z^{l+}+\lambda(z^{(l+1)+}-z^{l+}),i_0)$ for some $\la \in (0,1)$. By noting that $\hat{\al}$ has linear growth in $z$ from Lemma~\ref{Lem:UniOpt} and $|z^{l+}+\lambda(z^{(l+1)+}-z^{l+})| \leq |z|$ we have 	
}
	\begin{equation}\label{eq:d}
		| \delta f^{1}(t, \mathbf{x}, z, i_{0}) | \leq C(1+|z|).
	\end{equation}
	{Also note that due to \ref{hyp:H3} and Eq.~\eqref{eq:al-ub} we have that
	\beq\label{eq:f1-0-bd}
	 f^{1}(t, \tilde{\bm{X}}_{t}, \hat{\alpha}^{\pi}_{1}(t, \tilde{\bm{X}}_{t}, 0,I_{t-}),I_{t-}) \leq C, 
	 \eeq
	where $C$ is a constant independent of   $\pi$   and $R$.
	}
	
	\vspace{0.2in}
	\noindent
\emph{Step 1: Boundedness.}
	 Owing to \ref{hyp:H1} and \ref{hyp:H5} of Hypothesis~\ref{hyp:Nash}, SDE \eqref{eq:b} has a unique solution.
	Applying It\^o's formula, we get (see \cite[Sec 3.1]{ref11})
	\begin{align*}
		dY_{t}^{1}= &\{\partial_{t} v_{1}^{\pi}(t,\tilde{\bm{X}}_{t}, I_{t-})+\cg_i v^{\pi}(t,\tilde{\bm{X}}_{t}, I_{t-}) 
		+ Z_{t}^{1,1} \cdot \delta f^{1}(t, \tilde{\bm{X}}_{t}, Z_{t}^{1,1},I_{t-}) \} dt \\
		&+\sum_{j=1}^N Z_{t}^{1,j} \cdot \Sigma_{t}^{j} dW^{j}_t 
		+ \sum_{i_{0}, j_{0} \in \mathcal{S}}(v_{1}^{\pi}(t, \tilde{\bm{X}}_{t}, j_{0})-v_{1}^{\pi}(t, \tilde{\bm{X}}_{t}, i_{0})) d M_{i_{0}j_{0}}(t),
	\end{align*}
	{where $M_{i_0 j_0}(t)$ is defined in \ref{not:Mt}.} Using the fact that $v^{\pi}$ satisfies the truncated Nash system {and noting \eqref{eq:del-f}} we get
	\begin{align} \label{eq:e}
		d Y_{t}^{1}=&-f^{1}(t, \tilde{\bm{X}}_{t}, \hat{\alpha}^{\pi}_{1}(t, \tilde{\bm{X}}_{t}, 0,I_{t-}),I_{t-}) d t\nonumber\\
		&+\sum_{i_{0}, j_{0} \in \mathcal{S}}(v_{1}^{\pi}(t, \tilde{\bm{X}}_{t}, j_{0})-v_{1}^{\pi}(t, \tilde{\bm{X}}_{t}, i_{0})) d M_{i_{0}j_{0}}(t),
		+\sum_{j=1}^{N} Z_{t}^{1, j} \cdot \Sigma_{t}^{j} dW^{j}_t \\
		Y_{T}^{1}=&g^{1}(\tilde{\bm{X}}_{T},I_{T}). \nonumber
	\end{align}
	for $ t \in [t_{0}, T] $, where $ t_{0} $ is the initial time. Recall the bound given by \eqref{eq:f1-0-bd}. In addition, by Hypothesis~\ref{hyp:H7} {$ g^{1}(\tilde{\bm{X}}_{T}, I_T) $} is bounded independent of $ \pi $ and $ R $. Recall from Lemma~\ref{Lem:BdCoef} that $ v^{\pi}(t,\mathbf{x},i_{0}) $ along with its gradient is bounded (with the bound possibly depending on $ \pi $ and $ R $). Therefore, since $ M_{i_{0}j_{0}} $ is a purely discontinuous square-integrable martingale, so is the stochastic integral with respect to $M_{i_0 j_0}$ in \eqref{eq:e}. A similar conclusion holds for the other stochastic integral thanks to Corollary~\ref{cor:sigma-bd}. Thus, the expectation of $ Y_{t_0}^{1} $ is bounded independent of $\pi$ and $R$. By initializing at any $t_{0} \in [0,T]$ and choosing $ \tilde{\bm{X}}_{t_0} $ be any $ \mathbf{x} \in {\R}^{dN} $ and $I_{t_0}$ be any $i_0 \in \cs$, we deduce that $ v^{\pi}_{1} $ is bounded, independent of $ \pi $ and $ R $. Similar arguments give us that $ v_{i}^{\pi}, i \neq 1 $ are bounded independently of $\pi$ and $R$.  

\vspace{0.2in}
\noindent
\emph{Step 2: H\"older property.}
Consider the same diffusion processes in \eqref{eq:b}. Similar to \eqref{eq:e} we have for $ i \neq 1 $
	\begin{align*}
		d Y_{t}^{i}=&-f^{i}(t, \tilde{\bm{X}}_{t}, A_{t}^{\pi,i},I_{t-}) d t
		+ Z_{t}^{i,1} \cdot \delta f^{1}(t, \tilde{\bm{X}}_{t}, Z_{t}^{1,1},I_{t-} )dt\nonumber\\
		&+\sum_{j=1}^{N} Z_{t}^{i, j} \cdot \Sigma_{t}^{j} dW^{j}_t+\sum_{i_{0}, j_{0} \in \mathcal{S}}(v_{i}^{\pi}(t, \tilde{\bm{X}}_{t}, j_{0})-v_{i}^{\pi}(t, \tilde{\bm{X}}_{t}, i_{0})) d M_{i_{0}j_{0}}(t),\\
		Y_{T}^{i} =& g^{i}(\tilde{\bm{X}}_{T},I_{T}).\nonumber
	\end{align*}
	Let $ F $ be a $ \mathcal{C}^2 $ real function. Using the generalized It\^o's formula (cf. \cite[Theorem 33]{Protter}) we have:
	$$
		F(Y^{i}_{t})- F(Y^{i}_{0}) =\int_{0^{+}}^{t} F^{\prime}(Y^{i}_{s-}) d Y^{i}_{s}
		+\frac{1}{2} \int_{0^{+}}^{t} F^{\prime \prime}(Y^{i}_{s-}) d[Y^{i}]_{s}^{c} 
		+\sum_{0<s \leq t}\{F(Y^{i}_{s})-F(Y^{i}_{s-})-F^{\prime}(Y^{i}_{s-}) \Delta Y^{i}_{s}\}.
	$$
	From the definition of optional quadratic covariations (cf. \cite[Sec 3.1]{ref11}) we have the following orthogonality relation :
	\begin{align}\label{eq:mg-ortho}
		[W_{t}^{i}, W_{t}^{j}]=0 \text{ when } i \neq j, \quad [M_{i_{0} j_{0}}(t), W_{t}^{j}]=0, \nonumber \\
		[M_{i_{0} j_{0}}(t), M_{p_{0} q_0}(t)]=0 \text{ when }(i_{0}, j_{0}) \neq (p_{0}, q_0).
	\end{align}
	Since $ [M_{i_{0}j_{0}}](s) $ is a pure jump process, $ [M_{i_{0}j_{0}}]^{c}(s)= 0 $. Consequently
	\[
	\int_{0^{+}}^{t} F^{\prime \prime}(Y^{i}_{s-}) d[Y^{i}]_{s}^{c}= \int_{0^{+}}^{t}F^{\prime \prime}(Y^{i}_{s-})\sum_{j=1}^{N}|(\Sigma_{s}^{j})^{\prime} Z_{s}^{i, j}|^{2} ds 
	\]
	Also from \eqref{eq:M}-\eqref{eq:[M],<M>}
	\begin{align*}
		\sum_{0<s \leq t} F^{\prime}(Y^{i}_{s-}) \Delta Y^{i}_{s} 
		=
		\sum_{i_{0}, j_{0} \in \mathcal{S}} \int_{0^{+}}^{t}F^{\prime}(Y^{i}_{s-})(\delta v_{i}^{\pi})_{j_0,i_0}(s, \tilde{\bm{X}}_{t}, \cdot) dM_{i_{0}j_{0}}(s)\\
		+\sum_{i_{0}, j_{0} \in \mathcal{S}} \int_{0^{+}}^{t}F^{\prime}(Y^{i}_{s-})(\delta v_{i}^{\pi})_{j_0,i_0}(s, \tilde{\bm{X}}_{t}, \cdot)q_{i_{0}j_{0}}\mathds{1}_{\{I_{s-}=i_{0}\}} ds.
	\end{align*}
	Similarly
	\begin{align*}
		\sum_{0<s \leq t}(F(Y^{i}_{s})-F(Y^{i}_{s-}))=
		\sum_{i_{0}, j_{0} \in \mathcal{S}} \int_{0^{+}}^{t} ( F(Y^{i}_s) - F(Y^{i}_{s-}) )(dM_{i_{0}j_{0}}(s)+q_{i_{0}j_{0}}\mathds{1}_{\{I_{s-}=i_{0}\}} ds).
	\end{align*}
	Now use the notation for $i \neq 1$
	\begin{align*}
		\mathcal{L}^{i}_{t}=\frac{1}{2}F^{\prime \prime}(Y^{i}_{t-})\sum_{j=1}^{N}|(\Sigma_{t}^{j})^{\prime} Z_{t}^{i, j}|^{2}
		+F^{\prime}(Y^{i}_{t-})\{-f^{i}(t, \tilde{\bm{X}}_{t}, A_{t}^{\pi,i},I_{t-})+ Z_{t}^{i,1} \delta f^{1}(t, \tilde{\bm{X}}_{t}, Z_{t}^{1,1},I_{t-} ) \\
		-\sum_{i_{0}, j_{0} \in \mathcal{S}}(\delta v_{i}^{\pi})_{j_0,i_0}(t, \tilde{\bm{X}}_{t}, \cdot)q_{i_{0}j_{0}}\mathds{1}_{\{I_{t-}=i_{0}\}}\} 
		+\sum_{i_{0}, j_{0} \in \mathcal{S}} (F(Y^{i}_{t})-F(Y^{i}_{t-}))q_{i_{0}j_{0}}\mathds{1}_{\{I_{t-}=i_{0}\}}  ,
	\end{align*}
	and for $i=1$
	\begin{align*}
		\mathcal{L}_{t}^{1}&= \frac{1}{2}F^{\prime \prime}(Y^{1}_{t-})\sum_{j=1}^{N}|(\Sigma_{t}^{j})^{\prime} Z_{t}^{1, j}|^{2}
		-F^{\prime}(Y^{1}_{t-})\{f^{1}(t, \tilde{\bm{X}}_{t}, \hat{\alpha}^{\pi}_{1}(t, \tilde{\bm{X}}_{t}, 0,I_{t-}),I_{t-})\vphantom{\sum_{i_{0}, j_{0} \in \mathcal{S}}}\\
		&+\sum_{i_{0}, j_{0} \in \mathcal{S}}(\delta v_{1}^{\pi})_{j_0,i_0}(t, \tilde{\bm{X}}_{t}, \cdot)q_{i_{0}j_{0}}\mathds{1}_{\{I_{t-}=i_{0}\}}\}
		+\sum_{i_{0}, j_{0} \in \mathcal{S}} (F(Y^{1}_{t})-F(Y^{1}_{t-}))q_{i_{0}j_{0}}\mathds{1}_{\{I_{t-}=i_{0}\}}.  
	\end{align*}
	Thus the It\^o's formula becomes: for all $i$
	\begin{multline}\label{eq:Ito-F}
		F(Y^{i}_{t})- F(Y^{i}_{0}) =
		\int_{0^{+}}^{t}\mathcal{L}^{i}_{s} ds + \int_{0^{+}}^{t}F^{\prime}(Y^{i}_{s-})\sum_{j=1}^{N} Z_{s}^{i, j} \cdot \Sigma_{s}^{j} dW^{j}_{s} \\
		+\sum_{i_{0}, j_{0} \in \mathcal{S}} \int_{0^{+}}^{t}(F(Y^{i}_{s})-F(Y^{i}_{s-}))dM_{i_{0}j_{0}}(s).
	\end{multline}
	We now apply It\^o's formula to $ F(Y_{t}^{1})=(Y_{t}^{1})^{2} $. In the following, we use $C$ to denote a generic constant that may change from step to step. Since from \emph{Step 1} $ v^{\pi}(t,\tilde{\bm{X}}_{t}, i_{0}) $ can be bounded independent of $ \pi $ and $ R $, we have by noting \eqref{eq:f1-0-bd}
	\begin{align*}
		d((Y_{t}^{1})^{2}) \geq ( {-C} -C Y_{t-}^{1}+\sum_{j=1}^{N}|(\Sigma_{t}^{j})^{\prime} Z_{t}^{1, j}|^{2}-
		\sum_{i_{0}, j_{0} \in \mathcal{S}} (Y_{t-}^{1})^{2} q_{i_{0}j_{0}}\mathds{1}_{\{I_{t-}=i_{0}\}} )dt\\
		+2Y_{t-}^{1}\sum_{j=1}^{N} Z_{t}^{1, j} \cdot \Sigma_{t}^{j} dW^{j}_{t}
		+\sum_{i_{0}, j_{0} \in \mathcal{S}} ((Y^{1}_{t})^{2}-(Y^{1}_{t-})^{2})dM_{i_{0}j_{0}}(t).
	\end{align*}
	{Note that from Hypothesis~\ref{hyp:H6}, $|(\Sigma_t^1)' Z_t^{1,1}|^2 \geq \nu_1 |Z_t^{1,1}|^2$.}
	Since $ Y_{t}^{1} $ is bounded independent of $ \pi $ and $ R $, {by taking expectation} we conclude that for a stopping time $ \tau \in [t_0,T]$
	\begin{equation}\label{eq:c}
		\mathbb{E}\lc\int_{\tau}^{T}|Z_{s}^{1,1}|^{2} d s \mid \mathcal{F}_{\tau}\rc \leq C,
	\end{equation}
	which means that the martingale $\{ {\int_{t_{0}}^{t} Z_{s}^{1,1} dW_{s}^{1}}; {t_{0} \leq t \leq T} \}$ is of Bounded Mean Oscillation (BMO).
	For $ i \neq 1 $, apply It\^o's formula to $ F(Y_{t}^{i})=\exp(\eta Y_{t}^{i}) $. From \eqref{eq:ff} and \eqref{eq:al-ub} $f^{i}$ has quadratic growth in $z$. {Furthermore since $\delta f^{1}$ has linear growth in $z$ (cf. \eqref{eq:d}) we have $|Z_t^{i,1} \der f^1(t, \tilde{\bm{X}}_t, Z_t^{1,1}, I_{t-})| \gtrsim - \lln \lla Z_t^{i,1}, Z_t^{1,1} \rra \rrn \gtrsim -(|Z_t^{i,1}|^2 + |Z_t^{1,1}|^2 )$. Using these in \eqref{eq:Ito-F} we have}
	\begin{align*}
		d[\exp (\eta Y_{t}^{i})] \geq \exp (\eta Y_{t-}^{i})[\frac{\eta^{2}}{2} \sum_{j=1}^{N}|(\Sigma_{t}^{j})^{\prime} Z_{t}^{i, j}|^{2}
		-C \eta(1+|Z_{t}^{i, i}|^{2}+|Z_{t}^{i, 1}|^{2}+|Z_{t}^{1,1}|^{2})-C] d t \\
		+\eta \exp (\eta Y_{t-}^{i})\sum_{j=1}^{N} Z_{t}^{i, j} \cdot \Sigma_{t}^{j} dW^{j}_{t}
		+\sum_{i_{0}, j_{0} \in \mathcal{S}} (\exp(\eta Y^{i}_{t})-\exp(\eta Y^{i}_{t-}))dM_{i_{0}j_{0}}(t),
	\end{align*}
	where $C$ is a constant {independent of $\pi$ and $R$}. 
	\\
	{Recalling \ref{hyp:H6} of Hypothesis~\ref{hyp:Nash}, by choosing \( \eta \) large enough we have for a new constant \( c \) :
		\begin{align*}
			\frac{\eta^{2}}{2} \sum_{j=1}^{N}|(\Sigma_{t}^{j})^{\prime} Z_{t}^{i, j}|^{2}
			-C \eta(|Z_{t}^{i, i}|^{2}+|Z_{t}^{i, 1}|^{2}) &\geq \frac{\nu_1\eta^{2}}{2} \sum_{j=1}^{N}|Z_{t}^{i, j}|^{2}
			-C \eta(|Z_{t}^{i, i}|^{2}+|Z_{t}^{i, 1}|^{2})\\
			&\geq c\sum_{j=1}^{N}|Z_{t}^{i, j}|^{2}.
		\end{align*}
	Therefore
	\begin{align*}
		d[\exp (\eta Y_{t}^{i})] &\geq \exp (\eta Y_{t-}^{i})\lc c \sum_{j=1}^{N}|Z_{t}^{i, j}|^{2}
		-C\eta |Z_{t}^{1,1}|^{2}-C\eta \rc d t \\
		&+\eta \exp (\eta Y_{t-}^{i})\sum_{j=1}^{N} Z_{t}^{i, j} \cdot \Sigma_{t}^{j} dW^{j}_{t}
		+\sum_{i_{0}, j_{0} \in \mathcal{S}} (\exp(\eta Y^{i}_{t})-\exp(\eta Y^{i}_{t-}))dM_{i_{0}j_{0}}(t).
	\end{align*}
	Fixing the constant $\eta$ and} recalling \eqref{eq:c} together with the fact that $Y_{t}^{i}$ along with its gradients $Z_t^{i,j}$ are bounded (where the bound for $Z$ may depend on $\pi$ and $R$), we conclude that for any stopping time $ \tau \in [t_0, T] $:
	\[\mathbb{E}[\int_{\tau}^{T}|Z_{s}^{i,i}|^{2} d s \mid \mathcal{F}_{\tau}] \leq C.\]
	Therefore, all the martingales \( (\int_{t_{0}}^{t} Z_{s}^{i, i} dW_{s}^{i})_{t_{0} \leq t \leq T} \), for \( i=1, \cdots, N \), are BMO and, most importantly, their \( \mathrm{BMO} \) norms can be bounded independently of $R$, $\pi$  and the initial condition of \( \tilde{\bm{X}}_{t} \).
	We now return to \eqref{eq:b}. Letting:
	\[
	K_{t}^{1}=B_{t}^{1}+\delta f^{1}(t, \tilde{\bm{X}}_{t}, Z_{t}^{1,1},I_{t-}), \quad K_{t}^{i}=B_{t}^{i}, \quad i \neq 1,
	\]
	for \( t \in [t_{0}, T] \), we observe from \eqref{eq:d}, Hypothesis~(H1) and \eqref{eq:al-ub} that:
	\[
	|K_{t}^{i}| \leq C(1+|Z_{t}^{i, i}|),
	\]
	from which we deduce that the martingale \( (\sum_{i=1}^{N} \int_{t_{0}}^{t} K_{s}^{i} \cdot dW_{s}^{i})_{t_{0} \leq t \leq T} \) is also BMO and that its BMO norm is less than \( C \), the constant \( C \) being allowed to increase from line to line as long as it remains independent of \( R, \pi \) and the initial condition of \( \tilde{\bm{X}}_{t} \). {The BMO property implies that the Girsanov density:
	\begin{align*}
		\mathcal{E}_{t}=\exp (-\sum_{i=1}^{N} \int_{t_{0}}^{t} K_{s}^{i} (\Sigma_{s}^{i})^{-1} dW_{s}^{i}
		-\frac{1}{2} \sum_{i=1}^{N} \int_{t_{0}}^{t}|(\Sigma_{s}^{i})^{-1 \prime} K_{s}^{i}|^{2} d s), \quad t \in[t_{0}, T]
	\end{align*}
	satisfies:
	\(
	\mathbb{E}[(\mathcal{E}_{T})^{r}] \leq C,
	\)
	for any exponent \( r>1 \). Letting \( \mathbb{Q}=\mathcal{E}_{T} \cdot \mathbb{P} \),
	we deduce that, for any event \( E \in \mathcal{F} \),
	\( \mathbb{Q}(E) \leq C^{1 / r} \mathbb{P}(E)^{r /(r-1)}\), that is $\mathbb{P}(E) \geq C^{-(r-1) / r^{2}} \mathbb{Q}(E)^{(r-1) / r}$. 
	The parameters in the above estimate are independent of \( \pi, R \) and the initial conditions. The above lower bound says that we can control from below the probability that the process \( \tilde{\bm{X}} \) hits a given Borel subset in \( \mathbb{R}^{N d} \) in terms of the probability that a Brownian motion in \( \mathbb{R}^{N d} \) hits the same Borel subset. Thus using the representation formula for \( v^{\pi}_{ 1}(t, \tilde{\bm{X}}_t, i_0) \) (similar to  \eqref{eq:e}), the fact that \( g^{1}, f^{1} \) and \( \hat{\alpha}(\cdot, \cdot, 0, \cdot) \) {are bounded and smooth in \( x \)}, and the Krylov and Safonov estimates \cite{delarue-holder} for the Hölder regularity of the solutions of second-order parabolic PDEs with measurable coefficients, we have that \( v^{\pi}_{ 1} \) is Hölder continuous on \( [0, T] \times(\mathbb{R}^{d})^{N} \). A similar argument holds for \( v^{\pi}_{i} \) with \( i \neq 1 \). Thus we have for all $i$
	\begin{equation*}
		|v^{\pi}_{i}(t, \mathbf{x}, i_{0})-v^{\pi}_{i}(t^{\prime}, \mathbf{x}^{\prime}, i_{0})| \leq C(|t^{\prime}-t|^{\gamma / 2}+|\mathbf{x}^{\prime}-\mathbf{x}|^{\gamma}),
	\end{equation*}
	for all \( (t, \mathbf{x}),(t^{\prime}, \mathbf{x}^{\prime}) \in[0, T] \times(\mathbb{R}^{d})^{N} \), for some constant \( C \) as above and for some exponent \( \gamma \in(0,1), \gamma \) being independent of \( \pi \) and \( R \). }
\end{proof}
We are now ready to show our desired result as discussed before in Section~\ref{sec:strategy}. The proof relies on Proposition~\ref{prop:1} and an additional lemma~\ref{pro:intZ}. 
Now we consider another diffusion process $\tilde{\bm{X}}_{t}=(\tilde{X}_{t}^1,...,\tilde{X}_{t}^N)$. Redefine for $1\leq i \leq N$:
\begin{equation}\label{eq:b'}
	\tilde{X}_{t}^{i}=\Sigma_{t}^{i} dW_{t}^{i}.
\end{equation}
{
By It\^o's formula
\begin{align}\label{eq:Y-ito}
	dY_t^i &=-\left(f^{i}(t, \tilde{\bm{X}}_{t}, A_{t}^{\pi, i},I_{t-})+ \sum_{j=1}^{N} Z_{t}^{i, j} \cdot B_{t}^{j}\right) dt+\sum_{j=1}^{N} Z_{t}^{i, j} \cdot \Sigma_{t}^{j} dW^{j}_t \nonumber\\
	&\hspace{1in}+\sum_{i_{0}, j_{0} \in \mathcal{S}}(v_{i}^{\pi}(t, \tilde{\bm{X}}_{t}, j_{0})-v_{i}^{\pi}(t, \tilde{\bm{X}}_{t}, i_{0})) d M_{i_{0}j_{0}}(t), \nonumber\\
	Y_{T}^{i} &= g^{i}(\tilde{\bm{X}}_{T},I_{T}).
\end{align}
}
Similar to before, we have the following It\^o's formula 
{
\begin{align}\label{eq:F-Y-Ito}
	F(t, Y^{i}_{t})- F(t_0, Y^{i}_{t_0}) = \int_{t_0^{+}}^{t}\frac{\partial}{\partial y} F(s, Y^{i}_{s-})\sum_{j=1}^{N} Z_{s}^{i, j} \cdot \Sigma_{s}^{j}dW^{j}_{s}+
	\int_{t_0^{+}}^{t}\mathcal{L}_{s}^{i} ds \nonumber \\
	+\sum_{i_{0}, j_{0} \in \mathcal{S}} \int_{t_0^{+}}^{t}(F(s,Y^{i}_{s})-F(s,Y^{i}_{s-}))dM_{i_{0}j_{0}}(s),
\end{align}
}
{
where
\begin{align*}
	\mathcal{L}_{t}^{i}=\frac{\partial}{\partial t}F(t,Y^i_{t-})-\frac{\partial}{\partial y}F(t,Y^{i}_{t-})G_{t}^{i}+\frac{1}{2}\frac{\partial^2}{\partial y^2}F(t,Y^{i}_{t-})\sum_{j=1}^{N}|(\Sigma_{t}^{j})^{\prime} Z_{t}^{i, j}|^{2} \\
	+
	\sum_{i_{0}, j_{0} \in \mathcal{S}} (F(t,Y^{i}_{t})-F(t,Y^{i}_{t-}))q_{i_{0}j_{0}}\mathds{1}_{\{I_{t-}=i_{0}\}}  ,
\end{align*}
}
and
\begin{align*}
	G_{t}^{i}= f^{i}(t, \tilde{\bm{X}}_{t}, A_{t}^{\pi, i},I_{t-})+ \sum_{j=1}^{N} Z_{t}^{i, j} \cdot B_{t}^{j}+
	\sum_{i_{0}, j_{0} \in \mathcal{S}}(v_{i}^{\pi}(t, \tilde{\bm{X}}_{t}, j_{0})-v_{i}^{\pi}(t, \tilde{\bm{X}}_{t}, i_{0}))q_{i_{0}j_{0}}\mathds{1}_{\{I_{t-}=i_{0}\}}.
\end{align*}
Denote $G_{t}=(G_{t}^{i})_{i=1,...,N}$. Since { $v_i^{\pi}$ are bounded}, $f^{i}$ has quadratic growth in $\alpha$, {$b^i$ has linear growth in $\al$}, and $\alpha$ has linear growth in $z$, we have using Cauchy–Schwarz inequality
\beq\label{eq:G-Z-bnd}
|\bm{G}_{t}|={\lp \sum_{i=1}^{N}|G_{t}^{i}|^2 \rp }^{\sfrac{1}{2}} \leq C(1+|\bm{Z}_{t}|^{2}),
\eeq
where the constant $ C $ is independent of $ \pi $ and $ R $. 

We now prove a result which will be useful in the sequel to bound $\bm{Z}$.
\begin{lemma}\label{pro:intZ}
	Consider the diffusion processes \eqref{eq:b'} and any $t_0 \in [0,T]$. Then there exists a constant \( \varrho>0 \), independent of \( R, \pi \) and the initial condition of \( \tilde{\bm{X}} \), such that:
	\[
	\mathbb{E}\lc \int_{t_{0}}^{\tau} \frac{|\bm{Z}_{s}|^{2}}{(s-t_{0})^{\beta}} d s\rc \leq C_{\varrho},
	\]
	where \( \beta=\gamma / 4 \) and \( \tau \) is the first hitting time:
	\begin{equation}\label{tau}
		\tau=\inf \{t \geq t_{0}:|\tilde{\bm{X}}_{t}-\tilde{\bm{X}}_{t_{0}}| \geq \varrho\} \wedge(t_{0}+\varrho^{2}) \wedge T ,
	\end{equation}
	and the constant \( C_{\varrho} \) being also independent of \( \pi, R \) and the initial condition of the process \( \tilde{\bm{X}} \).
\end{lemma}
\begin{proof} Denote the jump times (after $t_0$) of the Markov chain $I_{s}$ as $\Gamma=(\Gamma_{1}, ..., \Gamma_{J}, ...)$, where $J$ is the last jump time before $\tau$, that is $\Gamma_{J} \leq \tau < \Gamma_{J+1} $. In addition define the stopping times $ \tau_{i}:=\tau \wedge \Gamma_{i} $ for $ i \geq 1 $ and $ \tau_{0}:=t_{0} $. Note that by the definition of $\tau$ in \eqref{tau} we have that $\tau_i \leq t_0 + \varrho^2$. In order to prove the proposition, we first show that 
	\[
	\mathbb{E}\lc\int_{t_{0}}^{\tau_{1}} \frac{|\bm{Z}_{s}|^{2}}{(s-t_{0})^{\beta}} d s\rc \leq C_{\varrho}.
	\] 
	For a given \( \varepsilon>0 \), we consider the process \( F(t, \bm{Y}_{t})=(|\bm{Y}_{t}-\bm{Y}_{t_{0}}|^{2} /(\varepsilon+(t-t_{0})^{\beta}))_{t_{0} \leq t \leq \tau_{1}} \), where \( \bm{Y}=(Y^{1}, \cdots, Y^{N}) \). By the It\^o's formula \eqref{eq:F-Y-Ito}, we obtain: 
	{
	\begin{align}\label{eq:F(t,Yt)}
		&\frac{|\bm{Y}_{\tau_{1}}-\bm{Y}_{t_{0}}|^{2}}{\varepsilon+(\tau_{1}-t_{0})^{\beta}}=\sum_{i=1}^{N}F(\tau_1, Y_{\tau_1}^{i})\nonumber\\
		&=-2\int_{t_{0}}^{\tau_{1}} \frac{(\bm{Y}_{s-}-\bm{Y}_{t_{0}}) \cdot \bm{G}_{s}}{\varepsilon+(s-t_{0})^{\beta}} d s
		-{\beta}\int_{t_{0}}^{\tau_{1}} \frac{|\bm{Y}_{s-}-\bm{Y}_{t_{0}}|^{2}}{(s-t_{0})^{1-\beta}(\varepsilon+(s-t_{0})^{\beta})^{2}} d s\nonumber\\
		&+\sum_{i=1}^{N}\int_{t_{0}}^{\tau_{1}} \frac{\sum_{j=1}^{N}|(\Sigma_{t}^{j})^{\prime} Z_{t}^{i, j}|^{2}}{\varepsilon+(s-t_{0})^{\beta}} d s
		+\sum_{i=1}^{N}\int_{t_{0}}^{\tau_{1}} \frac{1}{(s-t_{0})^{\beta}}\sum_{i_{0}, j_{0} \in \mathcal{S}}  (F(s, Y_s^i) - F(s,Y_{s-}^i)) q_{i_{0}j_{0}}\mathds{1}_{\{I_{s-}=i_{0}\}}d s\nonumber\\
		&+\sum_{i=1}^{N}\int_{t_{0}}^{\tau_{1}}\frac{\partial}{\partial y}F(s,Y^{i}_{s-})\sum_{j=1}^{N} Z_{s}^{i, j} \cdot \Sigma_{s}^{j} dW^{j}_{s}
		+\sum_{i=1}^{N}\sum_{i_{0}, j_{0} \in \mathcal{S}} \int_{t_{0}}^{\tau_{1}}(F(s,Y^{i}_{s})-F(s,Y^{i}_{s-}))dM_{i_{0}j_{0}}(s).
	\end{align}
Note that 
\begin{equation}\label{eq:Sigma'Z}
	\sum_{i=1}^N \sum_{j=1}^N\left|\left(\Sigma_t^j\right)^{\prime} Z_t^{i, j}\right|^2  =\sum_{i=1}^N \sum_{j=1}^N\left(Z_t^{i, j}\right)^{\prime} \Sigma_t^j\left(\Sigma_t^j\right)^{\prime} Z_t^{i j} \geq \sum_{i=1}^N \sum_{j=1}^N \nu_1\left|Z_t^{i, j}\right|^2,
\end{equation}
where the last step follows from Hypothesis~\ref{hyp:H6}. By taking expectation of \eqref{eq:F(t,Yt)} and using \eqref{eq:Sigma'Z} we get that there exists a constant $C$ such that}
	\begin{align*}
		\mathbb{E}\lc\int_{t_{0}}^{\tau_{1}} \frac{|\bm{Z}_{s}|^{2}}{\varepsilon+(s-t_{0})^{\beta}} d s\rc &\leq  C\lp\mathbb{E}\lc\frac{|\bm{Y}_{\tau_{1}}-\bm{Y}_{t_{0}}|^{2}}{\varepsilon+(\tau_{1}-t_{0})^{\beta}}\rc
		+2 \mathbb{E}\lc\lln\int_{t_{0}}^{\tau_{1}} \frac{(\bm{Y}_{s-}-\bm{Y}_{t_{0}}) \cdot \bm{G}_{s}}{\varepsilon+(s-t_{0})^{\beta}} d s\rrn\rc \right.\\
		&+\mathbb{E}\lc\lln\sum_{i=1}^{N}\int_{t_{0}}^{\tau_{1}} \frac{1}{(s-t_{0})^{\beta}}\sum_{i_{0}, j_{0} \in \mathcal{S}} (F(s,Y^{i}_{s})-F(s,Y^{i}_{s-})) q_{i_{0}j_{0}}\mathds{1}_{\{I_{s-}=i_{0}\}}   d s\rrn\rc \\
		&\left.+{\beta}\mathbb{E}\lc\int_{t_{0}}^{\tau_{1}} \frac{|\bm{Y}_{s-}-\bm{Y}_{t_{0}}|^{2}}{(s-t_{0})^{1-\beta}(\varepsilon+(s-t_{0})^{\beta})^{2}} d s\rc\rp .
	\end{align*}
	Notice $Y_{s}^{i}$ is bounded 
	{and that $2\beta <1$. Therefore}
	\beq\label{eq:term-3-bnd}
	\mathbb{E}\lc\lln\int_{t_{0}}^{\tau_{1}} \frac{1}{(s-t_{0})^{\beta}}\sum_{i_{0}, j_{0} \in \mathcal{S}} (F(s,Y^{i}_{s})-F(s,Y^{i}_{s-}))q_{i_{0}j_{0}}\mathds{1}_{\{I_{s-}=i_{0}\}}   d s\rrn\rc,
	\eeq
	is also bounded.
	Notice that from Proposition~\ref{prop:1} the Hölder property of \( v^{\pi}=(v^{\pi}_{1}, \cdots, v^{\pi}_{N}) \) holds for all \( t \in [t_{0}, \tau_{1}]\), and consequently recalling the definition of $\tau$ in \eqref{tau}:
	\begin{equation}\label{eq:Y-bnd-1}
		|\bm{Y}_{t}-\bm{Y}_{t_{0}}| \leq C \varrho^{\gamma} .
	\end{equation}
	Using \eqref{eq:G-Z-bnd} and \eqref{eq:Y-bnd-1} we have
	\beq\label{eq:term-2-bnd}
	\mathbb{E}\lc\lln\int_{t_{0}}^{\tau_{1}} \frac{(\bm{Y}_{s-}-\bm{Y}_{t_{0}}) \cdot \bm{G}_{s}}{\varepsilon+(s-t_{0})^{\beta}} d s\rc\rrn \leq C \varrho^{\gamma} \mathbb{E}\lc\int_{t_{0}}^{\tau_{1}} \frac{1+|\bm{Z}_{s}|^{2}}{\varepsilon+(s-t_{0})^{\beta}} d s\rc.
	\eeq
	Therefore by using \eqref{eq:term-3-bnd}, \eqref{eq:term-2-bnd} and the H\"older property of $Y_t$ via \eqref{eq:f} we get
	\begin{align*}
		\mathbb{E}\lc\int_{t_{0}}^{\tau_{1}} \frac{|\bm{Z}_{s}|^{2}}{\varepsilon+(s-t_{0})^{\beta}} d s\rc \leq& C\lp\mathbb{E}\lc\frac{(\tau_{1}-t_{0})^{\gamma}+|\tilde{\bm{X}}_{\tau_{1}}-\tilde{\bm{X}}_{t_{0}}|^{2 \gamma}}{\varepsilon+(\tau_{1}-t_{0})^{\beta}}\rc 
		+\varrho^{\gamma} \mathbb{E}\lc\int_{t_{0}}^{\tau_{1}} \frac{1+|\bm{Z}_{s}|^{2}}{\varepsilon+(s-t_{0})^{\beta}} d s\rc\right.\\
		&\left.+\mathbb{E}\lc\int_{t_{0}}^{\tau_{1}} \frac{(s-t_{0})^{\gamma}+|\tilde{\bm{X}}_{s}-\tilde{\bm{X}}_{t_{0}}|^{2 \gamma}}{(s-t_{0})^{1-\beta}(\varepsilon+(s-t_{0})^{\beta})^{2}} d s\rc\rp .
	\end{align*}
	Plugging in \( \beta=\gamma / 4 \), and using the elementary inequality $(a+b)^r \leq a^r + b^r$, when $0<r< 1$ with $a, b > 0$ we have
	\begin{align}\label{eq:t0-tau1-Z2}
		\mathbb{E}\lc\int_{t_{0}}^{\tau_{1}} \frac{|\bm{Z}_{s}|^{2}}{\varepsilon+(s-t_{0})^{\beta}} d s\rc  \leq C(1+\mathbb{E}\lc\frac{|\tilde{\bm{X}}_{\tau_{1}}-\tilde{\bm{X}}_{t_{0}}|^{2}}{\varepsilon^{1 / \gamma}+(\tau_{1}-t_{0})^{1 / 4}}\rc^{\gamma} 
		+\varrho^{\gamma} \mathbb{E}\lc\int_{t_{0}}^{\tau_{1}} \frac{1+|\bm{Z}_{s}|^{2}}{\varepsilon+(s-t_{0})^{\beta}} d s\rc \nonumber \\
		+\mathbb{E}\lc\int_{t_{0}}^{t_{0}+\varrho^{2}} \frac{1}{(s-t_{0})^{1-\gamma / 4}} \mathbb{E}\lc\frac{|\tilde{\bm{X}}_{s \wedge \tau_{1}}-\tilde{\bm{X}}_{t_{0}}|^{2}}{\varepsilon^{2 / \gamma}+(s \wedge \tau_{1}-t_{0})^{1 / 2}}\rc^{\gamma} d s\rc) .
	\end{align}
	{It is readily verified from It\^o's formula that:
	\[
	\mathbb{E}\lc\frac{|\tilde{\bm{X}}_{s \wedge \tau_{1}}-\tilde{\bm{X}}_{t_{0}}|^{2}}{\varepsilon^{2 / \gamma}+(s \wedge \tau_{1}-t_{0})^{1 / 2}}\rc \leq C,
	\]
	with \( C \) independent of \( \varepsilon \) and of the initial condition of \( \tilde{\bm{X}}_{t} \).} {Indeed, from It\^o's formula we have 
	\begin{align*}
		\frac{|\tilde{X}_{s \wedge \tau_{1}}^i-\tilde{X}_{t_{0}}^i|^{2}}{\varepsilon^{2 / \gamma}+(s \wedge \tau_{1}-t_{0})^{1 / 2}}=-\int_{t_0}^{s \wedge \tau_1} \frac{|\tilde{X}_{u}^i-\tilde{X}_{t_0}^i|^2}{2(u-t_0)^{1/2}(\varepsilon^{2 / \gamma}+(u-t_{0})^{1 / 2})^2} du  \\
		+\int_{t_0}^{s \wedge \tau_1}\frac{2\left(\tilde{X}_{u}^i-\tilde{X}_{t_0}^i\right)\Sigma_u^i}{\varepsilon^{2 / \gamma} +(u-t_{0})^{1 / 2}}dW_{u}^i+\int_{t_0}^{s \wedge \tau_1} \frac{\left|\Sigma_u^i\right|^2}{\varepsilon^{2 / \gamma}+(u-t_{0})^{1 / 2}} du.
	\end{align*}
	Taking expectation and dropping the negative term on the right hand side we have
	\begin{align}\label{eq:bnd-ito-1}
		\mathbb{E}\left[\frac{|\tilde{X}_{s \wedge \tau_{1}}^i-\tilde{X}_{t_{0}}^i|^{2}}{\varepsilon^{2 / \gamma}+(s \wedge \tau_{1}-t_{0})^{1 / 2}}\right] 
		\leq \mathbb{E}\left[\int_{t_0}^{{t_0 + \varrho^2}} \frac{\left|\Sigma_u^i\right|^2}{\varepsilon^{2 / \gamma}+(u-t_{0})^{1 / 2}} du\right]\leq C,
	\end{align}
	where we have used the bound for $\tau_1$ and the fact that $|\Sigma_u^i|^2$ is bounded due to Corollary~\ref{cor:sigma-bd}. Similarly one can show that 
	\begin{align}\label{eq:bnd-ito-2}
		\mathbb{E}\lc\frac{|\tilde{\bm{X}}_{s \wedge \tau_{1}}-\tilde{\bm{X}}_{t_{0}}|^{2}}{\varepsilon^{1 / \gamma}+( \tau_{1}-t_{0})^{1 / 4}}\rc \leq C.
	\end{align}
}
	Finally using \eqref{eq:bnd-ito-1}-\eqref{eq:bnd-ito-2} in \eqref{eq:t0-tau1-Z2}, and by choosing \( \varrho \) small enough and letting \( \varepsilon \searrow 0 \), we have that
	\beq\label{eq:int-C-rhobnd-1}
	\mathbb{E}[\int_{t_{0}}^{\tau_{1}} \frac{|\bm{Z}_{s}|^{2}}{(s-t_{0})^{\beta}} d s] < C_{\varrho}.
	\eeq
	Similarly, by the Hölder property of $v^{\pi}$ for $t \in [\tau_{k},\tau_{k+1}]$ one can show that for all $k \geq 1$
	\beq\label{eq:int-C-rhobnd-2}
	\mathbb{E}[\int_{\tau_{k}}^{\tau_{k+1}} \frac{|\bm{Z}_{s}|^{2}}{(s-t_{0})^{\beta}} d s] < C_{\varrho},
	\eeq
	where $C_{\varrho}$ in \eqref{eq:int-C-rhobnd-1}-\eqref{eq:int-C-rhobnd-2} are same.
	Then
	\begin{align*}
		\mathbb{E}\lc\int_{t_{0}}^{\tau} \frac{|\bm{Z}_{s}|^{2}}{(s-\tau_{k})^{\beta}} d s\rc=\mathbb{E}\lc\sum_{k=0}^{\infty}\int_{\tau_{k}}^{\tau_{k+1}} \frac{|\bm{Z}_{s}|^{2}}{(s-\tau_{k})^{\beta}} d s\rc 
		&= \mathbb{E}\lc\sum_{k=0}^{\infty}\mathbb{E}\lc \left. \int_{\tau_{k}}^{\tau_{k+1}} \frac{|\bm{Z}_{s}|^{2}}{(s-\tau_{k})^{\beta}} d s \right| \Gamma\rc\rc \\
		&\leq  \mathbb{E}[J]\cdot \max_k\mathbb{E}\lc\int_{\tau_{k}}^{\tau_{k+1}} \frac{|\bm{Z}_{s}|^{2}}{(s-\tau_{k})^{\beta}} d s\rc.
	\end{align*}
	Since the number of jumps in the finite horizon of a CTMC with finite rates is bounded in expectation, combined with \eqref{eq:int-C-rhobnd-1}-\eqref{eq:int-C-rhobnd-2}, we complete the proof.
\end{proof}
Using Lemma~\ref{pro:intZ} we are able to bound $\nabla_{\mathbf{x}} v^{\pi}$.
\begin{proposition}\label{prop:4}
	{Independent of $\pi$ and $R$, the gradient  $\nabla_{\mathbf{x}} v^{\pi}(t, \mathbf{x}, i_0)$ is bounded for all $(t, \mathbf{x}) \in [0,T] \times (\R^{d})^N$ and $i_0 \in \cs$.}
\end{proposition}
\begin{proof}
	Consider again the diffusion processes \eqref{eq:b'} and recall the BSDE for $Y$ in \eqref{eq:Y-ito}. We fix \( \varrho \) as in Lemma~\ref{pro:intZ} and consider for some \( \mathbf{x}^{0} \in (\mathbb{R}^{d})^{N} \) a smooth cut-off function \( \eta:(\mathbb{R}^{d})^{N} \rightarrow \mathbb{R} \) equal to 1 on the open ball \( B_{\varrho / 2}(\mathbf{x}^{0}) \) and vanishing outside the open ball \( B_{\varrho}(\mathbf{x}^{0}) \). We then expand \( (Y_{t}^{i} \eta(\tilde{\bm{X}}_{t}))_{t_{0} \leq t \leq \varsigma} \) by Itô's formula. Let
	\beq\label{eq:zeta}
	 \varsigma=\inf \{t \geq t_{0}:|\tilde{\bm{X}}_{t}-\tilde{\bm{X}}_{t_{0}}| \geq \varrho\} \wedge T 
	 \eeq
	  we get 
		{
	\begin{align*}
		Y_\varsigma^i \eta(\tilde{\bm{X}}_\varsigma)-Y_{t_0}^i\eta(\tilde{\bm{X}}_{t_0})=&\int_{t_0}^\varsigma\eta(\tilde{\bm{X}}_s)dY_{s}^i+\sum_{j=1}^N\left(\int_{t_0}^\varsigma Y_{s-}^i\nabla_j\eta(\tilde{\bm{X}}_s)d\tilde{X}_s^j \right.\\
		&\left.+ \int_{t_0}^\varsigma \nabla_j \eta(\tilde{\bm{X}}_s) d[\tilde{X}^j, Y^i]_s^c+\frac{1}{2}\int_{t_0}^\varsigma \operatorname{tr}\lp Y_{s-}^i\nabla^2_{j,j}\eta(\tilde{\bm{X}}_s)d[\tilde{X}^j]_s\rp\right).
	\end{align*}
Thus}
	\begin{align}\label{eq:Y-eta}
		Y_{t_{0}}^{i} \eta(\tilde{\bm{X}}_{t_{0}})  &=\mathbb{E}[g^{i}(\tilde{\bm{X}}_{\varsigma},I_{\varsigma}) \eta(\tilde{\bm{X}}_{\varsigma})+\int_{t_{0}}^{\varsigma} \Psi^{i}(s, \tilde{\bm{X}}_{s}, Z^i_{s},I_{s-}) d s \mid \mathcal{F}_{t_{0}}] \nonumber\\
		&=\mathbb{E}[g^{i}(\tilde{\bm{X}}_{\varsigma},I_{\varsigma}) \eta(\tilde{\bm{X}}_{\varsigma})+\int_{t_{0}}^{T} \Psi^{i}(s, \tilde{\bm{X}}_{s \wedge \varsigma}, Z^i_{s \wedge \varsigma}, I_{s \wedge \varsigma-}) d s \mid \mathcal{F}_{t_{0}}],
	\end{align}
where 
	\begin{align*}
		\Psi^{i}(t,\tilde{\bm{X}}_t,Z^i_t,i_{0})=&
		\eta(\tilde{\bm{X}}_t)\left(f^{i}(t, \tilde{\bm{X}}_{t}, A_{t}^{\pi, i},I_{t-})+ \sum_{j=1}^{N} Z_{t}^{i, j} \cdot B_{t}^{j}\right)\\
		&-\sum_{j=1}^N\left(\nabla_j\eta(\tilde{\bm{X}}_t)\Sigma_t^j\left(\Sigma_t^j\right)^{\prime} Z_t^{i, j}+\frac{1}{2}\operatorname{tr}\lp Y_{t-}^i\nabla_{j,j}^2\eta(\tilde{\bm{X}}_t)\Sigma_t^j\lp\Sigma_t^j\rp^{\prime}\rp\right).
	\end{align*}
Above, we used the fact that, if \( \varsigma <T \), then \( |\tilde{\bm{X}}_{\varsigma}-\tilde{\bm{X}}_{t_{0}}|=\varrho \) and hence both \( Y_{\varsigma}^{i} \eta(\tilde{\bm{X}}_{\varsigma})=g^{i}(\tilde{\bm{X}}_{\varsigma}, I_{\varsigma}) \eta(\tilde{\bm{X}}_{\varsigma}) \) and \( \Psi^{i}(s, \tilde{\bm{X}}_{\varsigma}, Z^{{i}}_{\varsigma}, I_{\varsigma-})\) equal zero. Indeed, \( \Psi^{i}(t, \mathbf{x}, z,i_{0})=0 \) if \( |\mathbf{x}-\mathbf{x}^{0}| \geq \varrho \).  On the other hand, we have the usual It\^o formula when $\inf \{t \geq 0: |\tilde{\bm{X}}_t - \tilde{\bm{X}}_{t_0}| \geq \varrho\}$. 
	{By choosing smooth enough $\eta$ and thanks to \ref{hyp:H3}, Corollary~\ref{cor:sigma-bd} and Lemma~\eqref{eq:al-ub}, we have} that \( |\Psi^{i}(t, \mathbf{x}, z, i_{0})| \leq C(1+|z|^{2}) \), for \( t \in[0, T], \mathbf{x} \in(\mathbb{R}^{d})^{N} , z \in {\mathbb{R}^{d\times N}} \) and \( i_{0} \in \mathcal{S}\).
	
	For $\mathbf{y} \in B_{\varrho}(\mathbf{x}^{0})$, define $ p^{t, I}_{\mathbf{x},k_0}(\mathbf{y})= \mathbb{P}[\tilde{\bm{X}}_{(t_{0}+t) \wedge \varsigma} \in d\mathbf{y}\mid \tilde{\bm{X}}_{t_{0}}=\mathbf{x},I_{t_{0}}=k_0, I_{s},s \in [t_{0},t_{0}+t]] $ be the probability density of $ \tilde{\bm{X}}_{(t_{0}+t) \wedge \varsigma} $ conditioned on the initial state and the path of the Markov chain $ I_{s} $. {Denote $\mathbb{E}^{I,k_0}[\cdot]=\mathbb{E}[\cdot \mid I_{t_{0}}=k_0, I_{s},s \in [t_{0},T]]$.} 
	{Note that $p_{\mathbf{x}, k_0}^{t,I}(\mathbf{y})$ is the density of a multivariate normal distribution with variance ${\text{\text{diag}}(\E^{I,k_0} \int_{t_{0}}^{t}\Sigma^i_s\Sigma^{i \prime}_s ds)} $}. Since $\eta \equiv 0$ on $\partial B_{\varrho}(\mathbf{x}^0)$, by choosing \( \tilde{\bm{X}}_{t_{0}}=\mathbf{x} \) such that \( |\mathbf{x}-\mathbf{x}^{0}| \leq \varrho / 2 \) and setting $ I_{t_{0}}=k_0 $ we get from \eqref{eq:Y-eta} the following equation:
	\begin{multline}\label{eq:v}
		v_{i}^{\pi}(t_{0}, \mathbf{x}, k_0)= \mathbb{E}^{k_0}\left[\int_{B_{\varrho}(\mathbf{x}^{0})}p^{T-t_{0}, I}_{\mathbf{x},k_0}(\mathbf{y})g^{i}(\mathbf{y}, I_{\varsigma})\eta(\mathbf{y}) d\mathbf{y}\right. \\
		\left.+\int_{t_{0}}^{T}\int_{B_{\varrho}(\mathbf{x}^{0})} p^{s-t_{0}, I}_{\mathbf{x},k_0}(\mathbf{y}) \Psi^{i}(s, \mathbf{y}, \nabla_{\mathbf{x}}v(s \wedge \varsigma, \mathbf{y}, I_{s \wedge \varsigma-}),I_{s \wedge \varsigma-} ) d\mathbf{y}ds\right] .
	\end{multline}
	Now {by Cauchy-Schwarz} we have for any bounded and measurable function \( \psi:(\mathbb{R}^{d})^{N} \mapsto \mathbb{R} \):
	\begin{multline*}
		\lln\int_{B_{\varrho}(\mathbf{x}^{0})} \nabla_{x} p^{t,I}_{\mathbf{x}^{0},k_0}(\mathbf{y}) \psi(\mathbf{y}) d \mathbf{y}\rrn=\lln\int_{B_{\varrho}(\mathbf{x}^{0})}{\text{diag}\left(\E^{I,k_0} \int_{t_{0}}^{t}\Sigma^i_s\Sigma^{i \prime}_s ds\right)^{-1}}(\mathbf{y}-\mathbf{x}^0)p^{t,I}_{\mathbf{x}^{0},k_0}(\mathbf{y}) \psi(\mathbf{y}) d \mathbf{y}\rrn\\
		\leq\lp\int_{B_{\varrho}(\mathbf{x}^{0})} \frac{|\mathbf{y}-\mathbf{x}^0|^2}{|\E^{I,k_0} \int_{t_{0}}^{t}\Sigma^i_s\Sigma^{i \prime}_s ds|^2}p^{t,I}_{\mathbf{x}^{0},k_0}(\mathbf{y}) d \mathbf{y}\rp^{1/2}\lp\int_{B_{\varrho}(\mathbf{x}^{0})}  p^{t,I}_{\mathbf{x}^{0},k_0}(\mathbf{y}) |\psi(\mathbf{y})|^2 d \mathbf{y}\rp^{1/2}.
	\end{multline*}
	Noting that {$\int_{B_{\varrho}(\mathbf{x}^0)} |y-x^0|^2p^{t,I}_{\mathbf{x}^{0},k_0}(\mathbf{y}) d \mathbf{y} = |\text{diag}(\E^{I,k_0} \int_{t_{0}}^{t}\Sigma^i_s\Sigma^{i \prime}_s ds)|$}, we have by Corollary~\ref{cor:sigma-bd} that there exists $ c $ such that
	\begin{align}\label{eq:dxp}
		\lln\int_{B_{\varrho}(\mathbf{x}^{0})} \nabla_{x} p^{t,I}_{\mathbf{x}^{0},k_0}(\mathbf{y}) \psi(\mathbf{y}) d \mathbf{y}\rrn
		&\leq \frac{1}{\sqrt{{\E^{I,k_0} \int_{t_{0}}^{t}\nu_1^2 ds}}}{\lp\int_{B_{\varrho}(\mathbf{x}^{0})} p^{t,I}_{\mathbf{x}^{0},k_0}(\mathbf{y})|\psi(\mathbf{y})|^{2} d \mathbf{y}\rp}^{1 / 2}\nonumber\\
		&\leq \frac{c}{\sqrt{t}}{\lp\int_{B_{\varrho}(\mathbf{x}^{0})} p^{t,I}_{\mathbf{x}^{0},k_0}(\mathbf{y})|\psi(\mathbf{y})|^{2} d \mathbf{y}\rp}^{1 / 2} \nonumber\\
		&=\frac{c}{\sqrt{t}} {\lp \mathbb{E}^{{I,k_0}}{\lc |\psi(\tilde{\bm{X}}_{(t+t_{0}) \wedge \varsigma})|^{2} \mid \tilde{\bm{X}}_{t_{0}}=\mathbf{x}^{0}\rc}\rp}^{1 / 2},
	\end{align}
	which holds for a constant \( c \) independent of \( t, \mathbf{x}^{0} \) and \( \psi \) (but depending on \( \varrho \) ). {Moreover, from the symmetry of $p^{t,I}_{\mathbf{x}^{0},k_0}(\mathbf{y})$ that $p^{t,I}_{\mathbf{x}^{0},k_0}(\mathbf{x}^0+\mathbf{u})=p^{t,I}_{\mathbf{x}^{0},k_0}(\mathbf{x}^0-\mathbf{u})$ we have
	\begin{align*}
	\int_{B_{\varrho}(\mathbf{x}^{0})}	\nabla_{x}& p^{t,I}_{\mathbf{x}^{0},k_0}(\mathbf{y})g^i(\tilde{\bm{X}}_{t_0},I_\varsigma)\eta(\tilde{\bm{X}}_\varsigma)  d \mathbf{y}\\
	&=g^i(\tilde{\bm{X}}_{t_0},I_\varsigma)\eta(\tilde{\bm{X}}_\varsigma)\int_{B_{\varrho}(\mathbf{x}^{0})}\text{diag}\left(\E^{I,k_0} \int_{t_{0}}^{t}\Sigma^i_s\Sigma^{i \prime}_s ds\right)^{-1}(\mathbf{y}-\mathbf{x}^0)p^{t,I}_{\mathbf{x}^{0},k_0}(\mathbf{y})  d \mathbf{y}\\
	&=g^i(\tilde{\bm{X}}_{t_0},I_\varsigma)\eta(\tilde{\bm{X}}_\varsigma)\dfrac{1}{2}\int_{B_{\varrho}(\mathbf{0})}\text{diag}\left(\E^{I,k_0} \int_{t_{0}}^{t}\Sigma^i_s\Sigma^{i \prime}_s ds\right)^{-1}(\mathbf{u}-\mathbf{u})p^{t,I}_{\mathbf{x}^{0},k_0}(\mathbf{x}^0+\mathbf{u})  d \mathbf{u}=0.\\
	\end{align*}
	Therefore from \eqref{eq:dxp} we have that
	\begin{multline}\label{eq:dxpg-1}
		\lln\int_{B_{\varrho}(\mathbf{x}^{0})} \nabla_{x} p^{T-t_0,I}_{\mathbf{x}^{0},k_0}(\mathbf{y}) g^{i}(\mathbf{y}, I_{\varsigma})\eta(\mathbf{y}) d \mathbf{y}\rrn\\=\lln\int_{B_{\varrho}(\mathbf{x}^{0})} \nabla_{x} p^{T-t_0,I}_{\mathbf{x}^{0},k_0}(\mathbf{y}) \left(g^{i}(\mathbf{y}, I_{\varsigma})\eta(\mathbf{y})-g^i(\tilde{\bm{X}}_{t_0},I_\varsigma)\eta(\tilde{\bm{X}}_\varsigma)\right) d \mathbf{y}\rrn\\
		\leq\frac{c}{\sqrt{T-t_0}} {\lp \mathbb{E}^{I,k_0}{\lc |g^{i}(\tilde{\bm{X}}_{\varsigma},I_{\varsigma}) \eta(\tilde{\bm{X}}_{\varsigma})-g^i(\tilde{\bm{X}}_{t_0},I_\varsigma)\eta(\tilde{\bm{X}}_\varsigma)|^{2} \mid \tilde{\bm{X}}_{t_{0}}=\mathbf{x}^{0}\rc}\rp}^{1 / 2}.
	\end{multline}
	By Lipschitz continuity of $g^i$ and boundedness of $\eta$ in \eqref{eq:dxpg-1} we have 
	\begin{align}\label{eq:dxpg}
		\lln\int_{B_{\varrho}(\mathbf{x}^{0})} \nabla_{x} p^{T-t_0,I}_{\mathbf{x}^{0},k_0}(\mathbf{y}) g^{i}(\mathbf{y}, I_{\varsigma})\eta(\mathbf{y}) d \mathbf{y}\rrn \leq \frac{c}{\sqrt{T-t_0}} {\lp \mathbb{E}^{I,k_0}{\lc |\tilde{\bm{X}}_\varsigma-\tilde{\bm{X}}_{t_0}|^{2} \mid \tilde{\bm{X}}_{t_{0}}=\mathbf{x}^{0}\rc}\rp}^{1 / 2} \leq C,
	\end{align}
	where the last inequality holds since by Corollary~\ref{cor:sigma-bd} we have
	{\begin{align*}
		\mathbb{E}^{I,k_0}{\lc |\tilde{\bm{X}}_\varsigma-\tilde{\bm{X}}_{t_0}|^{2} \mid \tilde{\bm{X}}_{t_{0}}=\mathbf{x}^{0}\rc}=\E^{I,k_0} \sum_i \int_{t_{0}}^{\varsigma}{|\Sigma^i_s\Sigma^{i \prime}_s|} ds \leq C(T-t_0).
	\end{align*}}
	Also, notice that
	\begin{equation}\label{eq:psyvars}
		|\Psi^{i}(s, \tilde{\bm{X}}_{s \wedge \varsigma}, Z^i_{s \wedge \varsigma}, I_{s \wedge \varsigma-})|  
		=\mathds{1}_{\{s<\varsigma\}}|\Psi^{i}(s, \tilde{\bm{X}}_{s}, Z^i_{s}, I_{s-})| \leq C \mathds{1}_{\{s<\varsigma\}}(1+|Z^i_{s}|^{2}) .
	\end{equation}
}
	Therefore, by differentiating \eqref{eq:v} with respect to \( \mathbf{x} \) at \( \mathbf{x}=\mathbf{x}^{0} \) and {by using \eqref{eq:dxp}, \eqref{eq:dxpg} and \eqref{eq:psyvars}, we deduce that 
	\begin{multline*}
		|\nabla_{\mathbf{x}} v_{i}^{\pi}(t_{0}, \mathbf{x}^{0}, k_0)| \leq   C+C \int_{t_{0}}^{T} \frac{1}{\sqrt{s-t_{0}}}{\lp 1+{\lp\mathbb{E}^{I,k_0}[\mathds{1}_{\{s<\varsigma\}}|Z^i_{s}|^{4} \mid \tilde{\bm{X}}_{t_{0}}=\mathbf{x}^{0}]\rp}^{1 / 2}\rp} d s \\
		\leq C\lc1+\int_{t_{0}}^{T}\frac{\sup _{\mathbf{y} \in(\mathbb{R}^{d})^{N}, i_{0} \in \mathcal{S}}|\nabla_{\mathbf{x}} v_i^{\pi}(s, \mathbf{y}, i_{0})|}{\sqrt{s-t_{0}}} {\lp\mathbb{E}[\mathds{1}_{\{s<\varsigma\}}|Z^i_{s}|^{2} \mid \tilde{\bm{X}}_{t_{0}}=\mathbf{x}^{0}]\rp}^{1 / 2} d s\rc .
	\end{multline*}
	 Applying Cauchy-Schwarz inequality we have
	 \begin{multline*}
	 		|\nabla_{\mathbf{x}} v_{i}^{\pi}(t_{0}, \mathbf{x}^{0}, k_0)|\leq C\\
	 		+C\left(\int_{t_{0}}^{T} \frac{\sup _{\mathbf{y} \in(\mathbb{R}^{d})^{N}, i_{0} \in \mathcal{S}}|\nabla_{\mathbf{x}} v_i^{\pi}(s, \mathbf{y}, i_{0})|^{2}}{(s-t_{0})^{1-\beta}} d s\right)^{1 / 2}\left(\int_{t_{0}}^{T} \frac{\mathbb{E}^{I,k_0}[\mathds{1}_{\{s<\varsigma\}}|Z^i_{s}|^{2} \mid X_{t_{0}}=\mathbf{x}^{0}]}{(s-t_{0})^{\beta}} d s\right)^{1 / 2}, 
	 \end{multline*} }
	 Assume that \( T-t_{0} \leq \varrho^{2} \) and recall that $\varrho$ is given by Lemma~\ref{pro:intZ} while $\tau$ and $\zeta$ by \eqref{tau} and \eqref{eq:zeta} respectively. Therefore \( \mathds{1}_{\{s<\varsigma\}} \) equals \( \mathds{1}_{\{s<\tau\}} \). Using the conclusion of Lemma~\ref{pro:intZ} and then taking the square, we finally have:
	\begin{align*}
		|\nabla_{\mathbf{x}} v_{i}^{\pi}(t_{0}, \mathbf{x}^{0},k_0)|^{2} \leq 
		C\left[1+\int_{t_{0}}^{T} \frac{\sup _{\mathbf{y} \in(\mathbb{R}^{d})^{N}, i_{0 \in \mathcal{S}}}|\nabla_{\mathbf{x}} v_{{i}}^{\pi}(s, \mathbf{y}, i_{0})|^{2}}{(s-t_{0})^{1-\beta}} d s\right].
	\end{align*}
	We break the above integration $\int_{t_0}^T = \int_{t_0}^{t_0 + \ep} + \int_{t_0 + \ep}^T$ for any $\ep>0$. Recalling that the constant \( C \) is independent of \( \mathbf{x}^{0}, t_{0} \) and $ k_0 $, we then conclude that for a new value of \( C \):
	\begin{align*}
		\sup _{\mathbf{x} \in(\mathbb{R}^{d})^{N}, i_{0} \in \mathcal{S}}|\nabla_{\mathbf{x}} v_{{i}}^{\pi}(t_{0}, \mathbf{x}, i_{0})|^{2} \leq
		C+C \epsilon^{\beta} \sup _{t_{0} \leq t \leq T, \mathbf{x} \in(\mathbb{R}^{d})^{N}, i_{0} \in \mathcal{S}}|\nabla_{\mathbf{x}} v_{{i}}^{\pi}(t, \mathbf{x}, i_{0})|^{2} \\
		+\frac{C}{\epsilon^{1-\beta}} \int_{t_{0}}^{T} \sup _{s \leq t \leq T, \mathbf{x} \in(\mathbb{R}^{d})^{N}, i_{0} \in \mathcal{S}}|\nabla_{\mathbf{x}} v_{{i}}^{\pi}(t, \mathbf{x}, i_{0})|^{2} d s .
	\end{align*}
	Notice that the right-hand side increases as \( t_{0} \) decreases in \( [T-\varrho^{2}, T] \). Therefore,
	\begin{align*}
		\sup _{t_{0} \leq t \leq T, \mathbf{x} \in(\mathbb{R}^{d})^{N}, i_{0} \in \mathcal{S}}|\nabla_{\mathbf{x}} v_{{i}}^{\pi}(t, \mathbf{x}, i_{0})|^{2} \leq
		C+C \epsilon^{\beta} \sup _{t_{0} \leq t \leq T, \mathbf{x} \in(\mathbb{R}^{d})^{N}, i_{0} \in \mathcal{S}}|\nabla_{\mathbf{x}} v_{{i}}^{\pi}(t, \mathbf{x}, i_{0})|^{2} \\
		+\frac{C}{\epsilon^{1-\beta}} \int_{t_{0}}^{T} \sup _{s \leq t \leq T, \mathbf{x} \in(\mathbb{R}^{d})^{N}, i_{0} \in \mathcal{S}}|\nabla_{\mathbf{x}} v_{{i}}^{\pi}(t, \mathbf{x}, i_{0})|^{2} d s .
	\end{align*}
	Choosing \( \epsilon \) such that \( C \epsilon^{\beta} \leq \frac{1}{2} \), we get
	\begin{align*}
		\sup _{t_{0} \leq t \leq T, \mathbf{x} \in(\mathbb{R}^{d})^{N}, i_{0} \in \mathcal{S}}|\nabla_{\mathbf{x}} v_{{i}}^{\pi}(t, \mathbf{x}, i_{0})|^{2} \leq 
		C +\frac{C}{\epsilon^{1-\beta}} \int_{t_{0}}^{T} \sup _{s \leq t \leq T, \mathbf{x} \in(\mathbb{R}^{d})^{N}, i_{0} \in \mathcal{S}}|\nabla_{\mathbf{x}} v_{{i}}^{\pi}(t, \mathbf{x}, i_{0})|^{2} d s .
	\end{align*}
	Applying Gronwall's lemma, we deduce that for any fixed $i_{0} \in \cs$
	\[|\nabla_{\mathbf{x}} v^{\pi}(t, \mathbf{x}, i_{0})| \leq C, \]
	on \( [T-\varrho^{2}, T] \times (\mathbb{R}^{d})^{N} \). We can now duplicate the argument on \( [T-2 \varrho^{2}, T-\varrho^{2}] \times (\mathbb{R}^{d})^{N} \) by letting \( v^{\pi}(T-\varrho^{2}, \cdot, \cdot) \)  play the role of the terminal condition instead of \( g \). We complete the proof by iterating the argument a finite number of times.
\end{proof}
From Proposition~\ref{prop:4} and the discussion in Section~\ref{sec:strategy}, by taking limit as $R \rightarrow\infty$, our main result is now immediate.

\begin{theorem}\label{Thm:3.1}
	Under Hypothesis~\ref{hyp:Nash}, the Nash system \eqref{eq:Nash} has a unique solution $\{v_i(t, \mathbf{x},i_0), \, i \in [N], i_0 \in \cs \}  $ which is bounded and continuous on $[0,T] \times \mathbb{R}^{Nd}$, differentiable in $\mathbf{x} \in \mathbb{R}^{Nd}$ with a bounded and continuous gradient on $ [0,T) \times \mathbb{R}^{Nd}$.
\end{theorem}
\begin{remark}
		Note that by the boundedness of gradient of $v$ from Theorem~\ref{Thm:3.1} and inequality \eqref{eq:al-ub}, we now obtain that the optimal strategy  $\hat{\al}^k (t, \mathbf{x}, \nabla_k v_k(t, \mathbf{x}, i_0), i_0)$ is bounded. Note that this follows despite the action space being the entire $\R^m$ as mentioned in Remark~\ref{rem:A}.
\end{remark}

\subsection{Uniqueness of Nash Equilibrium}\label{Sec:uniqueNE}
A Markovian strategy is one in which the strategies in \eqref{eq:player-sde}-\eqref{eq:cost} are given by $\beta_s^k = \phi^k(s, \bm{X}_s, I_s)$ for deterministic functions $\phi^k$.
Note that  
by Theorem~\ref{Thm:3.1} one has at least one bounded Markovian Nash equilibrium given by 
$$
\hat{\phi}^k(t, \mathbf{x},i_0) = \hat{\al}^k (t, \mathbf{x}, \nabla_k v_k(t, \mathbf{x}, i_0), i_0),
$$ 
where $v_k$ solves \eqref{eq:Nash}.  We now show that it is unique. Indeed, assume there exists a bounded Markovian strategy $\hat{\bm{\phi}}$ and $\hat{\bm{X}}$ solves \eqref{eq:player-sde} when modulated by $\hat{\bm{\phi}}$. Fix any $k$ and assume every player $i \neq k$ follows the strategy $\hat{\phi}^i$. Then the best response for player $k$ is a stochastic control problem with HJB equation given by \eqref{eq:pde-bestr} with $\hat{\beta}^i$ replaced by $\hat{\phi}^i(t, \mathbf{x}, i_0)$.{
\begin{align}\label{eq:pde-bestrphi}
	0 =& \partial_t U_k(t, \mathbf{x}, i_0) + H^k(t, \mathbf{x}, \nabla_kU_k(t, \mathbf{x}, i_0),i_0 ) 
	+ \sum_{{i \neq k}} \nabla_{i} U_{k}(t, \mathbf{x}, i_{0}) \cdot  b^k(t, \mathbf{x}, \hat{\phi}^i(t,\mathbf{x},i_0), i_0) \nonumber\\
	&+ \frac{1}{2} \sum_{i=1}^N \textnormal{tr} (\nabla_i^2 U_k(t, \mathbf{x}, i_0) (\si^k)' \si^k(t, \mathbf{x}, i_0)) +\sum_{j_0 \in \cs} q_{i_0 j_0} (U_k(t, \mathbf{x}, j_0) - U_k(t, \mathbf{x}, i_0)), \\
	&U_k(T, \mathbf{x}, i_0) = g^k(\mathbf{x}, i_0),  \nonumber
\end{align}
where $H^k$ is the Hamiltonian given by
\begin{align*}
	H^k(t,\mathbf{x}, p, i_0)= \inf_{\beta^k \in \R^m} H^k(t,\mathbf{x}, p, \beta^i, i_0)=\inf_{\beta^k \in \R^m} [ f^k(t, \mathbf{x}, \beta^k, i_0) + b^k(t, \mathbf{x}, \beta^k, i_0) \cdot p ]. 
\end{align*}
}
Given $(\hat{\phi}^i)_{i \neq k}$ {the equation \eqref{eq:pde-bestrphi}} has a unique solution $U_k$ similar (but simpler) to Theorem~\ref{Thm:3.1}, with optimal strategy given by a unique $\hat{\al}^k(t, \mathbf{x}, \nabla_k U_k(t, \mathbf{x}, i_0), i_0)$ {due to Lemma~\ref{Lem:UniOpt}}. {Now similar to \cite[Theorem 2.2]{carmona-delarue-siam} {using the convexity in $x^k$ of $\si(t,\mathbf{x},i_0)$} we have}
\begin{equation*}
	\E [ \int_0^T {| \hat{\al}^k(t, \hat{\bm{X}}_t, \nabla_kU_k(t, \hat{\bm{X}}_t, I_t), I_t)} {- \hat{\phi}^k(t, \hat{\bm{X}}_t, I_t) |}^2 dt ] = 0.
\end{equation*}
The marginal law of $\hat{X}_t$ and $I_t$ is absolutely continuous with respect to Lebesgue and counting measure and has positive density. Consequently for all $(t, \mathbf{x}, i_0)$, $\hat{\phi}^k(t, \mathbf{x}, i_0) = \hat{\al}^{{k}}(t, \mathbf{x}, \nabla_k U_k(t, \mathbf{x}, i_0), i_0)$. Since $k$ is any arbitrary player index,
\begin{equation*}
	 \hat{\phi}^k(t, \mathbf{x}, i_0) = \hat{\al}^{{k}}(t, \mathbf{x}, \nabla_k U_k(t, \mathbf{x}, i_0), i_0),
\end{equation*}
for all $k$.
This transforms \eqref{eq:pde-bestrphi} to \eqref{eq:Nash} which we know has a unique solution. Therefore {for $k=1,\cdots,N$, $U_k = v_k$ and consequently $\hat{\phi}^k(t, \mathbf{x}, i_0) = \hat{\al}^k(t, \mathbf{x}, \nabla_i v_i(t, \mathbf{x}, i_0), i_0)$}. The uniqueness of the latter proves uniqueness of $\hat{\phi}$.

\section{Regime Switching $N$-player Game with Mean Field Interaction}\label{sec:N-mfg}
{
In this section, we extend our findings to cover the situation where player interactions are replaced by mean field-type interactions. This is required to relate our $N-$player game to the mean field game problem in the sequel. We start by recalling some useful notions of derivatives in the space of measures taken from \cite{ref2}.
\begin{definition}[Linear functional derivative]\label{def:lfd}
	A function $ U: \mathcal{P}(\mathbb{R}^{d}) \rightarrow \mathbb{R} $ is said to have a linear functional derivative if there exists a continuous map 
	\(\frac{\delta U}{\delta m}: \mathcal{P}\left(\mathbb{R}^{d}\right) \times \mathbb{R}^{d}  \mapsto  \mathbb{R}\),
	such that for any $ m,m' \in \mathcal{P}(\mathbb{R}) $
	\[ U(m')-U(m)=\int_{0}^{1} \int_{\mathbb{R}^{d}} \frac{\delta U}{\delta m}\left((1-s) m+s m^{\prime}, y\right) d\left(m^{\prime}-m\right)(y) d s .
	\]
\end{definition}
\begin{definition}[$L$-derivative]\label{def:ld}
	Let $ U: \mathcal{P}(\mathbb{R}^{d}) \rightarrow \mathbb{R} $  have a linear functional derivative in the sense of Definition~\ref{def:lfd}, and for all $m \in \cp(\R^d)$ the function $\frac{\der U}{\der m}(m, x)$ is differential in $x$. Then the $L-$derivative \( D_{m} U: \mathcal{P}(\mathbb{R}^{d}) \times \mathbb{R}^{d} \mapsto \mathbb{R}^{d} \) is defined by
	\[
	D_{m} U(m, y):=\nabla_{y} \frac{\delta U}{\delta m}(m, y) .
	\]
\end{definition}
Following Definitions~\ref{def:lfd} and \ref{def:ld}, the second order linear functional and $L$- derivatives are denoted as follows: 
\[
\frac{\delta^{2} U}{\delta m^{2}}\left(m, y, y^{\prime}\right)=\frac{\delta}{\delta m}\left(\frac{\delta U}{\delta m}(\cdot, y)\right)\left(m, y^{\prime}\right), 
\text{ and } 
D_{m}^{2} U\left(m, y, y^{\prime}\right)=\nabla_{y, y^{\prime}}^{2} \frac{\delta^{2} U}{\delta m^{2}}\left(m, y, y^{\prime}\right).\]
}
We are now ready to describe our game in the mean-field interaction setting.
\subsection{Mean Field Interaction $N$-player Game}
Consider the mean field type interaction within the setup of the $N$-player differential game with regime switching where the $k$-th player satisfies:
\beq\label{eq:Mplayer-sde}
dX_{s}^{k}=b(s, X_{s}^k,\mu^{N-1,k}_s , \beta_s^k , I_{s}) d s+\sigma(s, X_{s}^k,\mu^{N-1,k}_s, I_{s}) dW_{s}^{k},
\eeq
where we denote $\mu^{N-1,k}_s=\mu^{N-1,k}(X_{s})$ to be the empirical measure {$\frac{1}{N-1} \sum_{i \neq k}^{N-1} \delta_{X^{i}_{s}}$}.  Here {$ b:[0,T] \times \mathbb{R}^{d} \times \mathcal{P}(\mathbb{R}^d) \times A \times \mathcal{S} \rightarrow \mathbb{R}^{d} $ and 
$ \sigma:[0,T] \times \mathbb{R}^{d} \times \mathcal{P}(\mathbb{R}^d) \times A \times \mathcal{S} \rightarrow \mathbb{R}^{d}$,
where $A$ and $\cs$ are defined in Section~\ref{sec:prob-setup}.} 
We impose that the game is symmetric, that is the coefficients $b,\sigma, f, g$ do not depend on $k$. In this case, the objective function is of the form:
\begin{multline}\label{eq:Mcost}
	J_{k}(t, x, \mu, i_0, \beta^{k}, \bm{\beta}^{-k} )\\= \mathbb{E}[{\int_{t}^{T} f (s, X_{s}^k,\mu^{N-1,k}_s, \beta_{s}^{k}, I_{s}) d s} +g ( X_{T}^k,\mu^{N-1,k}_s , I_{T})| X_t^k = x, \mu_t^{N-1,k} = \mu, I_t = i_0],
\end{multline}
for $k=1, \ldots, N$ 
where $X^k$ solves \eqref{eq:Mplayer-sde} modulated by the strategy $\beta^k$. In addition, $ f:[0,T] \times \mathbb{R}^{d} \times \mathcal{P}(\mathbb{R}^d) \times {A} \times \mathcal{S} \rightarrow \mathbb{R} $, and $ g:\mathbb{R}^{d} \times \mathcal{P}(\mathbb{R}^d) \times {A} \times \mathcal{S} \rightarrow \mathbb{R} $. We now show that this game has a unique Nash equilibrium. 

\subsection{Assumptions}
For the purpose of completeness, we state here similar assumptions to Hypothesis~\ref{hyp:Nash}.
\begin{hypothesis}\label{hyp:Nashm}
	The admissible control set $A$ is the entire $\R^{m}$. For any fixed $i_{0} \in \mathcal{S}$, there exist positive constants $L,L^\prime$ and $\lambda$ such that: 
	
	\begin{enumerate}[label={(A\arabic*)}]
		\item \label{hyp:A1} The drift function $b(t,x,\mu,\al,i)$ has the following form, where $b_{0}, b_1$ and $b_{2}$ are $\mathbb{R}^d, \mathbb{R}^{d \times d}$ and $\mathbb{R}^{d \times m}$ valued respectively, and are bounded by $c_L$ for all $i_0 \in \mathcal{S}$. In particular, $b$ reads
		\begin{equation*}
			b(t,x,\mu,\alpha,i_0) = b_{0}(t,\mu,i_0) + b_{1}(t,i_0) x + b_{2}(t,i_0) \alpha.
		\end{equation*}
		Moreover, $b_{0}$ satisfies for any 
		$\mu,\mu' \in {\mathcal P}_{2}(\mathbb{R}^d)$:
		$
		\vert b_{0}(t,\mu',i_0) - b_{0}(t,\mu,i_0) \vert \leq c_\nu W_{2}(\mu,\mu')$. 
		
		\item \label{hyp:A2} For any $\mu \in {\mathcal P}_{2}(\mathbb{R}^d)$, the function 
		$\mathbb{R}^d \ni x \hookrightarrow g(x,\mu,i_0)$ is once continuously differentiable and convex and has a $L$-Lipschitz-continuous first order derivative. Moreover,
			\begin{align}
			|g(x,\mu,i_{0})| \leq L,   \label{eq:gg2}\\
			|\nabla_xg(x,\mu,i_0)|\leq L, \label{eq:ggrad2}
		\end{align} 
		
		\item {\label{hyp:A3}For any $(t, \mu) \in [0, T] \times \mathcal{P}_2(\mathbb{R}^d)$, the function $f(t, x, \mu, \alpha, i_0)$ is once continuously differentiable in $x$ and $\alpha$, with derivatives $\nabla_x f$ and $\nabla_\alpha f$ that are $L$-Lipschitz continuous in $x$ and $\alpha$. Additionally, the following inequality holds:
			\begin{equation*}
				\nabla_\alpha f(t, x, \mu', \alpha, i_0) - \nabla_\alpha f(t, x, \mu, \alpha, i_0) \leq c_\alpha W_2(\mu', \mu).
			\end{equation*}
			Furthermore, $f$ is strongly convex in $\alpha$ in the sense that:
			\begin{equation*}
				f(t, x', \mu, \alpha', i_0) - f(t, x, \mu, \alpha, i_0) - \langle (x' - x, \alpha' - \alpha), \nabla_{(x, \alpha)} f(t, x, \mu, \alpha, i_0) \rangle \geq \lambda |\alpha' - \alpha|^2.
			\end{equation*}
}
		
		\item \label{hyp:A4}For all $ (t,x,\mu,\alpha) \in [0,T] \times \mathbb{R}^{dN}  \times A$, 
		\[|\nabla_{\alpha} f(t, x,\mu, \alpha,i_{0})| \leq L^\prime(1+|\alpha|), \text{ where } L^\prime < \lambda. \]
		\begin{equation}\label{eq:ff2}
			|f(t, \mathbf{x}, \alpha,i_{0})| \leq L(1+|\alpha|^{2}).
		\end{equation}
		\item \label{hyp:A5}For any $ i,j \in \mathcal{S}, i\neq j $, the transition rates
		$q_{ij} \leq L $.
		
		{\item \label{hyp:A6}$ \sigma: [0, T] \times \mathbb{R}^{d} \times \mathcal{P}_2(\mathbb{R}^d) \times \cs \mapsto \mathbb{R}^{d \times d} $ is Lipschitz continuous in $ x$:
		\begin{equation*}
			 |\sigma(t,x,\mu,i_{0})-\sigma(t,y,\mu,i_{0})| \leq L(|x-y|).
		\end{equation*}
		}
		\item \label{hyp:A7}{ $\sigma(t,x,\mu,i_{0})$ is convex in $x$. Moreover, for $(t,x,\mu) \in [0,T] \times \mathbb{R}^{d} \times \mathcal{P}_2(\mathbb{R}^d)$, $\sigma(t,x,\mu,i_{0})$ is invertible and has continuous second derivative in $t$ and $x$. Denote $a_{ij}(t,x,\mu,i_{0})=\sigma(\sigma)^{\prime}_{ij}(t,x,\mu,i_{0})$ {and $\xi \in \R^d$. Then} the following are satisfied}
		\begin{align*}
			\nu_{1}\xi^{2} &\leq a_{ij}(t,x,\mu,i_{0})\xi_{i}\xi_{j} \leq \nu_{2}\xi^{2}, \quad \nu_i > 0 \\
			&\left|\frac{\partial a_{ij}(t,x,\mu,i_{0})}{\partial x}\right| \leq \nu_{2} . 
		\end{align*}
		
		\item \label{hyp:A8} For $h=b_0,f,g,\sigma$ and some constant $\Gamma$
		\begin{align*}
			& \mathbb{R}^d \times \mathcal{P}_2\left(\mathbb{R}^d\right) \times \mathbb{R}^d \ni(t, x, \mu, v) \mapsto\left(D_\mu h(t, x, \mu)(v), \nabla_v D_\mu h(t, x, \mu)(v)\right) \\
			& \mathbb{R}^d \times \mathcal{P}_2\left(\mathbb{R}^d\right) \times \mathbb{R}^d \times \mathbb{R}^d \ni\left(t, x, \mu, v, v^{\prime}\right) \mapsto D_\mu^2 h(t, x, \mu)\left(v, v^{\prime}\right)
		\end{align*}
		are bounded by $\Gamma$ and are $\Gamma$-Lipschitz continuous with respect to $(x, \mu, v)$ and to $\left(x, \mu, v, v^{\prime}\right)$, for any $t \in[0, T]$.
			\end{enumerate}
\end{hypothesis}
{ Similar to Corollary~\ref{cor:sigma-bd} we have the following result bounding $\si$ in the new setting. 
	\begin{corollary}\label{cor:sigma-bd-2}
		From Hypothesis~\ref{hyp:A7} we have that for any $(t,x,\mu,i_0) \in [0,T] \times \mathbb{R}^{dN} \times \mathcal{S}$, there exist constant $C_1,C_2>0$ such that 
		\begin{align*}
			C_1\leq|\sigma(t,x,\mu,i_0)|^2\leq C_2.
		\end{align*}
	\end{corollary}
}
We now state our existence and uniqueness result in the mean-field interaction setting.
\begin{theorem} 
	Under Hypothesis~\ref{hyp:Nashm}, the N-player mean field type stochastic differential game with regime switching described in \eqref{eq:Mplayer-sde}-\eqref{eq:Mcost} has unique Nash equilibrium with bounded Markovian optimal strategy.{ Moreover, for all $i\in[N], i_0\in\cs$, the value function $v_i(t,x,\mu,i_0)=\inf_{{\beta \in \mathbb{A}}} J_k(t, x, \mu, i_0, \beta, \hat{\bm \beta}^{-k})$ is {bounded and continuous on $[0,T] \times \mathbb{R}^{d}  \times \mathcal{P}_2(\mathbb{R}^d)$, differentiable in $x$ and $\mu$, with a bounded and continuous gradient on $ [0,T) \times \mathbb{R}^{d} \times \mathcal{P}_2(\mathbb{R}^d)$.}}
\end{theorem}
\begin{proof}
	One can easily recover the notations in Theorem~\ref{Thm:3.1} by letting 
	\begin{align*}
		&f^k(t, \mathbf{x}, \al, i_0) = f(t, x^i, \mu^{N-1,k}(\mathbf{x}), \al, i_0 ) , \quad g^k(\mathbf{x},i_0)=g(t,x^i,\mu^{N-1,k}(\mathbf{x}),i_0),\\
		&b^k(t, \mathbf{x}, \al, i_0) = b(t, x^i, \mu^{N-1,k}(\mathbf{x}), \al, i_0 ) , \quad \sigma^k(t,\mathbf{x}, i_0)= \sigma(t,x^i,\mu^{N-1,k}(\mathbf{x}),i_0).
	\end{align*}
	{The differentiability of $b^k,f^k,g^k,\sigma^k$ follows from \ref{hyp:A1}, \ref{hyp:A2}, \ref{hyp:A3}, \ref{hyp:A7} and \ref{hyp:A8} by using \cite[Prop. 6.30]{ref3}. In particular, $\sigma^k$ has continuous second derivatives.} 
	Then the result directly follows from Theorem~\ref{Thm:3.1} and discussion in Section~\ref{Sec:uniqueNE}.
\end{proof}

\section{Regime Switching Mean Field Game.}\label{sec:rs-mfg}
In this section we study the regime switching mean field game problem. This corresponds to the regime switching $N$-player game with mean field interaction, but in the limit when $N \to \infty$. We start by describing the problem at hand.
\subsection{Mean Field Game Problem}
Assume that a Markov chain $(I_t)_{0 \leq t \leq T}$ and a $d$-dimension Brownian motion $(W_t)_{0 \leq t \leq T}$ are defined on the probability space \( (\Omega, \mathcal{F},(\mathcal{F}_{s})_{s \in[t, T]}, \mathbb{P}) \) mentioned in Section~\ref{sec:prob-setup}. For each $t>0$, denote $\mathcal{F}_t^I=\sigma\{I_s: 0 \leq s \leq t\}$, $\mathcal{F}_t^W=\sigma\{W_s:$ $0 \leq s \leq t, \}$, and put $\mathcal{F}_t=\mathcal{F}_t^{W, I}=\sigma\{W_s, I_s: 0 \leq s \leq t\}$. Let $b$, $\si$ be similar to Section~\ref{sec:N-mfg}. Consider the following controlled McKean-Vlasov type stochastic differential equation for a specific continuous $\mathcal{P}_2(\mathbb{R}^d)$-valued stochastic process $\bm{\mu}=(\mu_s)_{0 \leq s \leq T}$ adapted to $\mathcal{F}_t^{I}$:
\begin{align}\label{eq:player-sde2}
	d X_{s}=b\left(s, X_{s}, \mu_{s}, \beta_{s}, I_{s}\right) d s+\sigma\left(s, X_{s}, \mu_{s}, I_{s}) dW_{s},
	\right.
\end{align}
for $\beta \in \mathbb{A}$. We consider the game problem when continuum of players follow this stochastic differential equation, where each player considers an optimal control problem with $f$ and $g$ to be the running and terminal costs respectively given again in Section~\ref{sec:N-mfg}. The cost functional is given by
\begin{equation*}
	J(t,x, \bm{\mu},\beta,i_0)=\mathbb{E}\left[\left.\int_t^{T} f\left(s, X_{s}, \mu_{s}, \beta_{s}, I_{s-}\right) d s+g\left( X_{T}, \mu_{T}, I_{T}\right) \right|X_t=x, I_t=i_0\right],
\end{equation*}
where $(X_s)_{s \in[0,T]}$ solves the SDE~\eqref{eq:player-sde2}, and $\bm{\mu}$ is as before. Then denote
	\begin{equation}\label{eq:valuemu}
		v^{\bm{\mu}}(t,x,i_0)= \inf_{\beta \in \mathbb{A}} J(t,x,\bm{\mu},\beta,i_0),
	\end{equation}
	to be the optimal value function.
\begin{definition}\label{def:mfg}
A Nash equilibrium of a mean field game problem is said to be reached if a given flow of random probability measures $\hat{\mu}_s:[t,T] \times \R^d \rightarrow \mathbb{R}$ and a control $\hat{\beta}$ satisfies the following two conditions:
\begin{enumerate}
\item $\hat{\beta}=\textnormal{argmin}_{\beta \in \mathbb{A}}J(t,x, \hat{\bm{\mu}},\beta,i_0)$ 
\item Denote the optimal path under optimal control $\hat{\beta}$ with probability measure flow $\bm{\mu}$ to be $\hat{X}_s[\bm{\mu}]$. Then for all $s \in [t,T]$, $ \hat{\mu}_s=\mathcal{L}\left(\hat{X}_{s}[\hat{\bm{\mu}}] \mid \mathcal{F}_{s}^{I}\right)$.
\end{enumerate}
\end{definition}
\begin{remark}
	$\mathcal{L}(X_t \mid \mathcal{F}^I_t)$ is well defined since $\mathcal{S}$ is a finite set. In the rest of the paper, we will use $\mathcal{L}^1(X_t)$ to denote a version of the conditional law of $\mathcal{L}(X_t \mid \mathcal{F}^I_t)$.
	\end{remark}
\subsection{Assumptions}

In addition to Hypothesis~\ref{hyp:Nashm}, we will require some additional assumptions.
\begin{hypothesis}\label{hyp:Nash2}
	Let Hypothesis~\ref{hyp:Nashm} hold except for \eqref{eq:gg2}, \eqref{eq:ggrad2} and \eqref{eq:ff2}. Furthermore, there exist constant $c_L>0$ such that 
	\begin{enumerate}[label={(A\arabic*)}, start=9]	
		\item \label{hyp:A10} For all $t \in [0,T]$, $x,x' \in \mathbb{R}^d$, $\alpha,\alpha' \in \mathbb{R}^m$ and $\mu,\mu' \in {\mathcal P}_{2}(\mathbb{R}^d)$ the following holds uniformly in $i_0 \in \cs$:
		\begin{multline*}
				\bigl\vert (f,g)(t,x',\mu',\alpha',i_0) - (f,g)(t,x,\mu,\alpha,i_0)\bigr\vert 
				\\
				\leq c_{L} \bigl[ 1 + \vert (x',\alpha') \vert + \vert (x,\alpha) \vert + M_{2}(\mu) + M_{2}(\mu') \bigr]
				\bigl[ \vert (x',\alpha') - (x,\alpha) \vert +
				W_{2}(\mu',\mu) \bigr].
		\end{multline*} 
		
			\item \label{hyp:A12} For all $t \in [0,T]$, $x,x' \in \mathbb{R}^d$, $\alpha,\alpha' \in \mathbb{R}^m$ and $\mu,\mu' \in {\mathcal P}_{2}(\mathbb{R}^d)$ the following holds uniformly in $i_0 \in \cs$:
		\begin{align*}
			\bigl\vert \nabla_x f(t,x',\mu',\alpha',i_0) - \nabla_x f(t,x,\mu,\alpha,i_0)\bigr\vert &\leq c_{\theta}  \vert (x',\alpha') - (x,\alpha) \vert + c_\nu W_{2}(\mu',\mu) \\
			\bigl\vert \nabla_x g(x',\mu',i_0) - \nabla_x g(x,\mu,i_0)\bigr\vert &\leq c_{\theta}  \vert x' - x \vert + c_\mu W_{2}(\mu',\mu).
		\end{align*} 
		
		\item \label{hyp:A13} There exists a positive constant $K_h$ such that for any $x_1, x_2 \in \mathbb{R}^d, \mu \in \mathcal{P}_2\left(\mathbb{R}^d\right)$, and $i_0 \in \mathcal{S}$, we have
		
		$$
		\left\langle \nabla_x g\left(x_1, \mu, i_0\right)-\nabla_x g\left(x_2, \mu, i_0\right), x_1-x_2\right\rangle \geq K_h\left|x_1-x_2\right|^2, \quad \mathbb{P} \text {-a.s. }
		$$
		
		\item \label{hyp:A14} $b_2(t,i_0)$ is invertible for all $t$ and $i_0$, and there exists a positive constant $K_\Psi$ such that for any $x_1, x_2 \in \mathbb{R}^d, \al_1, \al_2 \in \mathbb{R}^m, q_1,q_2 \in \mathbb{R}^{d \times d}, \mu \in \mathcal{P}_2\left(\mathbb{R}^d\right)$, and $i_0 \in \mathcal{S}$, we have $\mathbb{P}$-almost surely
		\begin{multline*}
			K_h\left(\left|x_1-x_2\right|^2+\left|q_1-q_2\right|^2\right)\leq-K_h\left|\left(\nabla_\al f\left(t,x_1, \mu,\al_1, i_0\right)-\nabla_\al f\left(t, x_2, \mu, \al_2, i_0\right)\right)b_2(t,i_0)^{-1}\right|^2\\
			+\left\langle \nabla_{(x,\al)} f\left(t,x_1, \mu,\al_1, i_0\right)-\nabla_{(x,\al)} f\left(t, x_2, \mu, \al_2, i_0\right), (x_1-x_2,\al_1-\al_2)\right\rangle+\\ 
			\langle\nabla_x\left(\operatorname{tr}(q_1^\prime\sigma(t,x_1,\mu,i_0))-\operatorname{tr}(q_2^\prime \sigma(t,x_2,\mu,i_0))\right),q_1-q_2\rangle-\left[\sigma(t,x_1,\mu,i_0)-\sigma(t,x_2,\mu,i_0),q_1-q_2\right],
		\end{multline*}
		here $[A, B]=\sum_{i=1}^d\left\langle A_i, B_i\right\rangle$, where $A_i$ and $B_i, i=1,2, \ldots, d$, are the $i$-th column of $d \times d$ matrices $A$ and $B$. Moreover, the constants
		\begin{equation*}
			c_\theta\max\{c_L, c_\alpha\},c_\nu, c_\mu<\min \left\{(\sqrt{3}-1) K_h, K_{\Psi} / \sqrt{3}\right\}.
		\end{equation*}
		
		\item \label{hyp:A15} If $\sigma$ does not depend on $\mu$, the above assumption can be relaxed to the following. $b_2(t,i_0)$ is invertible for all $t$ and $i_0$, and there exists a positive constant $K_\Psi$ such that for any $x_1, x_2 \in \mathbb{R}^d, \al_1, \al_2 \in \mathbb{R}^m, q_1,q_2 \in \mathbb{R}^{d \times d}, \mu \in \mathcal{P}_2\left(\mathbb{R}^d\right)$, and $i_0 \in \mathcal{S}$, we have $\mathbb{P}$-almost surely
		\begin{multline*}
			K_h\left|x_1-x_2\right|^2\leq-K_h\left|\left(\nabla_\al f\left(t,x_1, \mu,\al_1, i_0\right)-\nabla_\al f\left(t, x_2, \mu, \al_2, i_0\right)\right)b_2(t,i_0)^{-1}\right|^2\\
			+\left\langle \nabla_{(x,\al)} f\left(t,x_1, \mu,\al_1, i_0\right)-\nabla_{(x,\al)} f\left(t, x_2, \mu, \al_2, i_0\right), (x_1-x_2,\al_1-\al_2)\right\rangle\\ 
			+\langle\nabla_x\left(\operatorname{tr}(q_1^\prime\sigma(t,x_1,i_0))-\operatorname{tr}(q_2^\prime \sigma(t,x_2,i_0))\right),q_1-q_2\rangle-\left[\sigma(t,x_1,i_0)-\sigma(t,x_2,i_0),q_1-q_2\right].
		\end{multline*}
		Moreover, the constants
		\begin{equation*}
			c_\theta\max\{c_L, c_\alpha\},c_\nu, c_\mu<\min \left\{2(\sqrt{2}-1) K_h, K_{\Psi} / \sqrt{2}\right\}.
		\end{equation*}
		
		{\item \label{hyp:gradsigma} $\nabla_x\sigma(t,x,\mu,i_0)$ is bounded by $c_L$ and Lipschitz continuous, and
		\begin{equation*}
			\bigl\vert \sigma(t,x,\mu',i_0) - \sigma(t,x,\mu,i_0)\bigr\vert 
			\leq c_{\nu} W_{2}(\mu',\mu) .
		\end{equation*} 
		} 
	\end{enumerate}
\end{hypothesis}
\subsection{Stochastic Maximum Principle and FBSDE}
In this section, we provide some results establishing the stochastic maximum principle and the regime switching McKean Vlasov FBDSE referred to earlier. We mention that in a control setting and with a slightly different interaction, stochastic maximum principle has been established in \cite{nguyen2021general-smp-regime}, while study of such FBSDEs have been performed in \cite{rolon2024markovian}. While we use the results in \cite{rolon2024markovian} for existence and uniqueness, we derive the stochastic maximum principle in our setting and draw the connection to our FBSDE. To that end, let us first define the Hamiltonian to be
\begin{align}\label{eq:Ham}
	&\mathcal{H}(t,x,\mu,\al,p,q,i_0)={p\cdot b(t,x,\mu,\al,i_0)}  +f(t,x,\mu,\al,i_0)+\operatorname{tr}{\left(q^{\prime}\sigma (t,x,\mu,i_0)\right)}.
	\end{align}
Then the following lemma holds.
\begin{lemma}\label{Lem:UniOpt2}
Under Hypothesis~\ref{hyp:Nash2}, for any $i_0 \in \mathcal{S}$, there exists a unique minimizer $\hat{\alpha}(t, x, \mu, p, i_0)$ of the mapping
\[
\alpha \mapsto \mathcal{H}(t, x, \mu, \alpha, p, q, i_0).
\]
Furthermore, $\hat{\alpha}$ is measurable, locally bounded, and Lipschitz continuous with respect to $x$, $\mu$, and $p$, uniformly in $t$, with a Lipschitz constant $\max\{c_L, c_\alpha\}$. Moreover,
\[
|\hat{\alpha}(t, x, \mu, p, i_0)| \leq C(1 + |p|),
\]
for any $(t, x, \mu, i_0) \in [0, T] \times \mathbb{R}^d \times \mathcal{P}_2(\mathbb{R}^d) \times \mathcal{S}$.
\end{lemma}
\begin{proof}
	The proof is similar to Lemma~\ref{Lem:UniOpt} because of Hypothesis~\ref{hyp:Nashm}~\ref{hyp:A3}.
	\end{proof}
	Before explaining the relation between the solvability of our FBSDE and the Nash equilibrium of the regime switching mean field game, let us first establish two preliminary lemmas. These two lemmas establish the maximum principle of the stochastic control problem with regime switching given by \eqref{eq:player-sde2} - \eqref{eq:valuemu} albeit under some additional assumptions.  {Note that the proofs are similar to the treatments done in \cite{pham2009continuous, yong2012stochastic} but included below for completeness.}

	\begin{proposition}\label{Lem:FBSDEtoMean}
	For any given $\mathcal{F}^I$-adapted measure flow $\bm{\mu}$, if the process $(\hat{X}_t, \hat{P}_t, \hat{Q}_t, \hat{\Lambda}_t)$ solves the following FBSDE with regime switching
	\begin{align}\label{eq:fbsde2}
			&\hat{X}_t=x_0+\displaystyle \int_0^t b\left(s,X_s,\mu_s,\hat{\al}_s\left(s,\hat{X}_s,\mu_s,\hat{P}_{s-},I_{s-}\right),I_{s-}\right)ds+\int_0^t \sigma \left(s,\hat{X}_s,\mu_s,I_{s-}\right)dW_s, \nonumber \\
			&\hat{P}_t=\nabla_x g\left(\hat{X}_T,\mu_T,I_T\right)+ \displaystyle \int_t^T\nabla_x \mathcal{H}\left(s,\hat{X}_s,\mu_s,\hat{\al}_s\left(s,\hat{X}_s,\mu_s,\hat{P}_{s-},I_{s-}\right),\hat{P}_{s-},\hat{Q}_{s-},I_{s-}\right)ds \nonumber \\
			&\hspace{1cm}- \displaystyle \int_t^T \hat{Q}_{s-} dW_s-\int_t^T\hat{\Lambda}_s\cdot dM_s,
	\end{align}
	such that
	\beq\label{eq:inthatxpql}
	\mathbb{E}\left[\sup _{0 \leq t \leq T}\left(\left|\hat{X}_t\right|^2+\left|\hat{P}_t\right|^2\right)+\int_0^T\left|\hat{Q}_t\right|^2 d t\right]+\mathbb{E} \int_0^T\hat{\Lambda}_t^{\circ 2} \cdot d[M]_t<+\infty,
	\eeq
	then $\hat{\al}_t=\hat{\al}(t, \hat{X}_t,\mu_t, \hat{P}_{t-},I_{t-})  $ is the optimal control for the problem \eqref{eq:player-sde2}-\eqref{eq:valuemu}, that is, for any $\al \in \mathbb{A}$, $x_0 \in \mathbb{R}^d$ and $i_0 \in \cs$
	\beq\label{eq:hatal-opt}
	J(0,x_0,\bm{\mu},\hat{\alpha},i_0)\leq J(0,x_0,\bm{\mu},\alpha,i_0).
	\eeq
		\end{proposition}
\begin{proof}
For any $\al \in \mathbb{A}$, let us denote $X_t$ to be the associated diffusion controlled by  $\al$. Then 
\begin{multline}\label{eq:Jhatal-Jal}
J(0,x_0,\bm{\mu},\hat{\alpha},i_0)- J(0,x_0,\bm{\mu},\alpha,i_0)=\mathbb{E}^{x_0,i_0}\left[\int_0^T f\left(t, \hat{X}_t,\mu_t, \hat{\alpha}_t,I_{t-}\right)-f\left(t, X_t, \mu_t, \alpha_t, I_{t-}\right) d t\right]\\
+\E^{x_0,i_0} \left[g\left(\hat{X}_T,\mu_T,I_T\right)-g\left(X_T,\mu_T,I_T\right)\right] ,
\end{multline}
{where $\E^{x_0,i_0}[\cdot] = \E[\cdot \mid X_0=x, I_0=i_0]$.}
In order to control the second term in \eqref{eq:Jhatal-Jal} we need some preliminary bounds on $(\hat{X}_T-X_T) \cdot \hat{P}_T$. To that effect, recall the optional quadratic covariations given by \eqref{eq:mg-ortho}. 
Then by generalized It\^o's formula \cite{Protter}
\begin{multline*}
\mathbb{E}^{x_0,i_0}\left[\left(\hat{X}_T-X_T\right) \cdot \hat{P}_T\right] 
= \mathbb{E}^{x_0,i_0}\left[\int_0^T\left(\hat{X}_t-X_t\right) \cdot d \hat{P}_t+\int_0^T \hat{P}_{t-} \cdot\left(d \hat{X}_t-d X_t\right)+{A_1}\right.\\
\left.+\sum_{0<t<T} \left[ \left(\hat{X}_t-X_t\right) \cdot \hat{P}_t-\left(\hat{X}_{t-}-X_{t-}\right) \cdot \hat{P}_{t-} -\left(\hat{X}_{t-}-X_{t-}\right) \cdot {\Delta\hat{P}_t}\right]\right],
\end{multline*}
{where $A_1=\int_0^T \operatorname{tr}[\hat{Q}_{t-}^{\prime}(\sigma(t,\hat{X}_t, \mu_t,I_{t-})-\sigma(t,X_t, \mu_t,I_{t-}))] d t$. Since $\hat{X}_t-X_t$ is continuous we have}
\begin{align}\label{eq:hatX-XdotY}
	\mathbb{E}^{x_0,i_0}\left[\left(\hat{X}_T-X_T\right) \cdot \hat{P}_T\right]&=\mathbb{E}^{x_0,i_0}\left[\int_0^T\left(\hat{X}_t-X_t\right) \cdot d \hat{P}_t+\int_0^T \hat{P}_{t-} \cdot\left(d \hat{X}_t-d X_t\right)+{A_1}\right].
\end{align}
Now, by convexity of $g$ we have
\begin{align}\label{eq:conv-g-bnd}
	 \mathbb{E}^{x_0,i_0}\left[g\left(\hat{X}_T,\mu_T,I_T\right)-g\left(X_T,\mu_T,I_T\right)\right]& 
	\leq  \mathbb{E}^{x_0,i_0}\left[\left(\hat{X}_T-X_T\right) \cdot \nabla_x g\left(\hat{X}_T,\mu_T,I_T\right)\right]\nonumber\\
&	=\mathbb{E}^{x_0,i_0}\left[\left(\hat{X}_T-X_T\right) \cdot \hat{P}_T\right] .
\end{align}
Plugging \eqref{eq:hatX-XdotY} in \eqref{eq:conv-g-bnd} we get
\begin{multline}\label{eq:Eghat-g}
	\mathbb{E}^{x_0,i_0}\left[g\left(\hat{X}_T,\mu_T,I_T\right)-g\left(X_T,\mu_T,I_T\right)\right]\leq  \mathbb{E}^{x_0,i_0}\left[\int_0^T\left(\hat{X}_t-X_t\right) \cdot d \hat{P}_t\right.\\
	+\left.\int_0^T \hat{P}_{t-} \cdot\left(d \hat{X}_t-d X_t\right)+{A_1}\right].
\end{multline}
{Also notice that from \eqref{eq:inthatxpql}, Cor.~\ref{cor:sigma-bd-2} and square integrability of $\al$
\begin{equation*}
	\int_0^T \langle \hat{X}_t-X_t,\hat{Q}_{t-}dW_t\rangle ,\quad \int_0^T \langle \hat{X}_t-X_t,\hat{\Lambda}_t\cdot dM_t\rangle ,\quad \int_0^T \langle \hat{P}_{t-},\sigma\left(t,\hat{X}_t,\mu_t,I_{t-}\right) dW_t\rangle, 
\end{equation*}
are square integrable martingales. Therefore, using the definition \eqref{eq:fbsde2} in \eqref{eq:Eghat-g} we obtain}
\begin{align}\label{eq:ghat-g}
	\mathbb{E}^{x_0,i_0}&\left[g\left(\hat{X}_T,\mu_T,I_T\right)-g\left(X_T,\mu_T,I_T\right)\right]\nonumber\\
	&\leq \mathbb{E}^{x_0,i_0}\left[\int_0^T\left(\hat{X}_t-X_t\right) \cdot\left(-\nabla_x \mathcal{H}\left(t, \hat{X}_t,\mu_t, \hat{\alpha}_t, \hat{P}_{t-}, \hat{Q}_{t-},I_{t-}\right)\right) d t+{A_1+A_2}\right] ,
\end{align}
{where $A_2=\int_0^T \hat{P}_{t-} \cdot(b(t,\hat{X}_t,\mu_t, \hat{\alpha}_t,I_{t-})-b(t,X_t, \mu_t,\alpha_t,I_{t-})) d t$.} For the first term in the right hand side of \eqref{eq:Jhatal-Jal} we obtain by definition of $\mathcal{H}$
\begin{multline}\label{eq:fhat-f}
	\mathbb{E}^{x_0,i_0}\left[  \int_0^T f\left(t, \hat{X}_t,\mu_t, \hat{\alpha}_t,I_{t-}\right)-f\left(t, X_t,\mu_t, \alpha_t,I_{t-}\right) d t\right]\\
	=\mathbb{E}^{x_0,i_0}\left[\int_0^T \mathcal{H}\left(t, \hat{X}_t, \hat{\alpha}_t, \hat{P}_{t-}, \hat{Q}_{t-},I_{t-}\right)-\mathcal{H}\left(t, X_t, \alpha_t, \hat{P}_{t-}, \hat{Q}_{t-},I_{t-}\right) d t -{A_1-A_2}\right] .
\end{multline}
By adding \eqref{eq:ghat-g} and \eqref{eq:fhat-f} into \eqref{eq:Jhatal-Jal}, we obtain
\begin{multline}\label{eq:J-bnd}
	J(0,x_0,\bm\mu,\hat{\alpha},i_0)- J(0,x_0,\bm\mu,\alpha,i_0) \leq  \mathbb{E}^{x_0,i_0}\left[ \int_0^T \mathcal{H}\left(t, \hat{X}_t,\mu_t, \hat{\alpha}_t, \hat{P}_{t-}, \hat{Q}_{t-},I_{t-}\right)\right.\\
	\left.-\mathcal{H}\left(t, X_t,\mu_t, \alpha_t, \hat{P}_{t-}, \hat{Q}_{t-},I_{t-}\right) d t 
	 -\int_0^T\left(\hat{X}_t-X_t\right) \cdot \nabla_x \mathcal{H}\left(t, \hat{X}_t,\mu_t, \hat{\alpha}_t, \hat{P}_{t-}, \hat{Q}_{t-},I_{t-}\right) d t\right].
\end{multline}
{By the optimality of $\hat{\al}$ we know that $(\hat{\al}_t-\al_t) \cdot \nabla_\al \mathcal{H}(t, \hat{X}_t,\mu_t, \hat{\alpha}_t, \hat{P}_{t-}, \hat{Q}_{t-},I_{t-}) \leq 0$ for all $t \in [0,T]$. Thus by noting that $\mathcal{H}$ is convex in $(x,\al)$ we have}
\begin{multline}\label{eq:H-bnd}
	\mathbb{E}^{x_0,i_0}\left[\int_0^T \mathcal{H}\left(t, \hat{X}_t,\mu_t, \hat{\alpha}_t, \hat{P}_{t-}, \hat{Q}_{t-},I_{t-}\right)-\mathcal{H}\left(t, X_t,\mu_t, \alpha_t, \hat{P}_{t-}, \hat{Q}_{t-},I_{t-}\right) d t\right. \\
	 \left.{-\int_0^T\left(\hat{X}_t-X_t,\hat{\al}_t-\al_t\right) \cdot \nabla_{(x,\al)} \mathcal{H}\left(t, \hat{X}_t,\mu_t, \hat{\alpha}_t, \hat{P}_{t-}, \hat{Q}_{t-},I_{t-}\right) d t}\right]\leq 0.
\end{multline}
Finally from \eqref{eq:J-bnd} and \eqref{eq:H-bnd} we get our desired result~\eqref{eq:hatal-opt}.
	\end{proof}
		Similarly to \cite[Prop 2.5]{carmona-delarue-siam}, we can derive a more general form of Proposition~\ref{Lem:FBSDEtoMean}.
	\begin{lemma}\label{Lem:FBSDEtoMeanGen}
		Under the same assumptions and notations as in Proposition~\ref{Lem:FBSDEtoMean} consider additionally another $\mathcal{F}^I$-adapted random measure flow $\bm\mu'$ and the controlled diffusion process $X^{\prime}=\left(X_t^{\prime}\right)_{0 \leq t \leq T}$ defined by
		$$
		X_t^{\prime}=x_0^{\prime}+\int_0^t b\left(s, X_s^{\prime}, \mu_t^{\prime}, \beta_s,I_{s-}\right) d s+\int_0^t\sigma(t, X_s,\mu_s^{\prime},I_{s-})dW_s, \quad t \in[0, T]
		$$
		{for a control $\beta\in \mathbb{A}$}. Then{
			\begin{multline*}
				J(0,x_0,\bm\mu,\hat{\alpha},i_0)+\left\langle x_0^{\prime}-x_0, P_0\right\rangle+\lambda \mathbb{E} \int_0^T\left|\beta_t-\hat{\alpha}_t\right|^2 d t 
				\\\leq J(0,x_0^{\prime},\bm\mu,\left[\beta, \bm\mu^{\prime}\right],i_0)
				+\mathbb{E}^{x_0^\prime,i_0}\left[\int_0^T\left\langle b_0\left(t, \mu_t, I_{t-}\right)-b_0\left(t, \mu_t^{\prime}, I_{t-}\right), P_{t-}\right\rangle d t\right.\\
				\left.+\int_0^T \operatorname{tr}\left[Q_t^{\top}\left(\sigma\left(t,X_t^\prime, \mu_t,I_{t-}\right)-\sigma\left(t,X_t^\prime, \mu_t^\prime,I_{t-}\right)\right) \right] d t\right] ,
		\end{multline*}}
		where {
			\begin{equation*}
			J(0,x_0^{\prime},\bm\mu,\left[\beta, \bm\mu^{\prime}\right],i_0)=\mathbb{E}^{x_0^{\prime},i_0}\left[g\left(X_T^{\prime}, \mu_T,I_T\right)+\int_0^T f\left(t, X_t^{\prime}, \mu_t, \beta_t,I_{t-}\right) d t\right] .
			\end{equation*}}
	\end{lemma}
	\begin{proof}
		The proof is similar to Proposition~\ref{Lem:FBSDEtoMean} and employing Hypothesis~\ref{hyp:A3}.
	\end{proof}
	Define the {generalized Hamiltonian}
	\begin{align*}
		\mathcal{G}(t,x,\mu,\al,p,M,i_0)= {p\cdot b(t,x,\mu,\al,i_0) }+f(t,x,\mu,\al,i_0)+\frac{1}{2}\operatorname{tr}{\left(\sigma^{\prime} (t,x,\mu,i_0) M\sigma(t,x,\mu,i_0)\right)}.
		\end{align*}
	{Also recall the notations in \ref{not:M}.} Then the following lemma holds.
\begin{proposition}\label{lem:MeantoFBSDE}
	Suppose for a given $\mathcal{F}^I$-adapted measure flow $\bm{\mu}$ the value function $v^{\bm\mu}(t,x,i_0)$ given by \eqref{eq:player-sde2}-\eqref{eq:valuemu} is first differentiable in $t$ and third differentiable in $x$. If there exists an optimal control $\tilde{\beta} \in \mathbb{A}$ to \eqref{eq:valuemu} with associated controlled diffusion $\tilde{X}$, then
	\begin{multline}\label{eq:cg}
	\mathcal{G}\left(t, \tilde{X}_t,\mu_t, \tilde{\beta}_t, \nabla_x v^{\bm\mu}\left(t, \tilde{X}_t,I_{t-}\right), \nabla_x^2 v^{\bm\mu}\left(t, \tilde{X}_t,I_{t-}\right),I_{t-}\right)\\
	=\min _{a \in A} \mathcal{G}\left(t, \tilde{X}_t,\mu_t, a, \nabla_x v^{\bm\mu}\left(t, \tilde{X}_t,I_{t-}\right), \nabla_x^2 v^{\bm\mu}\left(t, \tilde{X}_t,I_{t-}\right),I_{t-}\right).
	\end{multline}
	Furthermore
	\begin{align*}
	\left(P_t, Q_t\right)=\left(\nabla_x v^{\bm\mu}\left(t, \tilde{X}_t,I_{t}\right), \nabla_x^2 v^{\bm\mu}\left(t, \tilde{X}_t,I_{t}\right) \sigma\left(t,\tilde{X}_t, \mu_t,I_{t}\right)\right), \\
	\Lambda_t\cdot dM_t=\sum_{k_0 \neq j_0}\left(\nabla_x v^{\bm\mu}(t, \tilde{X}_t,k_0)-\nabla_x v^{\bm\mu}(t, \tilde{X}_t,j_0)\right)dM_{k_0 j_0}(t),
\end{align*}
	solves the adjoint BSDE 
	\begin{equation}\label{eq:BSDE}
		P_t=\nabla_x g\left(x_T,\mu_T,I_T\right)+ \displaystyle \int_t^T\nabla_x \mathcal{H}(s,X_s,\mu_s,\tilde{\beta}_s,P_{t-},Q_{t-},I_{s-})ds- \displaystyle \int_t^T Q_{s-} dW_s-\int_t^T\Lambda_{s}\cdot dM_s.
		\end{equation}
	\end{proposition}
\begin{proof}
	Note that the HJB equation of the stochastic control problem \eqref{eq:player-sde2}-\eqref{eq:valuemu} is given by:
		\begin{align}\label{eq:v-mu-hjb}
		\partial_t v^{\bm\mu}+\inf _{a \in A}\left[\mathcal{G}\left(t, x,\mu, a, \nabla_x v^{\bm\mu}, \nabla_x^2 v^{\bm\mu},i_0\right)\right]+\sum_{j_{0} \in \mathcal{S}} q_{i_{0}, j_{0}}\left(v^{\bm\mu}\left(t, x, j_{0}\right)-v^{\bm\mu}\left(t, x, i_{0}\right)\right)=0.
	\end{align}
	{By It\^o's formula it is readily checked that the optimal trajectory $(\tilde{X}, \tilde{\beta})$ satisfies
	\begin{align}\label{eq:v-mu-opt}
		0=&\partial_t v^{\bm\mu}(t, \tilde{X}_t, I_{t-})+\mathcal{G}\left(t, \tilde{X}_t,\mu_t, \tilde{\beta}_t, \nabla_x v^{\bm\mu}\left(t, \tilde{X}_t,I_{t-}\right), \nabla_x^2 v^{\bm\mu}\left(t, \tilde{X}_t,I_t\right),I_{t-}\right)\nonumber \\
		&+\sum_{j_{0} \in \mathcal{S}} q_{I_{t-}, j_{0}}\left(v^{\bm\mu}\left(t, \tilde{X}_t, j_{0}\right)-v^{\bm\mu}\left(t, \tilde{X}_t, I_{t-}\right)\right).
		\end{align}}
		{Thus from \eqref{eq:v-mu-hjb} and \eqref{eq:v-mu-opt} we get \eqref{eq:cg}.}
	{Since $v^{\bm\mu}$ solves \eqref{eq:v-mu-hjb} we get that for any $x \in \mathbb{R}^d$
		\begin{align}\label{eq:v-mu-hjb-2}
			0	\leq& \partial_t v^{\bm\mu}(t, x, I_{t-})+\mathcal{G}\left(t, x,\mu_t, \tilde{\beta}_t, \nabla_x v^{\bm\mu}\left(t, x,I_{t-}\right), \nabla_x^2 v^{\bm\mu}\left(t, x,I_{t-}\right),I_{t-}\right)\nonumber \\
			&+\sum_{j_{0} \in \mathcal{S}} q_{I_{t-}, j_{0}}\left(v^{\bm\mu}\left(t, x, j_{0}\right)-v^{\bm\mu}\left(t, x, I_{t-}\right)\right).
		\end{align}}
	By the differentiability of $v^{\bm\mu}$ in $(t,x)$ and comparing \eqref{eq:v-mu-opt}-\eqref{eq:v-mu-hjb-2}, we have by the optimality condition that
	\begin{align*}
		\nabla_x\left(\partial_t v^{\bm\mu}(t, x, I_{t-})+\mathcal{G}\left(t, x,\mu_t, \tilde{\beta}_t, \nabla_x v^{\bm\mu}\left(t, x,I_{t-}\right), \nabla_x^2 v^{\bm\mu}\left(t, x,i_0\right),I_{t-}\right)\right.\\
		\left.\left.+\sum_{j_{0} \in \mathcal{S}} q_{I_{t-}, j_{0}}\left(v^{\bm\mu}\left(t, x, j_{0}\right)-v^{\bm\mu}\left(t, x, I_{t-}\right)\right)\right) \right|_{x=\tilde{X}_t}=0.
		\end{align*}
Recalling the expression of $\mathcal{G}$ and $\mathcal{H}$ the above equation becomes
	\begin{multline}\label{eq:partxvG}
	\partial_t \nabla_x v^{\bm\mu}\left(t, \tilde{X}_t,I_{t-}\right)+  \nabla_x^2 v^{\bm\mu}\left(t, \tilde{X}_t,I_{t-}\right) b\left(t,\tilde{X}_t,\mu_t, \tilde{\beta}_t,I_{t-}\right)\\+\frac{1}{2} \operatorname{tr}\left(\sigma \sigma^{\prime}\left(t,\tilde{X}_t,\mu_t,I_{t-}\right)\nabla_x^3 v^{\bm\mu}\left(t, \tilde{X}_t,I_{t-}\right)\right) +\nabla_x \mathcal{H}\left(t, \tilde{X}_t, \tilde{\beta}_t, {P_{t-}, Q_{t-}}\right)\\
	+\sum_{j_{0} \in \mathcal{S}} q_{I_{t-}, j_{0}}\left(\nabla_x v^{\bm\mu}\left(t, \tilde{X}_t, j_{0}\right)-\nabla_x v^{\bm\mu}\left(t, \tilde{X}_t, I_{t-}\right)\right)=0 .
\end{multline}
Applying It\^o's formula to $\nabla_x v^{\bm\mu}(t, \tilde{X},{I_{t}})$ we have {
\begin{multline}\label{eq:itopartxv}
	\nabla_x g\left(x_T,\mu_T,I_T\right)-\nabla_x v^{\bm\mu}(t, \tilde{X}_t,I_t)=\int_t^T\left[\partial_t \nabla_x v^{\bm\mu}\left(t, \tilde{X}_t,I_{t-}\right)\right.\\
	+\nabla_x^2 v^{\bm\mu}\left(t, \tilde{X}_t,I_{t-}\right) b\left(t,\tilde{X}_t,\mu_t, \tilde{\beta}_t,I_{t-}\right)+\frac{1}{2} \operatorname{tr}\left(\sigma \sigma^{\prime}\left(t,\tilde{X}_t,\mu_t,I_{t-}\right)\nabla_x^3 v^{\bm\mu}\left(t, \tilde{X}_t,I_{t-}\right)\right) \\
	\left.+\sum_{j_{0} \in \mathcal{S}} q_{I_{t-}, j_{0}}\left(\nabla_x v^{\bm\mu}\left(t, \tilde{X}_t, j_{0}\right)-\nabla_x v^{\bm\mu}\left(t, \tilde{X}_t, I_{t-}\right)\right)\right]dt
	+\int_t^T Q_{t-}dW_t+	\int_t^T\Lambda_t\cdot dM_t.
	\end{multline}
	Combining \eqref{eq:partxvG} and \eqref{eq:itopartxv} we have that \eqref{eq:BSDE} holds.	}

	\end{proof}

	\begin{proposition}\label{prop:fbsde-nasheq}
		Suppose the value function $v^{\mathcal{L}^1(X)}(t,x,i_0)$ is first differentiable in $t$ and third differentiable in $x$. Then the  process $(X_t, P_t, Q_t, \Lambda_t)$ solves the following mean-field FBSDE with regime switching
		\begin{align}\label{eq:fbsde}
				&X_t=x_0+\displaystyle \int_0^t b\left(s,X_s,\mathcal{L}^1(X_{s}),\hat{\al}(s, X_s,\mathcal{L}^1(X_{s}), P_{s-},I_{s-})  ,I_{s-}\right)ds+\int_0^t \sigma \left(s,X_s,\mathcal{L}^1(X_{s}),I_{s-}\right)dW_s, \nonumber\\
				&P_t=\nabla_x g\left(X_T,\mathcal{L}^1(X_{T}),I_T\right)+ \displaystyle \int_t^T\nabla_x \mathcal{H}\left(s,X_s,\mathcal{L}^1(X_{s}),\hat{\al}(s, X_s,\mathcal{L}^1(X_{s}), P_{s-},I_{s-}),P_{s-},Q_{s-},I_{s-}\right)ds\nonumber\\
				&\hspace{1cm}- \displaystyle \int_t^T Q_{s-} dW_s-\int_t^T\Lambda_{s-}\cdot dM_s,\tag{MV-FBSDE}
		\end{align}
		{such that
		\begin{equation*}
		\mathbb{E}\left[\sup _{0 \leq t \leq T}\left(\left|X_t\right|^2+\left|P_t\right|^2\right)+\int_0^T\left|Q_t\right|^2 d t\right]+\mathbb{E} \int_0^T\hat{\Lambda}_t^{\circ 2} \cdot d[M]_t<+\infty,
		\end{equation*}}
		if and only if the admissible control {$\hat{\al}_t=\hat{\al}(t, X_t,\mathcal{L}^1(X_{t}), P_{t-},I_{t-})$} is a Nash equilibrium of the mean-field stochastic differential game given by Definition~\ref{def:mfg}. 
		\end{proposition}
{
	\begin{proof}
		(Sufficient condition) We apply Proposition~\ref{Lem:FBSDEtoMean} and choose $\mu_t=\mathcal{L}^1(X_{t})$. Then note that $\hat{\al}_t$ is the optimal control of the stochastic control problem \eqref{eq:player-sde2}-\eqref{eq:valuemu} with the specific measure flow $\mathcal{L}^1(X_{t})$. Thus it solves the matching problem in Definition~\ref{def:mfg}.
		
		(Necessary condition) We apply Proposition~\ref{lem:MeantoFBSDE} and choose $\mu_t=\mathcal{L}^1(X_{t})$. Since $X_t$ is the associated diffusion controlled by $\hat{\al}_t$, there exist $(P_t,Q_t,\Lambda_t)$ together with $X_t$ which solve \eqref{eq:fbsde}. Thus we get our desired result.
	\end{proof}
\begin{remark}
	In Proposition~\ref{prop:fbsde-nasheq}, the proof of the sufficient condition does not require the differentiability of the value function $v$, which is an unnatural assumption. Therefore we do not need this assumption to claim solving the FBSDE~\eqref{eq:fbsde} implies getting a Nash equilibrium. This is enough for our intent in the sequel. However, Proposition~\ref{lem:MeantoFBSDE} provides better insight into the behavior of the processes in \eqref{eq:fbsde}.
\end{remark}
}
The following theorem gives the unique solvability of \eqref{eq:fbsde}. 
			\begin{theorem}
				\label{prop:yequ}
				Under Hypothesis~\ref{hyp:Nash2}, the forward-backward system \eqref{eq:fbsde} has a unique solution {in the space $\cs_T^2(\mathbb{R}^d) \times \cs_T^2(\mathbb{R}^d) \times \cl_T^2(\mathbb{R}^{d \times d}) \times \cm_T^2(\mathbb{R}^d)$}. 
				Moreover, for any solution $(X_{t},P_{t},Q_{t},\Lambda_{t})_{0 \leq t \leq T}$ to \eqref{eq:fbsde},
				there exists a function $u : [0,T] \times \mathbb{R}^d \hookrightarrow \mathbb{R}^d$, satisfying the growth and Lipschitz properties
				\begin{equation}
					\label{eq:uproperty}
					\forall t \in [0,T], \quad \forall x,x' \in \mathbb{R}^d, \quad 
					\left\{
					\begin{array}{l}
						| u(t,x,i_0) | \leq c ( 1+ | x |),
						\\
						| u(t,x,i_0) - u(t,x',i_0) | \leq c | x - x' |,
					\end{array}
					\right.
				\end{equation}
				for some constant $c\geq 0$, and such that, $\PP$-a.s., for all 
				$t \in [0,T]$,  $P_{t}=u(t,X_{t},I_{t})$.
			\end{theorem}
	\begin{proof}
	Since \(\hat{\alpha} = \hat{\alpha}(t, x, \mu, p, i_0)\) is the optimal point of the Hamiltonian \(\mathcal{H}(t, x, \mu, \alpha, p, q, i_0)\) defined in \eqref{eq:Ham}, it must satisfy the first-order condition:
	\[
	p \cdot b_2(t, i_0) = -\nabla_\alpha f(t, x, \mu, \alpha, i_0).
	\]
	Combining this relation with Hypothesis~\ref{hyp:Nash2}, we can apply \cite[Theorem 3.1, Theorem 3.2]{rolon2024markovian} to conclude that for any initial state \(x_0 \in \mathbb{R}^d\) and \(t_0 \in [0, T]\), there exists a unique solution to \eqref{eq:fbsde} on the interval \([t_0, T]\). Denote this solution by \((X_t^{t_0, x_0}, P_t^{t_0, x_0}, Q_t^{t_0, x_0}, \Lambda_t^{t_0, x_0})_{t_0 \leq t \leq T}\). To proceed, it suffices to show that the following Lipschitz condition holds:
	\begin{align}\label{eq:Yt0Lip}
		\forall x_0, x_0^{\prime} \in \mathbb{R}^d, \quad \left|P_{t_0}^{t_0, x_0} - P_{t_0}^{t_0, x_0^{\prime}}\right|^2 \leq c \left|x_0 - x_0^{\prime}\right|^2,
	\end{align}
	for some constant \(c\) independent of \(t_0\). Indeed, using It\^o’s formula we have
	\begin{multline*}
		\E\sup_{t_0 \leq t \leq T}\left|X^{t_0,x_0}_{t} - X^{t_0,x_0^\prime}_{t}\right|^2\leq|x_0-x_0^\prime|^2+2\E\int_{t_0}^T\left|X^{t_0,x_0}_{s} - X^{t_0,x_0^\prime}_{s}\right|(b^{t_0,x_0}_{s} -b^{t_0,x_0^\prime}_{s})ds\\
		+ \E\int_{t_0}^T\left|\sigma^{t_0,x_0}_{s} - \sigma^{t_0,x_0^\prime}_{s}\right|^2ds+2\E\sup_{t_0 \leq t \leq T}\left|\int_{t_0}^t\langle X^{t_0,x_0}_{s} - X^{t_0,x_0^\prime}_{s},(\sigma^{t_0,x_0}_{s} - \sigma^{t_0,x_0^\prime}_{s}) dW_s\rangle\right|,
	\end{multline*}
	where we denote
	\begin{equation*}
		b^{t_0,x_0}_{s}=b(s,X_s^{t_0,x_0},\mathcal{L}^1(X_{s}^{t_0,x_0}),\hat{\al}(s, X_s^{t_0,x_0},\mathcal{L}^1(X_{s}^{t_0,x_0}), P_{s-}^{t_0,x_0},I_{s-})  ,I_{s-}).
	\end{equation*}
	A similar definition holds for $\sigma_s^{t_0,x_0}$. By Hypothesis~\ref{hyp:A1}, \ref{hyp:A6}, \ref{hyp:gradsigma} and Lemma~\ref{Lem:UniOpt2} we know that $b$ and $\sigma$ are Lipschitz in $x$, $\mu$ and $p$. 
	 Consequently, applying Burkholder–Davis–Gundy inequality we obtain
	\begin{align*}
		&\E\sup_{t_0 \leq t \leq T}\left|X^{t_0,x_0}_{t} - X^{t_0,x_0^\prime}_{t}\right|^2\leq |x_0-x_0^\prime|^2+c\E\int_{t_0}^T\left|X^{t_0,x_0}_{s} - X^{t_0,x_0^\prime}_{s}\right|\left(\left|X^{t_0,x_0}_{s} - X^{t_0,x_0^\prime}_{s}\right|\right.\\
		&\left.+W_2\left(\cl^1(X^{t_0,x_0}_{s}),\cl^1( X^{t_0,x_0^\prime}_{s})\right)+\left|P^{t_0,x_0}_{s} - P^{t_0,x_0^\prime}_{s}\right|\right)ds\\
		&+c\E\int_{t_0}^T\left(\left|X^{t_0,x_0}_{s} - X^{t_0,x_0^\prime}_{s}\right|^2+W_2^2\left(\cl^1(X^{t_0,x_0}_{s}),\cl^1( X^{t_0,x_0^\prime}_{s})\right)\right)ds\\
		&+c\E\left(\int_{t_0}^T\left|X^{t_0,x_0}_{s} - X^{t_0,x_0^\prime}_{s}\right|^2\left(\left|X^{t_0,x_0}_{s} - X^{t_0,x_0^\prime}_{s}\right|^2+W_2^2\left(\cl^1(X^{t_0,x_0}_{s}),\cl^1( X^{t_0,x_0^\prime}_{s})\right)\right)ds\right)^{1/2},
	\end{align*}
	Here and in the following $c>0$ is a generic constant that may vary from line to line. Therefore, by using the fact that $W_2(\cl^1(X^{t_0,x_0}_{s}),\cl^1( X^{t_0,x_0^\prime}_{s}))\leq(\E(|X^{t_0,x_0}_{s}- X^{t_0,x_0^\prime}_{s}|^2\mid \cf_s^I))^{1/2}$ and Cauchy-Schwarz inequality, we deduce that
	\begin{equation*}
		\E\sup_{t_0 \leq t \leq T}\left|X^{t_0,x_0}_{t} - X^{t_0,x_0^\prime}_{t}\right|^2\leq |x_0-x_0^\prime|^2+c\E\int_{t_0}^T\left|X^{t_0,x_0}_{s} - X^{t_0,x_0^\prime}_{s}\right|^2ds+c\E\int_{t_0}^T\left|P^{t_0,x_0}_{s} - P^{t_0,x_0^\prime}_{s}\right|^2ds.
	\end{equation*}
	By Gronwall's inequality we have 
	\begin{equation}\label{eq:EXtoxo-Xtoxo'}
		\E\sup_{t_0 \leq t \leq T}\left|X^{t_0,x_0}_{t} - X^{t_0,x_0^\prime}_{t}\right|^2\leq |x_0-x_0^\prime|^2+c\E\int_{t_0}^T\left|P^{t_0,x_0}_{s} - P^{t_0,x_0^\prime}_{s}\right|^2ds.
	\end{equation}
	Thanks to Hypothesis~\ref{hyp:A12} and Lemma~\ref{Lem:UniOpt2}, we can apply \cite[Lem 2.1]{rolon2024markovian} to obtain
	\begin{align*}
		\left|P_{t_0}^{t_0, x_0} - P_{t_0}^{t_0, x_0^{\prime}}\right|^2 &\leq c\left(|x_0-x_0^\prime|^2+\E\int_{t_0}^T\left|X^{t_0,x_0}_{s} - X^{t_0,x_0^\prime}_{s}\right|^2+W_2^2\left(\cl^1(X^{t_0,x_0}_{s}),\cl^1( X^{t_0,x_0^\prime}_{s})\right)ds\right)\\
		&\leq c\left(|x_0-x_0^\prime|^2+(T-t_0)\E\sup_{t_0 \leq t \leq T}\left|X^{t_0,x_0}_{t} - X^{t_0,x_0^\prime}_{t}\right|^2\right).
	\end{align*}
	Substituting \eqref{eq:EXtoxo-Xtoxo'} into the above inequality we obtain
	\begin{equation*}
		\left|P_{t_0}^{t_0, x_0} - P_{t_0}^{t_0, x_0^{\prime}}\right|^2\leq c\left( |x_0-x_0^\prime|^2+\E\int_{t_0}^T\left|P^{t_0,x_0}_{s} - P^{t_0,x_0^\prime}_{s}\right|^2ds \right).
	\end{equation*}
	Finally, applying Gronwall's inequality we get the desired result \eqref{eq:Yt0Lip}.  
		\end{proof}
	\begin{remark}
		It is readily checked that the Linear-Quadratic problem is encompassed within our framework and satisfies Hypothesis~\ref{hyp:Nash2}. For explicit examples, we refer to \cite[Remark 3.3]{carmona-delarue-siam} and \cite[Sec 4]{rolon2024markovian}.
		\end{remark}	
		
\section{Propagation of Chaos and Approximate Nash Equilibrium}\label{sec:prop}
In this section, we obtain propagation of chaos results connecting the $N$-player and mean field regime switching games. Similar results have been obtained by \cite{carmona-delarue-siam, ref2} in the setup without common noise and in \cite{ ref3} considering Brownian common noise and in \cite{song-lqg} for regime switching common noise in linear quadratic framework.
First, we introduce a preliminary result from \cite[Lemma 6.1]{ref3}.
\begin{lemma}\label{lem:propchaos}
	If \( \mu \in \mathcal{P}_{q}\left(\mathbb{R}^{d}\right) \) for some \( q>4 \), there exists a constant \( C \) depending only upon \( d, q \) and \( M_{q}(\mu) \) such that:
	\[
	\mathbb{E}\left[W_{2}^{2}\left(\bar{\mu}^{N}, \mu\right)\right] \leq C \epsilon^2_N
	\]
	where \( \bar{\mu}^{N} \) denotes the empirical measure of any sample of size \( N \) from \( \mu \) and 
	\begin{equation}\label{eq:epN}
		\epsilon_N=N^{-1 / \max (d, 4)}\left(1+\log (N) \mathbf{1}_{\{d=4\}}\right)^{1/2}.
	\end{equation}
\end{lemma}
Throughout this section, Hypothesis~\ref{hyp:Nash2} remains in force. Let $(X_{t}, Y_{t}, Z_{t}, \Lambda_{t})_{0 \leq t \leq T}$ denote a solution to \eqref{eq:fbsde}, and let $u$ be as defined in Theorem~\ref{prop:yequ}. Recall that $(\mathcal{L}(X_t))_{0 \leq t \leq T}$ denote the flow of the conditional law of $X_t$ given the filtration $\mathcal{F}_t^I$, for $0 \leq t \leq T$. Additionally, let $\hat{J}$ represent the optimal cost of the limiting mean-field problem:
\begin{equation}\label{eq:Jhat}
	\hat{J} = \mathbb{E}^{x_0,i_0}\biggl[ g(X_{T},\mathcal{L}^1(X_{T}),I_T) + \int_{0}^T f\bigl(t,X_{t},\mathcal{L}^1(X_{t}),\hat{\alpha}(t,X_{t},\mathcal{L}^1(X_{t}),P_{t},I_{t-}),I_{t-}\bigr) dt \biggr],
\end{equation}
where $\hat{\alpha}$ is the unique minimizer of the Hamiltonian \eqref{eq:Ham}. Our goal is to demonstrate that using the strategy $\hat{\alpha}$, derived from the limiting problem, in the $N$-player setting yields an approximate Nash equilibrium. To formalize this, we fix a sequence $((W_{t}^i)_{0 \leq t \leq T})_{i \geq 1}$ of independent $d$-dimensional Brownian motions and put the filtration $\mathcal{F}_t=\mathcal{F}_t^{W, I}=\sigma\{W_s^i, I_s: 0 \leq s \leq t,i=1,\cdots,N\}$. For each integer $N$, consider the $N$-player problem where the dynamics of the players are given by the solution $(X_{t}^1, \dots, X_{t}^N)_{0 \leq t \leq T}$ to the system of $N$ stochastic differential equations:
\begin{align}
	\label{eq:Nplayersdeal}
	&dX_{t}^i = b\bigl(t,X_{t}^i,\mu^N_{t},\hat{\alpha}\bigl(t,X_{t}^i,\mathcal{L}^1(X_{t}),u(t,X_{t}^i,I_{t-}),I_{t-}\bigr),I_{t-} \bigr) dt + \sigma(t,X_t^i,\mu^N_{t},I_{t-}) dW_t^i, \nonumber\\
	&\mu^N_{t} = \frac{1}{N} \sum_{j=1}^N \delta_{X_{t}^j},
\end{align}
for $t \in [0,T]$, $X_{0}^i = x_{0}$, and where the players employ the strategy functions obtained from the limiting problem:
\begin{equation}
	\label{eq:alphaNi}
	\hat{\alpha}_{t}^{N,i,I} = \hat{\alpha}(t,X_{t}^i,\mathcal{L}^1(X_{t}),u(t,X_{t}^i,I_{t-}),I_{t-}), \qquad 0 \leq t \leq T, \;\; i \in \{1, \dots, N\}.
\end{equation}
	By It\^o's formula we have
	\begin{equation}\label{eq:Xiito}
		|X_t^i|^2=\int_0^T2X_s^i\cdot b(t,X_s^i,\mu_s^N,\hat{\al}_{s}^{N,i,I},I_{s-})ds+\int_0^T2\langle X_s^i,\sigma_sdW_s\rangle+\int_0^T|\sigma_s|^2ds.
	\end{equation}
	Observe that $u$ is Lipschitz continuous and has linear growth in $x$ by \eqref{eq:uproperty} and $\hat{\alpha}(t,x,\mu,p,i_0)$ is Lipschitz continuous and at most of linear growth in $p$, uniformly in $t \in [0,T]$ by Lemma \ref{Lem:UniOpt2}. Therefore $\hat{\al}$ has linear growth in $x$. Consequently, by Hypothesis~\ref{hyp:A1}, Corollary~\ref{cor:sigma-bd-2} and applying Burkholder-Davis-Gundy inequality on the martingale part of \eqref{eq:Xiito}, we have
	\begin{align}\label{eq:XGron}
		\mathbb{E} \bigl[ \sup_{0 \leq t \leq T} | X_{t}^i |^2 \bigr]\leq \mathbb{E} c\left(1+\int_0^T \left|X_t^i\right|^2dt\right).
	\end{align}
	By Gronwall's inequality
	\begin{equation}\label{eq:Xbound}
		\sup_{N \geq 1} \max_{1 \leq i \leq N}  \mathbb{E} \bigl[ \sup_{0 \leq t \leq T} | X_{t}^i |^2 \bigr] < + \infty,
	\end{equation}
	and thus
\begin{equation}\label{eq:albound}
	\sup_{N \geq 1} \max_{1 \leq i \leq N}\mathbb{E} \int_{0}^T 
	| \hat{\al}_{t}^{N,i,I} |^2 dt  < + \infty.
\end{equation}
This also gives us the well-posedness of \eqref{eq:Nplayersdeal}. 
Next, consider independent and identically distributed (i.i.d.) processes $(\bar{X}^i)_{i \in \{1, \dots, N\}}$ defined by:
	\begin{equation}\label{eq:barX}
		d\bar{X}_{t}^i = b\bigl(t, \bar{X}_{t}^i, \mathcal{L}^1(X_{t}), \hat{\alpha}(t, \bar{X}_{t}^i, \mathcal{L}^1(X_{t}), u(t, \bar{X}^i_t, I_{t-}), I_{t-} \bigr) dt + \sigma(t, \bar{X}_t^i, \mathcal{L}^1(X_{t}), I_{t-}) dW_t^i,
	\end{equation}
	for $t \in [0, T]$, $\bar{X}_0^i =x_0$. These processes $\bar{X}^i$ are i.i.d. copies of $X$, with $\mathbb{P}_{\bar{X}_{t}^i} = \mathcal{L}^1(X_{t})$ for all $t \in [0, T]$ and $i \in \{1, \dots, N\}$.

We now introduce a lemma from \cite[Theorem 4.1]{shao2024conditional} (see also \cite[Theorem 3.1]{shao2022propagation}).
\begin{lemma}\label{lem:propchaoscondition}
	Assume that Hypothesis~\ref{hyp:Nash2} hold. For $1 \leq i \leq N$, let $X^i_t$ and $\bar{X}^i_t$ be defined by \eqref{eq:Nplayersdeal} and \eqref{eq:barX}, respectively. Then for $T>0$ there exists a constant $\tilde{C}>0$ depending on $T, d$, $q$ such that
	\begin{equation*}
		\max _{1 \leq k \leq N} \mathbb{E}\left[\sup _{t \in[0, T]}\left|X_t^i-\bar{X}^i_t\right|^2\right] \leq  \tilde{C} \epsilon^2_N,
	\end{equation*}
	and
	\begin{equation*}
		\sup _{t \in[0, T]} \mathbb{E}\left[W_2^2\left(\frac{1}{N} \sum_{k=1}^N \delta_{X_t^i}, \mathcal{L}(X_t \mid \mathcal{F}^I_t)\right)\right] \leq \tilde{C} \epsilon^2_N,
	\end{equation*}
	where $\epsilon_N$ is defined in \eqref{eq:epN}.
\end{lemma}
	For the purpose of comparison, we now consider a general setting where the players adopt an arbitrary set of strategies $((\beta_{t}^i)_{0 \leq t \leq T})_{1 \leq i \leq N}$. The private state $U^i$ of player $i \in \{1, \dots, N\}$ evolves according to the following stochastic differential equation:
	\begin{equation}
		\label{eq:Uisde}
		dU_{t}^i = b \bigl(t, U_{t}^i, \bar{\nu}_t^{N}, \beta_t^i, I_{t-} \bigr) dt + \sigma(t, U_t^i, \bar{\nu}_t^{N}, I_{t-}) dW_t^i, \qquad 
		\bar{\nu}^N_{t} = \frac{1}{N} \sum_{j=1}^N \delta_{U_{t}^j},
	\end{equation}
	for $t \in [0, T]$ and initial condition $U_{0}^i = x_{0}$. Here, $\bar{\nu}^N_{t}$ represents the empirical distribution of the players' states at time $t$, and $((\beta_{t}^i)_{0 \leq t \leq T})_{1 \leq i \leq N}$ are progressively measurable with respect to $(\mathcal{F}_t)_{0 \leq t \leq T}$ and satisfy $\mathbb{E} \int_{0}^T | \beta_{t}^i |^2 dt < + \infty$. For each player $i \in \{1, \dots, N\}$, the corresponding cost functional is defined as:
	\begin{equation}\label{eq:Jbar}
		\bar{J}^{N,i,I}(\beta^1,\dots,\beta^N) =  {\mathbb E}^{x_0,i_0}
		\biggl[ g\bigl(U_{T}^i,\bar{\nu}_{T}^N,I_T
		\bigr) + \int_{0}^T f(t,U_{t}^i,\bar{\nu}_{t}^N,\beta_{t}^i,I_{t-}) dt  \biggr],
	\end{equation}
	the cost to the $i$-th player.
We are now ready to state the result and follow the standard proof in  \cite{carmona-delarue-siam, bensoussan2016linear, cardaliaguet2010notes}. For simplicity all the expectation in the following are condition on $X_0=x_0, I_0=i_0$.
\begin{theorem}
	\label{th:apNash}
	Under Hypothesis~\ref{hyp:Nash2}, the strategies $(\hat{\al}_{t}^{N,i,I})_{0\le t\le T,\;1 \leq i \leq N}$ defined in \eqref{eq:alphaNi} form an approximate Nash equilibrium of the $N$-player game (\ref{eq:player-sde}--\ref{eq:cost}).  More precisely, there exists a constant $c>0$ such that, for each $N \geq 1$, any player $i\in\{1,\cdots,N\}$ and any progressively measurable strategy 
	$\beta^i=(\beta^i_t)_{0\le t\le T}$ satisfies $\mathbb{E} \int_{0}^T | \beta_{t}^i |^2 dt < + \infty$, one has
		\begin{equation}\label{eq:apNash}
		\bar{J}^{N,i,I}(\hat{\al}_{t}^{N,1,I},\dots,\hat{\al}_{t}^{N,i-1,I},\beta^i,\hat{\al}_{t}^{N,i+1,I},\dots,\hat{\al}_{t}^{N,N,I}) \ge \bar{J}^{N,i,I}(\hat{\al}_{t}^{N,1,I},\cdots,\hat{\al}_{t}^{N,N,I})-c\epsilon_N,
	\end{equation}
	where $\epsilon_N$ is defined in \eqref{eq:epN}.
\end{theorem}
\begin{proof}  
	By symmetry, it suffices to prove \eqref{eq:apNash} for $i=1$. Consider a progressively measurable process $\beta^1 = (\beta^1_t)_{0 \leq t \leq T}$ satisfying $\mathbb{E} \int_{0}^T | \beta_{t}^1 |^2 dt < \infty$. Let $(X^{1,i})_{i \in \{1, \dots, N\}}$ be a set of player states satisfying \eqref{eq:Uisde}, where the first player uses the strategy $\beta^1$, and the remaining players use the strategies \eqref{eq:alphaNi}, i.e., $\beta^i_t = \hat{\alpha}_{t}^{N,i,I}$ for $i \in \{2, \dots, N\}$ and $t \in [0, T]$. Denote the empirical measure of these players by $\mu^{1,N}_{t} = \frac{1}{N} \sum_{j=1}^N \delta_{X_{t}^{1,j}}$. Our goal is to compare $\bar{J}^{N,1,I}(\beta^1, \hat{\alpha}_{t}^{N,2,I}, \dots, \hat{\alpha}_{t}^{N,N,I})$ to $\bar{J}^{N,1,I}(\hat{\alpha}_{t}^{N,1,I}, \dots, \hat{\alpha}_{t}^{N,N,I})$. Similar to \eqref{eq:Xiito} - \eqref{eq:albound}, using the boundedness of $b_0$, $b_1$, and $b_2$, along with Gronwall's inequality and Burkholder-Davis-Gundy inequality, we obtain the following estimates:
	\begin{equation}
		\label{eq:U1est}
		\mathbb{E}\bigg[\sup_{0 \leq t \leq T} | X_t^{1,1} |^2\bigg] \leq c \biggl( 1 + \mathbb{E} \int_{0}^T | \beta^1_{t} |^2 dt \biggr),
	\end{equation}
	and
	\begin{equation}
		\label{eq:Uibd}
		\mathbb{E}\bigg[\sup_{0 \leq t \leq T} | X_t^{1,i} |^2\bigg] \leq c, \quad \text{for } 2 \leq i \leq N.
	\end{equation}
	Here and in the following,	$c > 0$ is a generic constant that may vary from line to line. Summing these inequalities, we derive:
	\begin{equation}
		\label{eq:Uest}
		\frac{1}{N}\sum_{j=1}^N \mathbb{E}\bigg[ \sup_{0 \leq t \leq T} | X_t^{1,j} |^2\bigg] \leq c \biggl( 1 + \frac{1}{N} \mathbb{E} \int_{0}^T | \beta^1_{t} |^2 dt \biggr).
	\end{equation}
	Define
	$\hat{\alpha}_{t}^{i,I} = \hat{\alpha}\bigl(t, \bar{X}_{t}^i, \mathcal{L}^1(X_{t}), u(t, \bar{X}^i_t, I_{t-}), I_{t-} \bigr)$, $t \in [0, T]$, $i \in \{1, \dots, N\}$,
	where $\bar{X}^i$ is defined in \eqref{eq:barX}. 
	By Lemma~\ref{lem:propchaoscondition}, we have:
	\begin{equation}
		\label{eq:JMW1}
		\max_{1 \leq i \leq N} \mathbb{E} \bigl[ \sup_{0 \leq t \leq T} | X_{t}^i - \bar{X}_{t}^i |^2 \bigr] \leq c \epsilon^2_N,
	\end{equation}
	and
	\begin{equation}
		\label{eq:JMW2}
		\sup_{0 \leq t \leq T} \mathbb{E} \bigl[ W_{2}^2 (\mu_{t}^N, \mathcal{L}^1(X_{t})) \bigr] \leq c \epsilon^2_N.
	\end{equation}
	We first compare $\hat{J}$ to $\bar{J}^{N,i,I}(\hat{\alpha}_{t}^{N,1,I}, \dots, \hat{\alpha}_{t}^{N,N,I})$. By \eqref{eq:Jhat}-\eqref{eq:Jbar}, we have:
	\begin{multline*}
		\bigl|\hat{J} - \bar{J}^{N,i,I}(\hat{\alpha}_{t}^{N,1,I}, \dots, \hat{\alpha}_{t}^{N,N,I})\bigr|
		= \bigg|\mathbb{E}\bigg[g(\bar{X}_T^i, \mathcal{L}^1(X_T), I_T) + \int_0^T f\bigl(t, \bar{X}^i_t, \mathcal{L}^1(X_{t}), \hat{\alpha}_{t}^{i,I}, I_{t-}\bigr) dt \\
		- g(X^i_T, \mu^N_T, I_T) - \int_0^T f\bigl(t, X^i_t, \mu^N_t, \hat{\alpha}_{t}^{N,i,I}, I_{t-}\bigr) dt \bigg] \bigg|.
	\end{multline*}
	Using \ref{hyp:A10} and the Cauchy-Schwarz inequality, we obtain for each $i \in \{1, \dots, N\}$:
	\begin{align*}
		\bigl|&\hat{J} - \bar{J}^{N,i,I}(\hat{\alpha}_{t}^{N,1,I}, \dots, \hat{\alpha}_{t}^{N,N,I})\bigr| \\
		&\leq c \, \mathbb{E}\bigg[ \biggl( 1 + | \bar{X}_{T}^i |^2 + | X_{T}^i |^2 + \frac{1}{N} \sum_{j=1}^N | X_{T}^j |^2 \biggr) \biggr]^{1/2} \mathbb{E} \left[ | \bar{X}_T^i - X^i_T |^2 + W_2^2 (\mathcal{L}^1(X_{T}), \mu^N_T) \right]^{1/2} \\
		&\quad + c \int_0^T \biggl\{ \mathbb{E} \biggl[\biggl( 1 + | \bar{X}_{t}^i |^2 + | X_{t}^i |^2 + | \hat{\alpha}_{t}^{i,I} |^2 + | \hat{\alpha}_{t}^{N,i,I} |^2 + \frac{1}{N} \sum_{j=1}^N | X_{t}^j |^2 \biggr) \biggr]^{1/2} \\
		&\quad \cdot \mathbb{E} \bigl[ |\bar{X}_t^i - X^i_t |^2 + | \hat{\alpha}_{t}^{i,I} - \hat{\alpha}_{t}^{N,i,I} |^2 + W_2^2(\mathcal{L}^1(X_{t}), \mu^N_t) \bigr]^{1/2} \biggr\} dt,
	\end{align*}
By \eqref{eq:Xbound}--\eqref{eq:albound} we deduce:
	\begin{multline*}
			\bigl|\hat{J} - \bar{J}^{N,i,I}(\hat{\alpha}_{t}^{N,1,I}, \dots, \hat{\alpha}_{t}^{N,N,I})\bigr|
			\leq c \, \mathbb{E} \bigl[ | \bar{X}_T^i - X^i_T |^2 + W_2^2(\mathcal{L}^1(X_{T}), \mu^N_T) \bigr]^{1/2} \\
			\quad + c \biggl( \int_0^T \mathbb{E} \bigl[ |\bar{X}_t^i - X^i_t |^2 + | \hat{\alpha}_{t}^{i,I} - \hat{\alpha}_{t}^{N,i,I} |^2 + W_2^2(\mathcal{L}^1(X_{t}), \mu^N_t) \bigr] dt \biggr)^{1/2}.
	\end{multline*}
	By Lemma \ref{Lem:UniOpt2} and the Lipschitz property of $u$ in \eqref{eq:uproperty}, we observe that:
	\begin{align*}
		\mathbb{E}| \hat{\alpha}_{t}^{i,I} - \hat{\alpha}_{t}^{N,i,I} |
		&= \mathbb{E}\bigl|\hat{\alpha}\bigl(t, \bar{X}_t^i, \mathcal{L}^1(X_{t}), u(t, \bar{X}_t^i, I_{t-}), I_{t-}\bigr) - \hat{\alpha}\bigl(t, X^i_t, \mathcal{L}^1(X_{t}), u(t, X^i_t, I_{t-}), I_{t-}\bigr)\bigr| \\
		&\leq c \, \mathbb{E}|\bar{X}_t^i - X^i_t|.
	\end{align*}
	Using \eqref{eq:JMW1} and \eqref{eq:JMW2}, we conclude that for any $1\leq i \leq N$, 
	\begin{equation}
		\label{eq:apNash1}
		\bar{J}^{N,i,I}(\hat{\al}_{t}^{N,1,I},\dots,\hat{\al}_{t}^{N,N,I}) = \hat{J} + O( \epsilon_N).
	\end{equation}
	This result indicates that $\bar{J}^{N,i,I}$, when all players use $(\hat{\alpha}_{t}^{N,j,I})_{j = 1, \dots, N}$, approximates the optimal cost $\hat{J}$ of the mean-field problem with an $O(\epsilon_N)$ error. To prove the inequality \eqref{eq:apNash} for $i=1$, it now suffices to compare $\bar{J}^{N,1,I}(\beta^1, \hat{\alpha}_{t}^{N,2,I}, \dots, \hat{\alpha}_{t}^{N,N,I})$ to $\hat{J}$. To this end, we analyze the differences between the processes $X^{1,j}$ (defined in the first paragraph of this proof) and $X^j$ defined in \eqref{eq:Nplayersdeal}. Using the argument similar to \eqref{eq:U1est}--\eqref{eq:Uest}, along with the definitions of $X^{1,j}$ and $X^j$, we obtain the following bounds for any $t \in [0, T]$:
	\begin{align}
		&\mathbb{E}\bigg[\sup_{0 \leq s \leq t} | X_s^{1,1} - X_{s}^1 |^2 \bigg]
		\leq \frac{c}{N} \int_0^t \sum_{j=1}^N \mathbb{E} \bigg[\sup_{0 \leq r \leq s} | X_r^{1,j} - X_{r}^j |^2 \bigg] ds + c \, \mathbb{E} \int_{0}^T | \beta_{s}^1 - \hat{\alpha}_{s}^{N,1,I} |^2 ds,\nonumber\\
		&\mathbb{E}\bigg[\sup_{0 \leq s \leq t} | X_s^{1,i} - X_{s}^i |^2 \bigg]
		\leq \frac{c}{N} \int_0^t \sum_{j=1}^N \mathbb{E} \bigg[\sup_{0 \leq r \leq s} | X_r^{1,j} - X_{r}^j |^2 \bigg] ds, \quad 2 \leq i \leq N. \label{eq:Xi-Xi}
	\end{align}
	Taking the average of these inequalities and applying Gronwall's inequality, we derive:
	\begin{equation}
		\label{eq:U-Xest}
		\frac{1}{N} \sum_{j=1}^N \mathbb{E} \bigg[\sup_{0 \leq t \leq T} | X_t^{1,j} - X_{t}^j |^2 \bigg]
		\leq \frac{c}{N} \mathbb{E} \int_{0}^T | \beta_{t}^1 - \hat{\alpha}_{t}^{N,1,I} |^2 dt.
	\end{equation}
	Substituting \eqref{eq:U-Xest} back into \eqref{eq:Xi-Xi}, we further obtain:
	\begin{equation}
		\label{eq:Ui-Xiest}
		\sup_{0 \leq t \leq T} \mathbb{E} \bigl[ | X_t^{1,i} - X_{t}^i |^2 \bigr]
		\leq \frac{c}{N} \mathbb{E} \int_{0}^T | \beta_{t}^{1} - \hat{\alpha}_{t}^{N,1,I} |^2 dt, \quad 2 \leq i \leq N.
	\end{equation}
	In fact, similar to \cite[Theorem 4.2]{carmona-delarue-siam} and the statement of \cite[Cor 4.3]{carmona-delarue-siam}, it suffices to consider strategies $\beta^1_t$ satisfying $\mathbb{E} \int_{0}^T | \beta_{t}^1 |^2 dt \leq A$ for some constant $A > 0$. By \eqref{eq:Xbound}, \eqref{eq:JMW1}, and \eqref{eq:Ui-Xiest}, we have for $i \geq 2$,
	\begin{equation*}
		\sup_{0 \leq t \leq T} \mathbb{E} \bigl[ | X_t^{1,i} - \bar{X}_{t}^i |^2 \bigr]
		\leq \sup_{0 \leq t \leq T} \mathbb{E} \bigl[ | X_t^{1,i} - X_{t}^i |^2 \bigr] + \sup_{0 \leq t \leq T} \mathbb{E} \bigl[ | X_{t}^i - \bar{X}_{t}^i |^2 \bigr] \leq c_A \epsilon^2_N,
	\end{equation*}
	where the last inequality uses the fact that $\mathcal{O}(N^{-1} + \epsilon^2_N) = \mathcal{O}(\epsilon^2_N)$. Here and in the following, $c_A>0$ is a constant depending on $A$ that may vary from line to line. Consequently
	\begin{equation}
		\label{eq:JMW4}
		\frac{1}{N-1} \sum_{j=2}^N \mathbb{E} \bigl[ |X_t^{1,j} - \bar{X}^j_t |^2 \bigr] \leq c_A \epsilon^2_N.
	\end{equation}
	By the triangle inequality for the Wasserstein distance,
\begin{multline}\label{eq:Wnucondition}	
	\mathbb{E}\bigl[W_{2}^2(\mu^{1,N}_{t},\mathcal{L}^1(X_{t})) \bigr]
	\le c\bigg\{ \mathbb{E}\bigg[W_{2}^2\bigg(\frac1{N}\sum_{j=1}^N\delta_{X_t^{1,j}},\frac1{N-1}\sum_{j=2}^N\delta_{X_t^{1,j}}\bigg) \bigg]
	\\ 
	\hspace{30pt}+\frac1{N-1}\sum_{j=2}^N
	\mathbb{E}\bigl[|X_t^{1,j}-\bar{X}^j_t|^2\bigr] +\mathbb{E}\bigg[W_{2}^2\bigg(\frac1{N-1}\sum_{j=2}^N\delta_{\bar{X}^j_t},\mathcal{L}^1(X_{t})\bigg) \bigg] \biggr\}.
\end{multline}
		The first term is bounded by
		\begin{align*}
	\mathbb{E} \bigg[ W_{2}^2 \bigg( \frac{1}{N} \sum_{j=1}^N \delta_{X_t^{1,j}}, \frac{1}{N-1} \sum_{j=2}^N \delta_{X_t^{1,j}} \bigg) \bigg]
	&\leq \frac{1}{N(N-1)} \sum_{j=2}^N \mathbb{E} \bigl[ |X_t^{1,1} - X_t^{1,j}|^2 \bigr] \\
	&\leq \frac{2}{N} \left( \mathbb{E} \bigl[ |X_t^{1,1}|^2 \bigr] + \frac{\sum_{j=2}^N \mathbb{E} \bigl[ |X_t^{1,j}|^2 \bigr]}{N-1} \right) \leq \frac{c_A}{N},
	\end{align*}
	where the last inequality follows from \eqref{eq:U1est} and \eqref{eq:Uest}. Using \eqref{eq:JMW4} and Lemma \ref{lem:propchaos} to control the second and third terms in \eqref{eq:Wnucondition}, we conclude that
	\begin{equation}
		\label{eq:W2est4nu}
		\mathbb{E} \bigl[ W_{2}^2(\mu^{1,N}_{t}, \mathcal{L}^1(X_{t})) \bigr] \leq c_A \epsilon^2_N.
	\end{equation}
	Next, define $(\bar{U}^1_{t})_{0 \leq t \leq T}$ as the solution to the SDE
	\begin{equation*}
		d \bar{U}^1_{t} = b(t, \bar{U}_{t}^1, \mathcal{L}^1(X_{t}), \beta_{t}^1, I_{t-}) dt + \sigma(t, \bar{U}_t^1, \mathcal{L}^1(X_{t}), I_{t-}) dW_{t}^1, \quad 0 \leq t \leq T, \; \bar{U}^1_0 = x_0,
	\end{equation*}
	and the mean-field cost with $\beta^1$:
	\begin{equation*}
		J(\beta^1) = \mathbb{E} \biggl[ g(\bar{U}_{T}^1, \mathcal{L}^1(X_{T}), I_T) + \int_{0}^T f\bigl(t, \bar{U}_{t}^1, \mathcal{L}^1(X_{t}), \beta_{t}^1, I_{t-}\bigr) dt \biggr].
	\end{equation*}
	From the definition of $X^{1,1}$, we have
	\begin{equation*}
		X_t^{1,1} - \bar{U}^1_{t} = \int_0^t [b_0(s, \mu^{1,N}_{s}, I_{s-}) - b_0(s, \mathcal{L}^1(X_s), I_{s-})] ds + \int_0^t b_1(s, I_{s-}) [X^{1,1}_s - \bar{U}^1_s] ds.
	\end{equation*}
	Using the Lipschitz properties of $b$ and $\si$, \eqref{eq:W2est4nu} along with Gronwall's inequality provides
	\begin{equation}
		\label{eq:U-Ubarest}
		\sup_{0 \leq t \leq T} \mathbb{E} \bigl[ | X_t^{1,1} - \bar{U}_{t}^1 |^2 \bigr] \leq c_A \epsilon^2_N.
	\end{equation}
	With the same argument leading to \eqref{eq:apNash1}, using \eqref{eq:U1est}--\eqref{eq:Uest}, \eqref{eq:W2est4nu}-\eqref{eq:U-Ubarest}, we derive
	\begin{equation*}
		\bar{J}^{N,1,I}(\beta^1, \hat{\alpha}_{t}^{N,2,I}, \dots, \hat{\alpha}_{t}^{N,N,I}) \geq J(\beta^1) - c_A \epsilon_N.
	\end{equation*}
	Since $\hat{J} \leq J(\beta^1)$, by the optimality of $(\hat{\alpha}_{t}^{N,i,I})_{i \geq 1}$ we have
	\begin{equation}
		\label{eq:J1Jhat}
		\bar{J}^{N,1,I}(\beta^1, \hat{\alpha}_{t}^{N,2,I}, \dots, \hat{\alpha}_{t}^{N,N,I}) \geq \hat{J} - c_A \epsilon_N.
	\end{equation}
	Combining \eqref{eq:J1Jhat} with \eqref{eq:apNash1}, we conclude
	\begin{equation}
		\bar{J}^{N,1,I}(\beta^1, \hat{\alpha}_{t}^{N,2,I}, \dots, \hat{\alpha}_{t}^{N,N,I}) \geq \bar{J}^{N,1,I}(\hat{\alpha}_{t}^{N,1,I}, \dots, \hat{\alpha}_{t}^{N,N,I}) - c_A \epsilon_N.
	\end{equation}
	which is the desired result. Therefore, we have shown that $(\hat{\al}_{t}^{N,1,I},\dots,\hat{\al}_{t}^{N,N,I})$ is an $\epsilon$-Nash equilibrium for the $N$-player game.
\end{proof}
\section{Conclusion}
In this paper, we studied multi-player finite horizon stochastic differential games with regime-switching dynamics. We proved the existence and uniqueness of Nash equilibria under mild conditions, including settings where action spaces are not necessarily bounded, and established the well-posedness of the associated coupled HJB-PDE system, as detailed in Theorem~\ref{Thm:3.1}. These results were extended to symmetric multi-player games with mean-field interactions, where player dynamics are simplified through aggregate state distributions.
Then we explored the regime-switching mean field game problem as the population grows to infinity. Using a probabilistic approach, we derived the stochastic maximum principle and established the existence and uniqueness of solutions to the associated McKean–Vlasov FBSDE with regime switching. Finally, we established propagation of chaos and bridged regime-switching finite-player games and their mean field counterparts, showing that strategies derived from the mean field game provide approximate equilibria for finite-player games with explicit convergence rates.

Overall, our work contributes to the understanding of stochastic differential games with regime-switching dynamics, providing rigorous results on the existence, uniqueness, and approximation of Nash equilibria in both finite and large population settings. These findings deepen the theoretical foundations of stochastic differential games and offer insights into the behavior of strategic agents in complex environments.

\bigskip

\bibliographystyle{plain}
\bibliography{ref}

\begin{thebibliography}{10}

\bibitem{lqg-gomez-duncan}
Julian Barreiro-Gomez, Tyrone~E Duncan, and Hamidou Tembine.
\newblock Linear--quadratic mean-field-type games: Jump--diffusion process with
  regime switching.
\newblock {\em IEEE Transactions on Automatic Control}, 64(10), 2019.

\bibitem{mfc-regime}
Erhan Bayraktar, Alekos Cecchin, and Prakash Chakraborty.
\newblock Mean field control and finite agent approximation for
  regime-switching jump diffusions.
\newblock {\em Applied Mathematics \& Optimization}, 88(2), 2023.

\bibitem{bensoussan2020mfg-regime-jump}
Alain Bensoussan, Boualem Djehiche, Hamidou Tembine, and Sheung Chi~Phillip
  Yam.
\newblock Mean-field-type games with jump and regime switching.
\newblock {\em Dynamic Games and Applications}, 10:19--57, 2020.

\bibitem{bensoussan2000stochastic}
Alain Bensoussan and Jens Frehse.
\newblock Stochastic games for n players.
\newblock {\em Journal of optimization theory and applications}, 105, 2000.

\bibitem{bensoussan2013mean}
Alain Bensoussan, Jens Frehse, Phillip Yam, et~al.
\newblock {\em Mean field games and mean field type control theory}, volume
  101.
\newblock Springer, 2013.

\bibitem{bensoussan2016linear}
Alain Bensoussan, KCJ Sung, Sheung Chi~Phillip Yam, and Siu-Pang Yung.
\newblock Linear-quadratic mean field games.
\newblock {\em Journal of Optimization Theory and Applications}, 169:496--529,
  2016.

\bibitem{bierbrauer2004modeling}
Michael Bierbrauer, Stefan Tr{\"u}ck, and Rafa{\l} Weron.
\newblock Modeling electricity prices with regime switching models.
\newblock In {\em Computational Science-ICCS 2004: 4th International
  Conference, Krak{\'o}w, Poland, June 6-9, 2004, Proceedings, Part IV 4},
  pages 859--867. Springer, 2004.

\bibitem{borkar1992stochastic}
VS~Borkar and MK~Ghosh.
\newblock Stochastic differential games: occupation measure based approach.
\newblock {\em Journal of optimization theory and applications}, 73, 1992.

\bibitem{cardaliaguet2010notes}
Pierre Cardaliaguet.
\newblock Notes on mean field games.
\newblock Technical report, Technical report, 2010.

\bibitem{carmona-applications}
Rene Carmona.
\newblock Applications of mean field games in financial engineering and
  economic theory.
\newblock {\em arXiv preprint arXiv:2012.05237}, 2020.

\bibitem{carmona-delarue-siam}
Ren{\'e} Carmona and Fran{\c{c}}ois Delarue.
\newblock Probabilistic analysis of mean-field games.
\newblock {\em SIAM Journal on Control and Optimization}, 51(4), 2013.

\bibitem{ref2}
Ren{\'e} Carmona and Fran{\c{c}}ois Delarue.
\newblock {\em Probabilistic Theory of Mean Field Games with Applications I}.
\newblock Springer, 2018.

\bibitem{ref3}
Ren{\'e} Carmona, Fran{\c{c}}ois Delarue, et~al.
\newblock {\em Probabilistic theory of mean field games with applications II}.
\newblock Springer, 2018.

\bibitem{carmona2016mean}
Ren{\'e} Carmona, Fran{\c{c}}ois Delarue, and Daniel Lacker.
\newblock Mean field games with common noise.
\newblock {\em The Annals of Probability}, 44(6):3740--3803, 2016.

\bibitem{delarue-holder}
Fran{\c{c}}ois Delarue.
\newblock Estimates of the solutions of a system of quasi-linear pdes. a
  probabilistic scheme.
\newblock In {\em S{\'e}minaire de Probabilit{\'e}s XXXVII}. Springer, 2003.

\bibitem{djehiche2016mean}
Boualem Djehiche, Alain Tcheukam, and Hamidou Tembine.
\newblock Mean-field-type games in engineering.
\newblock {\em arXiv preprint arXiv:1605.03281}, 2016.

\bibitem{elias2014}
RS~Elias, MIM Wahab, and L~Fang.
\newblock A comparison of regime-switching temperature modeling approaches for
  applications in weather derivatives.
\newblock {\em European Journal of Operational Research}, 232(3):549--560,
  2014.

\bibitem{reg-switching-ex1}
Robert~J Elliott and Tak~Kuen Siu.
\newblock On risk minimizing portfolios under a markovian regime-switching
  black-scholes economy.
\newblock {\em Annals of Operations Research}, 176, 2010.

\bibitem{elliott2011stochastic}
Robert~J Elliott and Tak~Kuen Siu.
\newblock A stochastic differential game for optimal investment of an insurer
  with regime switching.
\newblock {\em Quantitative Finance}, 11(3), 2011.

\bibitem{friedman-1972}
Avner Friedman.
\newblock Stochastic differential games.
\newblock {\em Journal of differential equations}, 11(1), 1972.

\bibitem{Huang-Malhame-Caines}
Minyi Huang, Roland~P Malham{\'e}, and Peter~E Caines.
\newblock Large population stochastic dynamic games: closed-loop mckean-vlasov
  systems and the nash certainty equivalence principle.
\newblock 2006.

\bibitem{song-lqg}
Jiamin Jian, Peiyao Lai, Qingshuo Song, and Jiaxuan Ye.
\newblock The convergence rate of the equilibrium measure for the lqg mean
  field game with a common noise.
\newblock {\em arXiv preprint arXiv:2106.04762}, 2021.

\bibitem{jian2024convergence}
Jiamin Jian, Qingshuo Song, and Jiaxuan Ye.
\newblock Convergence rate of lqg mean field games with common noise.
\newblock {\em Mathematical Methods of Operations Research}, 99(3):233--270,
  2024.

\bibitem{ref8}
Olga~Aleksandrovna Ladyzhenskaia, Vsevolod~Alekseevich Solonnikov, and Nina~N
  Ural'tseva.
\newblock {\em Linear and quasi-linear equations of parabolic type}, volume~23.
\newblock American Mathematical Soc., 1968.

\bibitem{Lasry-Lions}
Jean-Michel Lasry and Pierre-Louis Lions.
\newblock Mean field games.
\newblock {\em Japanese journal of mathematics}, 2(1), 2007.

\bibitem{li2017modeling}
Yan Li, Lirong Cui, and Cong Lin.
\newblock Modeling and analysis for multi-state systems with discrete-time
  markov regime-switching.
\newblock {\em Reliability Engineering \& System Safety}, 166:41--49, 2017.

\bibitem{lim1998stochastic}
Tae-Jin Lim.
\newblock A stochastic regime switching model for the failure process of a
  repairable system.
\newblock {\em Reliability Engineering \& System Safety}, 59(2):225--238, 1998.

\bibitem{lv-xiong-2024linear}
Siyu Lv, Zhen Wu, and Jie Xiong.
\newblock Linear quadratic nonzero-sum mean-field stochastic differential games
  with regime switching.
\newblock {\em Applied Mathematics \& Optimization}, 90(2):44, 2024.

\bibitem{xiong-impulse-automatica}
Siyu Lv and Jie Xiong.
\newblock Nonzero-sum impulse games with regime switching.
\newblock {\em Automatica}, 145, 2022.

\bibitem{ma1994solving}
Jin Ma, Philip Protter, and Jiongmin Yong.
\newblock Solving forward-backward stochastic differential equations
  explicitly—a four step scheme.
\newblock {\em Probability theory and related fields}, 98(3):339--359, 1994.

\bibitem{ref11}
Son~L Nguyen, George Yin, and Tuan~A Hoang.
\newblock On laws of large numbers for systems with mean-field interactions and
  markovian switching.
\newblock {\em Stochastic Processes and their Applications}, 130(1), 2020.

\bibitem{nguyen2021general-smp-regime}
Son~L Nguyen, George Yin, and Dung~T Nguyen.
\newblock A general stochastic maximum principle for mean-field controls with
  regime switching.
\newblock {\em Applied Mathematics \& Optimization}, pages 1--40, 2021.

\bibitem{pham2009continuous}
Huy{\^e}n Pham.
\newblock {\em Continuous-time stochastic control and optimization with
  financial applications}, volume~61.
\newblock Springer Science \& Business Media, 2009.

\bibitem{Protter}
Philip Protter.
\newblock {\em Stochastic Integration and Differential Equation}.
\newblock Springer-Verlag, Berlin, Heidelberg, second edition, 1992.

\bibitem{rockafellar2009variational}
R~Tyrrell Rockafellar and Roger J-B Wets.
\newblock {\em Variational analysis}, volume 317.
\newblock Springer Science \& Business Media, 2009.

\bibitem{rolon2024markovian}
Esteban~J Rol{\'o}n~Guti{\'e}rrez, Son~Luu Nguyen, and George Yin.
\newblock Markovian-switching systems: Backward and forward-backward stochastic
  differential equations, mean-field interactions, and nonzero-sum differential
  games.
\newblock {\em Applied Mathematics \& Optimization}, 89(2):33, 2024.

\bibitem{savku-weber-22}
E~Savku and G-W Weber.
\newblock Stochastic differential games for optimal investment problems in a
  markov regime-switching jump-diffusion market.
\newblock {\em Annals of Operations Research}, 312(2), 2022.

\bibitem{shao2024conditional}
Jinghai Shao, Taoran Tian, and Shen Wang.
\newblock Conditional mckean-vlasov sdes with jumps and markovian
  regime-switching: Wellposedness, propagation of chaos, averaging principle.
\newblock {\em Journal of Mathematical Analysis and Applications},
  534(2):128080, 2024.

\bibitem{shao2022propagation}
Jinghai Shao and Dong Wei.
\newblock Propagation of chaos and conditional mckean-vlasov sdes with
  regime-switching.
\newblock {\em Frontiers of Mathematics in China}, 17(4):731--746, 2022.

\bibitem{yong2012stochastic}
Jiongmin Yong and Xun~Yu Zhou.
\newblock {\em Stochastic controls: Hamiltonian systems and HJB equations},
  volume~43.
\newblock Springer Science \& Business Media, 2012.

\bibitem{reg-switching-ex3}
Qing Zhang.
\newblock Stock trading: An optimal selling rule.
\newblock {\em SIAM Journal on Control and Optimization}, 40(1), 2001.

\bibitem{reg-switching-ex4}
Xin Zhang, Zhongyang Sun, and Jie Xiong.
\newblock A general stochastic maximum principle for a markov regime switching
  jump-diffusion model of mean-field type.
\newblock {\em SIAM Journal on Control and Optimization}, 56(4), 2018.

\end{thebibliography}

\end{sloppypar}
\end{document}